\documentclass{article}
\usepackage[english]{babel}

\usepackage[letterpaper,top=2cm,bottom=2cm,left=3cm,right=3cm,marginparwidth=1.75cm]{geometry}

\usepackage[colorlinks=true, allcolors=blue]{hyperref}

\usepackage{chngcntr}
\usepackage{makecell}
\usepackage{graphicx}
\usepackage{longtable}
\usepackage{enumitem}
\usepackage{lscape}
\graphicspath{ {C:/Users/steve/Documents/4H project} }
\usepackage{amsmath}
\usepackage{amssymb}
\usepackage{bbm}
\usepackage{float}
\usepackage{amsfonts}
\usepackage{amsthm}
\usepackage{listings}
\usepackage{xcolor}
\usepackage{mwe}
\definecolor{codegreen}{rgb}{0,0.6,0}
\definecolor{codegray}{rgb}{0.5,0.5,0.5}
\definecolor{codepurple}{rgb}{0.58,0,0.82}
\definecolor{backcolour}{rgb}{0.95,0.95,0.92}

\lstdefinestyle{mystyle}{
    backgroundcolor=\color{backcolour},   
    commentstyle=\color{codegreen},
    keywordstyle=\color{magenta},
    numberstyle=\tiny\color{codegray},
    stringstyle=\color{codepurple},
    basicstyle=\ttfamily\footnotesize,
    breakatwhitespace=false,         
    breaklines=true,                 
    captionpos=b,                    
    keepspaces=true,                 
    numbers=left,                    
    numbersep=5pt,                  
    showspaces=false,                
    showstringspaces=false,
    showtabs=false,                  
    tabsize=2
}

\usepackage{mathtools}
\usepackage{lscape}
\usepackage{amssymb}
\usepackage{caption}
\usepackage{array}
\usepackage[format=plain,
            labelfont=it,
            textfont=it]{caption}
\usepackage{tikz-cd}
\usepackage{mathdots}
\usetikzlibrary{decorations.markings}
\tikzset{negated/.style={
        decoration={markings,
            mark= at position 0.5 with {
                \node[transform shape] (tempnode) {$\backslash$};
            }
        },
        postaction={decorate}
    }
}
\usepackage{bm}
\usepackage[nottoc]{tocbibind}

\newcolumntype{M}[1]{>{\centering\arraybackslash}m{#1}}
\newcolumntype{N}{@{}m{0pt}@{}}

\newtheorem{theorem}{Theorem}[subsection]
\newtheorem*{theorem*}{Theorem A}
\newtheorem{cor}[theorem]{Corollary}
\newtheorem{lemma}[theorem]{Lemma}
\newtheorem{prop}[theorem]{Proposition}
\theoremstyle{definition}
\newtheorem{remark}[theorem]{Remark}
\newtheorem{definition}[theorem]{Definition}

\newtheorem{conjecture}[theorem]{Conjecture}

\newcommand{\p}{\phantom}
\newcommand{\XXX}{\hspace{-7pt}\substack{\p{X}XX\\X}}
\newcommand{\OXX}{\hspace{-7pt}\substack{\p{X}XX\\0}}
\newcommand{\XOX}{\hspace{-7pt}\substack{\p{X}0X\\X}}
\newcommand{\XXO}{\hspace{-5pt}\substack{\p{I}X0\\X}}
\newcommand{\OOX}{\hspace{-7pt}\substack{\p{X}0X\\0}}
\newcommand{\XOO}{\hspace{-6pt}\substack{\hspace{-1pt}\p{X}0\hspace{-1pt}\p{i}0\\X}}
\newcommand{\OOO}{\hspace{-6pt}\substack{\p{0}\hspace{1pt}0\hspace{1pt}0\\0}}

\newcommand{\diam}{\textup{diam}}

\newcommand{\nth}{\textup{th}}
\newcommand{\N}{\mathbb{N}}

\newcommand{\R}{\mathbb{R}}

\newcommand{\K}{\mathbb{K}}
\newcommand{\F}{\mathbb{F}}
\newcommand{\I}{\mathbb{I}}
\newcommand{\Z}{\mathbb{Z}}

\numberwithin{equation}{subsection}
\title{Combinatorics on Number Walls and the $P(t)$-adic Littlewood Conjecture}
\author{Steven Robertson}
\date{}
\begin{document}
\maketitle
\begin{abstract}
    \noindent In 2004, de Mathan and Teulié stated the $p$-adic Littlewood Conjecture ($p$-LC) in analogy with the classical Littlewood Conjecture. Let $\F_q$ be a finite field $P(t)$ be an irreducible polynomial with coefficients in $\F_q$. This paper deals with the analogue of $p$-LC over the ring of formal Laurent series over $\F_q$, known as the $P(t)$-adic Littlewood Conjecture ($P(t)$-LC). \\

    \noindent Firstly, it is shown that any counterexample to $P(t)$-LC for the case $P(t)=t$ induces a counterexample to $P(t)$-LC when $P(t)$ is \textit{any} irreducible polynomial. Since Adiceam, Nesharim and Lunnon (2021) disproved $P(t)$-LC when $P(t)=t$ and when $\F_q$ is a finite field with characteristic 3, one obtains a disproof of $P(t)$-LC over any such field in full generality (i.e for any choice of irreducible polynomial $P(t)$).\\

    \noindent The remainder of the paper is dedicated to proving two metric results on $P(t)$-LC with an additional monotonic growth function $f$ over an arbitrary finite field. The first \textendash ~a Khintchine-type theorem for $P(t)$-adic multiplicative approximation \textendash ~enables one to determine the measure of the set of counterexamples to $P(t)$-LC for any choice of $f$. The second complements this by showing that the Hausdorff dimension of the same set is maximal when $P(t)=t$ in the critical case where $f=\log^2$. These results are in agreement with the corresponding theory of multiplicative Diophantine approximation over the reals.\\
    
    \noindent Beyond the originality of the results, the main novelty of the work comes from the methodology used. Classically, Diophantine approximation employs methods from either Number Theory or Ergodic Theory. This paper provides a third option: combinatorics. Specifically, an extensive combinatorial theory is developed relating $P(t)$-LC to the properties of the so-called \textit{number wall} of a sequence. This is an infinite array containing the determinant of every finite Toeplitz matrix generated by that sequence. In full generality, the paper creates a dictionary allowing one to transfer statements in Diophantine approximation in positive characteristic to combinatorics through the concept of a number wall, and conversely. \\

\end{abstract}
\tableofcontents

\section{Introduction}

 \noindent Let $\alpha$ be a real number. Let $\left|\alpha\right|$ denote the usual absolute value and define $\left|\left\langle\alpha\right\rangle\right|$ as the distance from $\alpha$ to its nearest integer. A classical conjecture due to Littlewood states that for any real numbers $\alpha$ and $\beta$, \[\inf_{n\in\N\backslash\{0\}}|n|\,|\!\left<n\alpha\right>\!|\,|\!\left< n\beta\right>\!| =0.\]
 \noindent The Littlewood Conjecture (from now on abbreviated to LC) remains open. The current state of the art is a result by Einsiedler, Katok and Lindenstrauss \cite{EKL}, who proved the set of counterexamples to LC has Hausdorff dimension zero. In 2004, de Mathan and Teulié \cite{padic} suggested a variant commonly known as the $p$-adic Littlewood Conjecture (abbreviated to $p$-LC). Let $p$ be a prime number. Given a natural number $n$ expanded as $n=p^k\cdot n_1$, where $k$ and $n_1$ are natural numbers such that $n_1$ is coprime to $p$, define the \textbf{$p$-adic norm} of $n$ as $|n|_p:=p^{-k}$.  \begin{conjecture}[\textit{$\mathbf{p}$-\textbf{LC}, de Mathan and Teulié, 2004}]
     For any prime $p$ and any real number $\alpha$,\begin{equation}
    \inf_{n\in\N\backslash\{0\}}|n|\cdot|\langle n\alpha\rangle|\cdot |n|_p=0. \label{eq:padic}
\end{equation} 
 \end{conjecture} \noindent In 2007, Einsiedler and Kleinbock  \cite{2007} proved that the set of counterexamples to $p$-LC has Hausdorff dimension zero. Metric results have been obtained when a growth function is added to the left hand side of equation (\ref{eq:padic}). Explicitly, for a non-decreasing function $f$ and prime number $p$, define the sets
  \begin{align}
     W(p,f):&=\left\{\alpha\in[0,1): \inf_{n\in\N\backslash\{0\}}f(n)\cdot n\cdot\left|\left\langle n\alpha\right\rangle\right|\cdot|n|_{p}=0\right\},\label{W(p,f)}
 \end{align}  and \begin{equation}
 M(p,f)=[0,1)\backslash W(p,f).\label{M(p,f)}\end{equation} That is, $W(p,f)$ is the set of real numbers in the unit interval satisfying $p$-LC with growth function $f$, and $M(p,f)$ is the set of those numbers failing it. In 2009, Bugeaud and Moshchevitin \cite{log2CE} proved that $M(p,\log^2)$ has full Hausdorff dimension. Two years later, Bugeaud, Haynes and Velani \cite{Metric} proved the Lebesgue measure of the same set is zero. Then,  Badziahin and Velani \cite{mixed} established that $$\dim M(p,\log(n)\cdot\log(\log(n)))=1.$$ Above and throughout, $\dim$ refers to the Hausdorff dimension.\\

 \noindent The $p$-adic Littlewood Conjecture admits a natural analogue over function fields. In order to state it, let $q\in\N$ be a positive power of a prime and let $\mathbb{F}_q$ be the finite field with cardinality $q$. \\
 
 \noindent Furthermore, let $\F_q[t]$ be the ring of polynomials with coefficients in $\F_q$ and define $(\F_q[t])_n$ as the subset of $\F_q[t]$ comprised of only the polynomials of degree less than or equal to $n\in\N$. Similarly, ${\F_q}(t)$ is the field of rational functions over ${\F_q}$. The \textbf{absolute value} of $\Theta(t)\in{\F_q}(t)$ is then \begin{align}|\Theta(t)|=q^{\deg(\Theta(t))}.\label{absv}\end{align} 
    
\noindent Although the real and function field absolute values share a notation, it will be clear from context which is being used in each instance. Using this metric, the completion of the field of rational functions is given by ${\F_q}\!\left(\!\left(t^{-1}\right)\!\right)$, the field of formal Laurent series in the variable $t^{-1}$ with coefficients in ${\F_q}$. An element $\Theta(t)$ of this field is written explicitly as \begin{equation}\label{laurent}\Theta(t)=\sum^\infty_{i=-d(\Theta)}a_it^{-i}\end{equation} for some $d(\Theta)\in\mathbb{Z}$ and $a_i\in{\F_q}$ with $a_{d(\Theta)}\neq0$. With this expression, the \textbf{degree} of $\Theta(t)\in{\F_q}\left(\!\left(t^{-1}\right)\!\right)$ is defined as $\deg(\Theta(t)):= d(\Theta)$. The \textbf{fractional part} of $\Theta(t)$ is given by $$\langle \Theta(t) \rangle := \sum^\infty_{i=1}a_it^{-i}.$$
 
 \noindent In analogy with the real numbers, the \textbf{unit interval} $\I$ is defined as \begin{equation}\label{Unit_interval}\I=\left\{\Theta(t)\in{\F_q}\left(\!\left(t^{-1}\right)\!\right): |\Theta(t)|<1\right\}.\end{equation} That is, $\I$ is the set containing the fractional parts of every $\Theta(t)\in{\F_q}\left(\!\left(t^{-1}\right)\!\right)$. \\

 \noindent The analogue of the prime numbers is the set of irreducible polynomials. Just as one defined the $p$-adic norm for a prime $p$, given an irreducible polynomial $P(t)\in{\F_q}[t]$ of degree $d(P)\in\N$, define the $\mathbf{P(t)}$\textbf{-adic norm} of a polynomial $N(t)\in{\F_q}[t]$ as 
 \[|N(t)|_{P(t)}=|P(t)|^{-i}=q^{-d(P)\cdot i},~~~\text{          where          }~~~ i=\max\left\{j\ge0: P(t)^j|N(t)\right\}.\]
\noindent The function field analogue of $p$-LC, also due to de Mathan and Teulié \cite{padic}, reads as follows: 

\begin{conjecture}[\textit{$\mathbf{P(t)}$-\textbf{adic Littlewood Conjecture}, de Mathan and Teulié, 2004}]\label{tlc}
        For any irreducible polynomial $P(t)\in{\F_q}[t]$ and any Laurent series $\Theta(t)\in{\F_q}\left(\!\left(t^{-1}\right)\!\right)$, \[\inf_{N(t)\in{\F_q}[t]\backslash\{0\}}|N(t)|\cdot \left|\left\langle N(t)\cdot \Theta(t)\right\rangle\right|\cdot |N(t)|_{P(t)}=0.\]
    \end{conjecture}
 \noindent The $P(t)$-adic Littlewood Conjecture is abbreviated to $P(t)$-LC. When $P(t)=t$, it becomes the $\mathbf{t}$\textbf{-adic Littlewood Conjecture} (abbreviated to $t$-LC). \\
 
 \noindent In the same paper as it was conjectured, de Mathan and Teulié establish that $t$-LC is false when ${\F_q}$ is replaced with an infinite field. Indeed, the only difference in how the function field set up is defined in this case is found in the absolute value, which is redefined to be $|\Theta(t)|=2^{\deg(\Theta(t))}$. This work was extended by Bugeaud and de Mathan \cite{CE}, who provide explicit counterexamples to $t$-LC in this case. \\
 
 \noindent However, $P(t)$-LC as stated in Conjecture \ref{tlc} (that is, over finite fields) is more challenging. In 2017, Einsiedler, Lindenstrauss
and Mohammadi \cite{PC} proved results on the positive characteristic analogue of the measure classification results of Einsiedler, Katok and Lindenstrauss \cite{EKL}, but their work does not discuss the possibility that $P(t)$-LC fails on a set of Hausdorff dimension zero. \\

\noindent In a series of recent papers \cite{Faustin, GaRo, P(t)_LC}, explicit counterexamples to $t$-LC have been discovered over $\F_q$ for $q$ a power of 5 or $q$ a power of $p\equiv 3\mod 4$. Furthermore, Garrett and the second named author have constructed a family of sequences that are conjectured to provide counterexamples to $t$-LC over any field of odd characteristic \cite[Conjecture 1.8]{GaRo}. From these results, it is tempting to believe that $P(t)$-LC is false over all finite fields of odd characteristic and for all irreducible polynomials $P(t)$. 

\subsection{Statement of Results}
The first (and most elementary) result in this paper gives further evidence to this claim by showing that any counterexample to $t$-LC induces a counterexample to $P(t)$-LC for any irreducible polynomial $P(t)$. This result thus reduces the problem of totally disproving $P(t)$-LC to only disproving $t$-LC. It is not expected that such a transference result should exist for $p$-LC (i.e in the real setting). 
 \begin{theorem}[Transference between $P(t)$-LC and $t$-LC]\label{mainresult}
    Let $l$ be a natural number, let $\K$ be a field and let $P(t)\in{\K}[t]$ be an irreducible polynomial of degree $d(P)$. Additionally, define a Laurent series\begin{equation}\label{CEtheta} \Theta(t):=\sum_{i=1}^\infty b_it^{-i}\in{\K}\left(\!\left(t^{-1}\right)\!\right)\end{equation} and let $q=2$ if $\K$ is infinite and let $q$ be the cardinality of $\K$ otherwise. Assume that\begin{equation}\inf_{N(t)\in{\K}[t]\backslash\{0\}}|N(t)|\cdot \left|\left\langle N(t)\cdot \Theta(t)\right\rangle\right|\cdot |N(t)|_{t}=q^{-l}.\label{tinf}\end{equation} Then the Laurent series $\Theta(P(t))\in{\K}(\!(P(t)^{-1})\!)\subset{\K}\left(\!\left(t^{-1}\right)\!\right)$ satisfies \begin{equation}\label{ptinf}\inf_{N(t)\in{\K}[t]\backslash\{0\}}|N(t)|\cdot \left|\left\langle N(t)\cdot \Theta(P(t))\right\rangle\right|\cdot |N(t)|_{P(t)}=q^{-l\cdot d(P)}.\end{equation} In particular, if $\Theta(t)$ is a counterexample to $t$-LC in the sense that the infimum in equation (\ref{tinf}) is positive, then $\Theta(P(t))$ is a counterexample to $P(t)$-LC in the sense that the infimum in equation (\ref{ptinf}) is positive.
   \end{theorem}
   \noindent Combining the above theorem with the aforementioned counterexamples to $t$-LC over finite fields \cite{Faustin, GaRo, P(t)_LC} and infinite fields \cite{CE} leads to the following corollary: \begin{cor}
       Let ${\K}$ be an infinite field or a finite field with characteristic 5 or $p\equiv 3\mod 4$. Then, the $P(t)$-adic Littlewood Conjecture is false for any choice of irreducible polynomial $P(t)\in{\K}[t]$.
   \end{cor}
 
 \noindent The remaining theorems in this paper are of a metric nature. The underlying measure used in the related statements is the Haar measure over $\F_q\!\left(\!\left(t^{-1}\right)\!\right)$, denoted by $\mu$. Briefly, for an integer $l$, define as usual a ball of radius $q^{-l}$ around a Laurent series $\Theta(t)\in\F_q\!\left(\!\left(t^{-1}\right)\!\right)$ as \begin{equation}\label{ball}B\left(\Theta(t),q^{-l}\right)=\left\{\Phi\in\F_q\!\left(\!\left(t^{-1}\right)\!\right): |\Theta(t)-\Phi(t)|\le q^{-l}\right\}.\end{equation} This set consists of all the Laurent series which have the same coefficients as $\Theta(t)$ up to and including the power of $t^{-l}$. The Haar measure over $\F_q\!\left(\!\left(t^{-1}\right)\!\right)$ is characterised by its translation invariance: for any $l\in\N$ and any $\Theta(t)\in\F_q\!\left(\!\left(t^{-1}\right)\!\right)$  \begin{equation}\mu\left(\!B\!\left(\Theta(t),q^{-l}\right)\!\right)=q^{-l}.\label{Haar}\end{equation} Let $f:\{0\}\cup\{q^n:n\in\N\}\to\R^+$ be a monotonic increasing function. The following sets are the natural analogues of those defined in (\ref{W(p,f)}) and (\ref{M(p,f)}):\begin{align}\label{W_q(P(t),f)}
     W_q(P(t),f):&=\left\{\Theta(t)\in\mathbb{I}: \inf_{N(t)\in\F_q[t]\backslash\{0\}}f(|N(t)|)\cdot|N(t)|\cdot\left|\left\langle\Theta(t)\cdot N(t)\right\rangle\right|\cdot|N(t)|_{P(t)}=0\right\}
 \end{align} \noindent and \begin{equation}
     M_q(P(t),f):=\I\backslash W_q(P(t),f)\label{M_q(P(t),f)}.
 \end{equation} That is, $W_q(P(t),f)$ is the set of Laurent series satisfying $P(t)$-LC with growth function $f$, and $M_q(P(t),f)$ is the complement of $W_q(P(t),f)$. \\
 
 \noindent In view of Theorem \ref{mainresult}, both of the following metric results specialise to the case $P(t)=t$. The first of these two remaining theorems is a Khintchine-type result on $t$-adic multiplicative approximation. 
   \begin{theorem}
   \label{metric} Let $f:\{0\}\cup\{q^k\}_{k\ge0}\to\mathbb{R}^+$ be a function. Then, \begin{equation}\mu(W_q(t,f))=\begin{cases}0 & \text{ if }\sum_{k\ge0}\frac{k}{f(q^k)}<\infty,\\
   1 & \text{ if }\sum_{k\ge0}\frac{k}{f(q^k)}=\infty\text{ and }$f$\text{ non-decreasing}.
   \end{cases}\label{metsum}
    \end{equation}
\end{theorem}

\noindent This amounts to claiming that the set of Laurent series satisfying $t$-LC with growth function $f$ has zero or full measure depending on if the sum in (\ref{metsum}) converges or diverges. From Theorem \ref{metric}, the measure of $M_q(t,f)$ is also easily derived. This agrees with the analogous theorem over the real numbers, proved in \cite{Metric} by Bugeaud, Haynes and Velani. \\

\noindent The final theorem in this paper computes the Hausdorff dimension of $M_q(t,f)$ in the case $f(\cdot)=\log_q^2(\cdot)$. Recalling Definition \ref{M(p,f)}, this provides the function field analogue of the 2009 result by Bugeaud and Moshchevitin \cite{log2CE}, stating $\dim M(p,\log^2)=1$ in the real case. 
 \begin{theorem}\label{log2}
The set of counterexamples to the $t$-LC with growth function $\log^2$ has full Hausdorff dimension. That is, $$\dim(M_q(t,\log^2))=1.$$
 \end{theorem}
 \noindent Note that Theorem \ref{log2} is nontrivial as Theorem \ref{metric} implies that $M(t,\log^2)$ has measure zero. In the real case, Badziahin and Velani \cite{cantor} show $\log^2(\cdot)$ can be replaced by $\log(\log(\cdot))$. It is unclear if a direct adaptation of the methods from \cite{cantor} would yield the same result in the function field set up.  \\

 \noindent Whilst Theorems \ref{mainresult}, \ref{metric} and \ref{log2} are novel, the latter two are not unexpected. Instead, the true originality of this work lies in the technique that is used to prove them. When approaching problems in Diophantine approximation, one classically uses methods from either Number Theory or Ergodic Theory. This paper provides a third option: combinatorics. The strategy is to rephrase $t$-LC using the concept of a so-called \textit{number wall}. Focusing on $t$-LC is justified by Theorem \ref{mainresult} which shows $t$-LC underpins $P(t)$-LC. Section 3 formally introduces the reader to number walls, but for the sake of this introduction, the following definitions will suffice. \\
 
 \noindent A matrix is \textbf{Toeplitz} if every entry on a given diagonal is equal. A Toeplitz matrix is thus determined by the sequence comprising its first column and row. The number wall of a sequence $\mathbf{S}=(s_i)_{i\in\Z}$ is a two dimensioanl array containing the determinants of all the possible finite Toeplitz matrices generated by consecutive elements of $\mathbf{S}$. Specifically, the entry in row $m$ and column $n$ is the determinant of the $(m+1)\times (m+1)$ Toeplitz matrix with top left entry $s_{n}$. Zero entries in number walls can only appear in square shapes, called \textbf{windows}. \\

\noindent The following theorem, adapted from \cite[Theorem 2.2]{Faustin}, is the most elementary example of a much more comprehensive dictionary that translates the Diophantine problem underpinning $t$-LC to combinatorics on number walls. 

\begin{theorem}\label{growth}
    Let $\Theta(t)=\sum_{i=1}^\infty s_it^{-i}\in{\F_q}\left(\!\left(t^{-1}\right)\!\right)$, $l$ be a natural number and $f:\mathbb{N}\to\mathbb{R}^+$ be a positive monotonic increasing function. Then following are equivalent: 
\begin{enumerate}
    \item The inequality \[\inf_{N(t)\in\F_q[t]\backslash\{0\}}f(|N(t)|)\cdot |N(t)|\cdot |\langle N(t)\cdot \Theta(t) \rangle|\cdot |N(t)|_t\ge q^{-l}.\] is satisfied;
    \item For every $m,n\in\mathbb{N}$, the square portion of size $l+\lfloor\log_q(f(q^{m+n}))\rfloor-1$ with top left corner on row $m$ and column $m+n+1$ in the number wall generated by the sequence $\mathbf{S}=(s_i)_{i\ge1}$ over $\F_q$ is not entirely comprised of zero entries.
\end{enumerate}
\end{theorem}

 \noindent Above, $\log_q(x)$ is the base-$q$ logarithm of a real number $x$. 

\subsection{Structure of the Paper}
 \noindent The next section provides the proof of Theorem \ref{mainresult}, which does not rely on the aforementioned concept of a number wall. Then, Section \ref{Sect: NW_intro} introduces the reader to the fundamentals of number walls which are required to state the lemmas used in the proof of Theorems \ref{metric} and \ref{log2}. These lemmas are provided in Section \ref{Sect:Lemmata}, but their proofs are relegated to Section \ref{Sect: proofs}. Theorems \ref{log2} and \ref{metric} are then proved in Sections \ref{Sect:HD} and \ref{Sect:Khint} respectively, before an extensive theory of combinatorics on number walls is developed in Section \ref{Sect: proofs} that is used to prove the lemmata from Section \ref{Sect:Lemmata}. To assist the reader in grasping the new concepts introduced in this paper, a glossary is located after the bibliography. 
 
\subsubsection*{Acknowledgements} \noindent The Author is grateful to his supervisor Faustin Adiceam for his consistent support, supervision and advice throughout the duration of this project. Furthermore, the Author acknowledges the financial support of the Heilbronn Institute.

\section{From Counterexamples to $t$-LC to Counterexamples to $P(t)$-LC}

\noindent This section is dedicated to proving Theorem \ref{mainresult}; namely that a counterexample to $t$-LC implies a counterexample to $P(t)$-LC for any irreducible $P(t)$. The proof is completed over a generic field $\K$. Throughout the proof, $q$ is equal to the cardinality of $\K$ when $\K$ is finite, and to 2 otherwise. \\

\noindent Decomposing the polynomial $N(t)$ as $P(t)^k\cdot N'(t)$ for $N'(t)\in\K[t]$ coprime to $P(t)$, the equation central to $P(t)$-LC, namely (\ref{tlc}), is rephrased as \begin{equation}
\inf_{\substack{N'(t)\in\K[t]\backslash\{0\}\\k\ge0}} \left|\left\langle\Theta(t)\cdot N'(t) \cdot P(t)^k\right\rangle\right|\cdot|N'(t)|=0. 
\end{equation} \noindent Due to the infimum being over all polynomials $N'(t)$ and natural numbers $k$, it is simple to see that $N'(t)$ being coprime to $P(t)$ is not required. Throughout this section, $P(t)\in\F_q[t]$ is an irreducible polynomial, $d(P):=\deg(P(t))$ and every norm is taken with respect to the variable $t$. The following lemma is at the heart of the proof of Theorem \ref{mainresult}. Note that it differs from Theorem \ref{mainresult}, as the infimum in equation (\ref{inf5}) below is taken over all polynomials in $\K[P(t)]$, not in $\K[t]$.
   
   \begin{lemma}
   \label{lemma4} Let $l$ be a natural number and $\Theta(t):=\sum_{i=0}^\infty s_it^{-i}$ be a Laurent series in $\K\!\left(\!\left(t^{-1}\right)\!\right)$. Then,\begin{equation}     \inf_{\substack{N(t)\in\K[t]\backslash\{0\}\\k\ge0}} |N(t)|\cdot\left|\left\langle\Theta(t)\cdot N(t) \cdot t^k\right\rangle\right|=q^{-l} \label{inf2}
   \end{equation} if and only if \begin{equation}
       \inf_{\substack{N(P(t))\in\K[P(t)]\backslash\{0\}\\k\ge0}} |N(P(t))|\cdot\left|\left\langle\Theta(P(t))\cdot N(P(t)) \cdot P(t)^k\right\rangle\right|=q^{-d(P)\cdot l}.\label{inf5}
   \end{equation} 
\end{lemma}
 \begin{proof}[Proof of Lemma \ref{lemma4}]
 First, note that the function field norm (recall equation (\ref{absv})) can only take discrete powers of $q$. Therefore, the values of the infimum in equations (\ref{inf2}) and (\ref{inf5}) are greater than zero if and only if there is some polynomial $N(t)\in\K[t]$ and some $k\in\N$ that attains these values. Indeed, let $N(t)=\sum^{d(N)}_{j=0}n_jt^j$ be this nonzero polynomial in $\K[t]$, where $d(N):=\deg(N(t))$. That is,\vspace{-0.3cm}\begin{equation}|N(t)|\cdot \left|\left\langle\Theta(t)\cdot N(t) \cdot t^k\right\rangle\right|=q^{-l}.\vspace{-0.3cm}\nonumber\end{equation} \vspace{-0.3cm}\noindent Next, the definition of the function field norm implies \begin{equation}\label{t_to_P(t)_eqn_2}
     |N(t)|=q^{d(N(t))} \Leftrightarrow |N(P(t))|=q^{d(P)\cdot d(N)}.
 \end{equation}The goal is now to show that for any $n\in\N$, \begin{equation}
       \left|\left\langle\Theta(t)\cdot N(t) \cdot t^k\right\rangle\right|=q^{-l}\Leftrightarrow\left|\left\langle\Theta(P(t))\cdot N(P(t)) \cdot P(t)^k\right\rangle\right|=q^{-d(P)\cdot l}.\label{t_to_P(t)_eqn1}
 \end{equation}Indeed, equations (\ref{t_to_P(t)_eqn_2}) and (\ref{t_to_P(t)_eqn1}) combine to give \begin{align*}\left|\left\langle\Theta(P(t))\cdot N(P(t)) \cdot P(t)^k\right\rangle\right|\cdot |N(P(t))| &= \left(|N(t)|\cdot\left|\left\langle\Theta(t)\cdot N(t) \cdot t^k\right\rangle\right|\right)^{d(P)}\\&=q^{-d(P)\cdot l}, \end{align*} as required. \\
 
 \noindent To establish equation (\ref{t_to_P(t)_eqn1}), first note that \begin{align}
     \left\langle\Theta(P(t))\cdot N(P(t)) \cdot P(t)^k\right\rangle &= \left\langle\left(\sum^\infty_{i=1}s_iP(t)^{-i}\right)\cdot \left(\sum^{d(N)}_{j=0}n_jP(t)^j\right) \cdot P(t)^k\right\rangle\nonumber\\
     &=\left\langle\sum^\infty_{i=1}\left(\sum^{d(N)}_{j=0}n_j\cdot s_{i+j+k}\right)P(t)^{-i}\right\rangle,\label{t_to_P(t)_eqn_3}
 \end{align} implying that \begin{equation}
     \label{crucial} \deg\left(\left\langle\Theta(P(t))\cdot N(P(t)) \cdot P(t)^k\right\rangle\right)\in\left\{-i\cdot d(P): i\ge1\right\}.
 \end{equation}
 \noindent Finally, substituting $P(t)=t$ in equation (\ref{t_to_P(t)_eqn_3}) shows that $$\deg(\left\langle\Theta(P(t))\cdot N(P(t)) \cdot P(t)^k\right\rangle)=d(P)\cdot \deg(\left\langle\Theta(t)\cdot N(t) \cdot t^k\right\rangle),$$ which proves equation (\ref{t_to_P(t)_eqn1}) and completes the proof.
 \end{proof}
   \begin{proof}[Completion of the proof of Theorem \ref{mainresult}.]
   \noindent Consider a polynomial $N(t)\in\K[t]\backslash\{0\}$ expanded in base $P(t)$ as\begin{equation*}
       N(t)=\sum_{i=0}^{d_P(N)} N_i(t)P(t)^i
   \end{equation*} with $d_P(N):=\left\lfloor\frac{\deg{N(t)}}{d(P)}\right\rfloor$ and $\deg(N_i(t))\le d(P)-1$ for every $0\le i\le d(N)$. Let $n_{j,i}$ be the coefficient of $t^j$ in $N_i(t)$. Then, \begin{align}
       N(t)=\sum_{i=0}^{d_P(N)}\sum^{d(P)-1}_{j=0}n_{j,i}t^jP(t)^i
       =\sum^{d(P)-1}_{j=0}t^j\sum_{i=0}^{d(N)}n_{j,i}P(t)^i
       :=\sum^{d(P)-1}_{j=0}t^j \tilde{N}_j(P(t)),\label{jsum}
   \end{align} where $\tilde{N}_j(P(t)):=\sum^{d_P(N)}_{i=0}n_{j,i}P(t)^i\in\K[P(t)]$. Therefore, \begin{align}\inf_{\substack{N(t)\in\K[t]\backslash\{0\}\\k\ge0}} &|N(t)|\cdot\left|\left\langle\Theta(P(t))\cdot N(t) \cdot P(t)^k\right\rangle\right| \nonumber\\
   &= \inf ~\left|\left\langle\Theta(P(t))\cdot \left(\sum^{d(P)-1}_{j=0}t^j\tilde{N}_j(P(t)) \right)\cdot P(t)^k\right\rangle\right|\cdot\left|\sum^{d(P)-1}_{j=0}t^j\tilde{N}_j(P(t))\right|.\label{infinf}\end{align}
\noindent Above, and throughout the rest of the proof, the infimum is taken over natural numbers $k\ge0$ and all polynomials $\tilde{N}_0(P(t)),\dots,\tilde{N}_{d(P)-1}(P(t))\in\K[P(t)]$, ignoring the case where they are all zero. Using that for any finite set of Laurent series $L$, \[\sum_{\Theta(t)\in L}\langle\Theta(t)\rangle=\left\langle\sum_{\Theta(t)\in L}\Theta(t)\right\rangle,\]
   one has that\begin{align}
   (\ref{infinf})= \inf~\left|\sum^{d(P)-1}_{j=0}\left\langle\Theta(P(t))\cdot t^j\tilde{N}_j(P(t)) \cdot P(t)^k\right\rangle\right|\cdot\left|\sum^{d(P)-1}_{j=0}t^j\tilde{N}_j(P(t))\right|.\label{inf1}\end{align}
   
   \noindent As seen in Lemma \ref{lemma4}, for any $N(t)\in\K[t]$ one has  $$\deg\left(\left\langle\Theta(P(t))\cdot N(P(t)) \cdot P(t)^k\right\rangle\right)\in\{-d(P)\cdot i: i\in\N\}$$ and hence the coefficients for $t^{-1}$, $\dots$, $t^{-d(P)+1}$ are all zero. This implies that for any $0\le j \le d(P)-1$, \[\left\langle\Theta(P(t))\cdot t^jN(P(t)) \cdot P(t)^k\right\rangle=t^j\left\langle\Theta(P(t))\cdot N(P(t)) \cdot P(t)^k\right\rangle.\] Hence, \begin{equation}
(\ref{inf1})=\inf ~\left|\sum^{d(P)-1}_{j=0}t^j\left\langle\Theta(P(t))\cdot \tilde{N}_j(P(t)) \cdot P(t)^k\right\rangle\right|\cdot\left|\sum^{d(P)-1}_{j=0}t^j\tilde{N}_j(P(t))\right|\cdotp\label{inf3}
   \end{equation}
    If the degree of $\left\langle\Theta(P(t))\cdot N(P(t)) \cdot P(t)^k\right\rangle$ is $-i\cdot d(P)$, then for any $0\le j \le d(P)-1$ the degree of $t^j\left\langle\Theta(P(t))\cdot N(P(t)) \cdot P(t)^k\right\rangle$ is $-i\cdot d(P)+j\equiv j \mod d(P)$. This implies that the degree of each of the terms in $\sum^{d(P)-1}_{j=0}t^j\left\langle\Theta(P(t))\cdot \tilde{N}_j(P(t)) \cdot P(t)^k\right\rangle$ are all different and therefore \[\deg\left(\sum^{d(P)-1}_{j=0}t^j\left\langle\Theta(P(t))\cdot \tilde{N}_j(P(t)) \cdot P(t)^k\right\rangle\right)=\max_{0\le j <d(P)}\deg\left(t^j\left\langle\Theta(P(t))\cdot \tilde{N}_j(P(t)) \cdot P(t)^k\right\rangle\right).\] Similarly,  $$\deg\left(\sum^{d(P)-1}_{j=0}t^j\tilde{N}_j(P(t))\right)=\max_{0\le j <d(P)}\deg\left(t^j\tilde{N}_j(P(t))\right).$$ As a consequence, \begin{align}
   (\ref{inf3})&=\inf ~\max_{i\in\{0,\dots,d(P)-1\}}\left|t^i\tilde{N}_i(P(t))\right| \max_{j\in\{0,\dots,d(P)-1\}}\left|t^j\left\langle\Theta(P(t))\cdot \tilde{N}_j(P(t)) \cdot P(t)^k\right\rangle\right|\nonumber\\
   &=\inf ~\max_{i\in\{0,\dots,d(P)-1\}}q^i\left|\tilde{N}_i(P(t))\right|\max_{j\in\{0,\dots,d(P)-1\}}q^j\left|\left\langle\Theta(P(t))\cdot \tilde{N}_j(P(t)) \cdot P(t)^k\right\rangle\right|.\label{inf4}
   \end{align}
   
   \noindent Let $j'$ and $i'$ be the values of $j$ and $i$, respectively, that attain the maxima in (\ref{inf4}). Then, \begin{align*}
       (\ref{inf4})&\ge \min_{~h\in\{j',i'\}}q^h\left|\left\langle\Theta(P(t))\cdot \tilde{N}_h(P(t)) \cdot P(t)^k\right\rangle\right|q^h\left|\tilde{N}_h(P(t))\right|\\
       &\underset{(\ref{inf5})}{\ge} q^{-d(P)\cdot l+2h}.
   \end{align*}
\noindent This provides a lower bound for the value of (\ref{inf4}), namely the case where $h=0$. The fact that $\K[P(t)]\subset\K[t]$ combined with Lemma \ref{lemma4} show that this lower bound is attained, hence  \begin{align*}
       (\ref{inf4})= q^{-l\cdot d(P)}.
   \end{align*}
\noindent This completes the proof of Theorem \ref{mainresult}. 
   \end{proof}
\section{Combinatorial Properties of Number Walls} \label{Sect: NW_intro}
\noindent This section serves as an introduction to number walls and their connection to $t$-LC. Only the theorems crucial for this paper are mentioned, and the proofs can be found in the references. For a more comprehensive look at number walls, see \cite[Section 3]{Faustin},\cite[pp 85-89]{Conway1995TheBO}, \cite{Lunnon2009} and \cite{numbwall}.
\subsection{Fundamentals of Number Walls}
\subsubsection*{Hankel and Toeplitz Matrices}
\noindent The following definition provides the building blocks of a number wall.
\begin{definition} \label{Toe}
A matrix $(s_{i,j})$ for $0\le i\le n, 0\le j \le m$ is called \textbf{Toeplitz} (\textbf{Hankel}, respectively) if all the entries on a diagonal (on an anti-diagonal, respectively) are equal. Equivalently, $s_{i,j}=s_{i+1,j+1}$ ($s_{i,j}=s_{i+1,j-1}$, respectively) for any $n\in\mathbb{N}$ such that this entry is defined. 
\end{definition}
\noindent Given a doubly infinite sequence $\mathbf{S}:= (s_i)_{i\in\mathbb{Z}}$, natural numbers $m$ and $v$, and an integer $n$, define an $(m+1)\times (v+1)$ Toeplitz matrix $T_{\mathbf{S}}(m, v,n):= (s_{i-j+n})_{0\le i \le m, 0\le j \le v}$ and a Hankel matrix of the same size as $H_{\mathbf{S}}(m, v,n):= (s_{i+j+n})_{0\le i \le m, 0\le j \le v}$ :\\
\begin{align*}&T_{\mathbf{S}}(m,v,n):=\begin{pmatrix}
s_n & s_{n+1} & \dots & s_{n+v}\\
s_{n-1} & s_{n} & \dots & s_{n+v-1}\\
\vdots &&& \vdots\\
s_{n-m} & s_{n-m+1} & \dots & s_{n-m+v} 
\end{pmatrix},~H_{\mathbf{S}}(m,v,n):=\begin{pmatrix}
s_n & s_{n+1} & \dots & s_{n+v}\\
s_{n+1} & s_{n+2} & \dots & s_{n+v+1}\\
\vdots &&& \vdots\\
s_{n+m} & s_{n+m+1} & \dots & s_{n+m+v}
\end{pmatrix}.&\end{align*} If $v=m$, these are shortened to $T_{\mathbf{S}}(m,n)$ and $H_{\mathbf{S}}(m,n)$, respectively. The Laurent series $\Theta(t)=\sum^\infty_{i=1}s_it^{-1}\in\F_{q}\left(\!\left(t^{-1}\right)\!\right)$ is identified with the sequence $\mathbf{S}=(s_i)_{i\ge1}$. Accordingly, define $T_\Theta(m,v,n)=T_{\mathbf{S}}(m,v,n)$ and $H_\Theta(m,v,n):=H_{\mathbf{S}}(m,v,n)$.\\

\noindent By elementary row operations, it holds that \begin{equation}
    \det(H_{\mathbf{S}}(m,n))=(-1)^{\frac{m(m+1)}{2}} \det(T_{\mathbf{S}}(m,n+m)).\label{hanktoe}
\end{equation}
\subsubsection*{The Definition of a Number Wall}
\begin{definition}\label{nw}
Let $\mathbf{S}=(s_i)_{i\in\mathbb{Z}}$ be a doubly infinite sequence over a field $\F_q$. The \textbf{number wall} of the sequence $\mathbf{S}$ is defined as the two dimensional array of numbers $W_q(\mathbf{S})=(W_{m,n}(\mathbf{S}))_{n,m\in\mathbb{Z}}$ with \begin{equation*}
    W_{m,n}(\mathbf{S})=\begin{cases}\det(T_{\mathbf{S}}(m,n)) &\textup{ if } m\ge0,\\
    1 & \textup{ if } m=-1,\\
    0 & \textup{ if } m<-1. \end{cases}
\end{equation*}  
\end{definition} 
\noindent In keeping with standard matrix notation, $m$ increases as the rows go down the page and $n$ increases from left to right. One defines the number wall of a Laurent series $\Theta(t)\in \F_q\!\left(\!\left(t^{-1}\right)\!\right)$, denoted $W_q(\Theta(t))$, in the natural way by associating $\Theta(t)$ to the sequence of its coefficients.\\
\subsubsection*{Windows in a Number Wall}
\noindent A key feature of number walls is that the zero entries can only appear in specific shapes:
\begin{theorem}[Square Window Theorem]\label{window}
Zero entries in a number wall can only occur within windows; that is, within
square regions with horizontal and vertical edges.
\end{theorem}
\begin{proof}
See \cite[page 9]{numbwall}.
\end{proof}

\noindent For the problems under consideration, the zero windows carry the most important information in the number wall. 

\subsubsection*{The Inner Frame of a Window}

\noindent Theorem \ref{window} implies that the border of a window is the boundary of a square with nonzero entries. This motivates the following definition: 
\begin{definition}\label{frame}
The entries of a number wall surrounding a window are referred to as the \textbf{inner frame}.
\end{definition}
\noindent The entries of the inner frame are extremely well structured:
\begin{theorem}\label{ratio ratio}
The inner frame of a window with side length $l\ge1$ is comprised of 4 geometric sequences. These are along the top, left, right and bottom edges and they have ratios $P,Q,R$ and $S$ respectively with origins at the top left and bottom right. Furthermore, these ratios satisfy the relation \[\frac{PS}{QR}=(-1)^{l}.\]
\end{theorem}
\begin{proof}
   See \cite[Page 11]{numbwall}.
\end{proof}
\noindent See Figure \ref{basicwondow} for an example of a window of side length $l$. For $i\in\{0,\dots,l+1\}$, the inner frame is labelled by the entries $A_i,B_i,C_i,D_i$. The ratios of the geometric sequences comprising the inner frame are labelled as $P,Q,R$ and $S$.
\begin{figure}[H]
\centering
\includegraphics[width=0.3\textwidth]{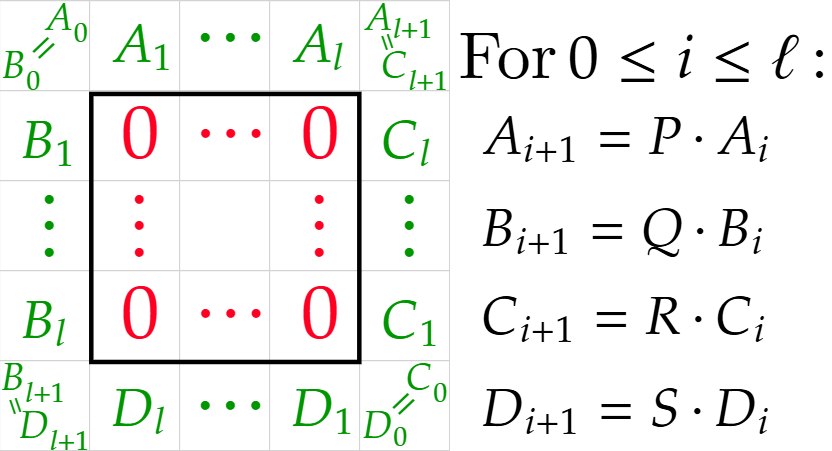}
\caption{Illustration of a window in a number wall. The window is in red and the inner frame is in blue.}
\label{basicwondow}
\end{figure}
\noindent Now that the fundamental properties of number walls have been stated, their connection to $t$-LC is proved.
\subsection{Proof of Theorem \ref{growth}} 
\begin{lemma}\label{famous} Let $H = (s_{i+j-1})$ with $1\le i,j\le n$
be an $m \times m$ Hankel matrix with entries in $\F_q$. Assume that
the first $r$ columns of $H$ are linearly independent but that the first $r + 1$ columns are linearly dependent (here, $1 \le r \le m-1$). Then the principal minor of order r, that is, $\det(s_{i+j-1})_{1\le i,j\le r}$, does not vanish.
    \end{lemma}
\begin{proof}
   See \cite[Lemma 2.3]{Faustin} or \cite[Chapter 10]{Theoryofmat}. \end{proof}
\begin{proof}[Proof of the Implication $1\Rightarrow2:$ in Theorem \ref{growth}]

\noindent The proof is by contraposition. Let $N(t)=M(t)\cdot t^k\in\F_q[t]$ be a nonzero polynomial with $M(t)$ coprime to $t$. Let $d(M)$ be the degree of $M(t):=\sum_{i=0}^{d(M)}m_it^i$. From the definition of the $t$-adic norm, \begin{align}f(|N(t)|)\cdot|N(t)|\cdot|\langle\Theta(t)\cdot N(t) \rangle|\cdot |N(t)|_t&=f(q^{d(M)+k})\cdot q^{d(M)}\cdot|\langle\Theta(t)\cdot M(t) \cdot t^k\rangle|. \label{twindow}
\end{align}  This quantity is strictly less than $q^{-l}$ if and only if \begin{equation}\label{coef}|\langle\Theta(t)\cdot M(t)\cdot t^k\rangle|<q^{-l-d(M)-b_f^{(q)}(d(M)+k)},\end{equation} where $b_f^{(q)}(d(M)+k)=\lfloor\log_q(f(q^{d(M)+k}))\rfloor$. This is equivalent to the coefficients of $t^{-i}$ in $\langle\Theta(t)\cdot M(t)\cdot t^k\rangle$ being zero for $i\in\{1,\dots,l+d(M)+b_f^{(q)}(d(M)+k)\}.$ These coefficients are given by $\sum_{j=0}^{d(M)} a_js_{i+j+k},$ and hence the inequality (\ref{coef}) can be written as the matrix equation \begin{equation}H_\Theta\left( s_M^{(q,f)}(l,k), d(M),k+1\right)\cdot\begin{pmatrix}
    m_0\\\vdots\\m_{d(M)}
\end{pmatrix}=\mathbf{0}~~~\text{where}~~~ s_M^{(q,f)}(l,k)=l+d(M)+b_f^{(q)}(d(M)+k)-1.\label{zeromat}\end{equation} This implies that $H_\Theta\left( s_M^{(q,f)}(l,k), d(M), k+1\right)$ has less than maximal rank, meaning the square matrix $H_\Theta\left(s_M^{(q,f)}(l,k),k+1\right)$ has zero determinant from the Hankel structure. In particular, $H_\Theta(i, k+1)$ is singular for $d(M)\le i \le s_M^{(q,f)}(l,k)$, since each of these matrices satisfies the same linear recurrence relation in the first $d(M)+1$ columns. \\
    
\noindent Let $r_M^{(q,f)}(l,k)=l+b_f^{(q)}(d(M)+k)$. Relation (\ref{hanktoe}) shows that the determinant of $H_\Theta(i, k+1)$ is, up to sign, equal to the determinant of $T_\Theta(i,k+1+i)$. Therefore, there is a diagonal of zeros from entry $(k+d(M)+1,d(M))$ to entry $(k+d(M)+r_M^{(q,f)}(l,k),~d(M)+r_M^{(q,f)}(l,k)-1)$ in the number wall of $\Theta(t)$. The Square Window Theorem (Theorem \ref{window}) implies that there is a window with side length at least $r^{(q,f)}_{M}(l,k)$ in row $d(M)$ and column $d(M)+k+1$. Setting $m:=d(M)$ and $n:=k$ completes the proof in this direction.\end{proof}
\begin{proof}[Proof of the Implication $2\Rightarrow1:$ in Theorem \ref{growth}]
\noindent Proceeding again by contraposition, assume that there is a window of size at least $l+b_f^{(q)}(m+n)$ in row $m$ and column $n+m+1$ for some $m,n\in\N$. Define $$r^{(q,f)}_{m,n}(l):=l+b_f^{(q)}(m+n).$$ The diagonal of this window corresponds to a sequence of nested singular square Toeplitz matrices, $T_\Theta(i, n+i+1)$ for $m\le i \le m+r^{(q,f)}_{m,n}(l)-1$. Hence, the matrices $H_\Theta(i,n+1)$ for $i$ in the same range are all singular. The largest of these Hankel matrices in this diagonal is singular and thus has linearly dependent columns. Since the second largest is also singular, Lemma \ref{famous} implies that the first $m+r^{(q,f)}_{m,n}(l)-1$ columns are linearly dependent. Furthermore, since the third largest matrix is also singular this implies that the first $m+r^{(q,f)}_{m,n}(l)-2$ columns are linearly dependent. A simple induction then shows that $H_\Theta(m+r^{(q,f)}_{m,n}(l)-1, m, n+1)$ has less than maximal rank. \\
    
\noindent Let then $\mathbf{a}=(a_0,\dots,a_m)^\top\in\F_q^{m+1}$ be a vector such that $H_\Theta(m+r^{(q,f)}_{m,n}(l)-1, m,n+1)\cdot \mathbf{a}=\mathbf{0}$. Define the polynomials $M(t)=\sum_{i=0}^ma_it^i$ and $N(t):=M(t)\cdot t^n$. By the same steps used to obtain (\ref{coef}) and (\ref{zeromat}), one has that \[f(|N(t)|)\cdot |N(t)|\cdot|\langle\Theta(t)\cdot N(t) \cdot t^n\rangle|\cdot |N(t)|_t<q^{-l}.\]
    \noindent Hence, a polynomial $N=M\cdot t^n$ with $m:=\deg(M)$ and $(M,t)=1$ being a solution to \[f(|N|)\cdot|N|\cdot\left|\langle N\cdot\Theta\right\rangle|\cdot|N|_t<q^{-l}\] is equivalent to there being a window of size $r^{(q,f)}_{m,n}(l)$ on row $m$ and column $m+n+1$. This completes the proof. \\
\end{proof}
\noindent Taking $f$ to be the constant function, the following corollary (originally from \cite{Faustin}) becomes clear. 
\begin{cor} \label{growcor}
Let $\Theta(t)=\sum^\infty_{i=1}s_it^{-i}\in\F_q\!\left(\!\left(t^{-1}\right)\!\right)$ be a Laurent series and $\mathbf{S}=(s_i)_{i\in\mathbb{N}}$ be the sequence of its coefficients. Then, $\Theta(t)$ is a counterexample to $t$-LC if and only if there exists an $l$ in the natural numbers such that $W_q(\textbf{S})$ has no windows of size larger than $l$.
\end{cor}
\noindent One often wishes to \textit{build} sequences that have given properties in their number walls. Therefore, the portion of a number wall generated by a finite sequence, called a \textbf{finite number wall}, is of interest. The next subsection begins to develop a combinatorial theory on finite number walls.
\subsection{Finite Number Walls}\label{Sect: 3.3}
\noindent Whilst a number wall is defined for doubly infinite sequences, the definition can easily be extended to allow for finite sequences. \begin{definition} \label{def: fin_nw}
     A \textbf{finite number wall} is a two-dimensional infinite array defined exactly as a number wall is (Definition \ref{nw}), but where only the entries that are determined by $\mathbf{S}$ are given values and the remainder are left as variables. More precisely, given a finite sequence $\mathbf{S}$ in $\F_q$, the finite number wall $W_q(\mathbf{S})=(W_{m,n})_{m,n\in\Z}$ is defined as \begin{equation*}
    W_{m,n}(\mathbf{S})=\begin{cases}\det(T_{\mathbf{S}}(m,n)) &\textup{ if } m\ge0 \text{ and every entry of }T_{\mathbf{S}}(m,n)\text{ is defined},\\
     \text{an unknown variable }&\textup{ if } m\ge0 \text{ and not all entries of }T_{\mathbf{S}}(m,n)\text{ are defined},\\
    1 & \textup{ if } m=-1,\\
    0 & \textup{ if } m<-1. \end{cases}
\end{equation*}  
\end{definition} \noindent A finite number wall is illustrated below. For now, ignore the red square in Figure \ref{fin_nw}.

\begin{figure}[H]
    \centering
    \includegraphics[width=0.65\linewidth]{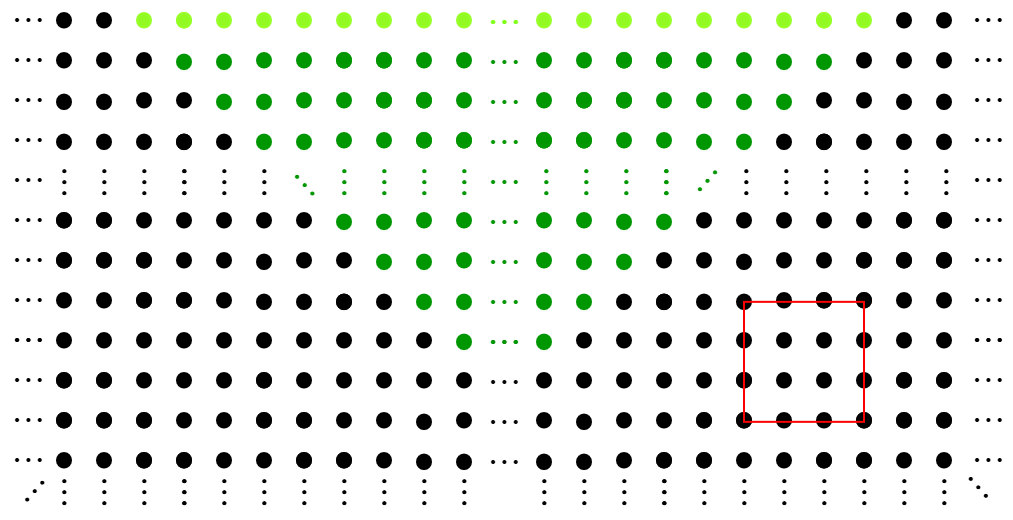}
    \caption{Each dot represents an entry in a number wall. The finite sequence (light green dots) that generates the finite number wall (whole picture) is on row zero. Each dot represents an entry in the finite number wall, with the dark green dots being those that are known explicitly and the black dots being those that are still variables.}
    \label{fin_nw}
\end{figure}  
\noindent To avoid visual clutter, figures will usually ignore the undefined entries of finite number walls.\\

\noindent It is simple to see that a finite sequence of length $r$ defines the portion of a finite number wall in the shape of an isosceles triangle with the greatest row index being equal to $\left\lfloor\frac{r-1}{2}\right\rfloor$. Therefore, it may seem paradoxical that a \textit{finite} number wall is actually an \textit{infinite} two-dimensional array instead of a finite triangular array. It is defined this way, since the methods in this paper often require one to extend a finite sequence that generates a finite number wall in such a way that the resulting number wall has zero entries in a given row and column whose entries are not currently defined. 

\subsubsection*{Dot Diagrams}

\noindent Being two-dimensional arrays over a finite field, number walls are an inherently visual object. Indeed, one often depicts a number wall over $\F_q$ by assigning each $x\in\F_q$ unique colour depending on its value. For an example of this, see the left image in Figure \ref{dotdiagram}.\\

\noindent Whilst generating number walls in this way can be insightful, it provides no help if one is trying to prove a result for number walls \textit{in general}. The introduction of \textbf{dot diagrams} remedies this by giving a way to visualise an arbitrary number wall of finite size, allowing for a clearer exposition during proofs. Each dot in the dot diagram represents an entry of the finite number wall, with the top row representing the zeroth row of the number wall. That is, the sequence that generated the finite number wall. \\

\noindent Figure \ref{fin_nw} is an example of a dot diagram. However, to improve the clarity of all future dot diagrams, all the values of the finite number wall that are undefined are not drawn. For example, in Figure \ref{fin_nw}, only the green parts would be drawn. The dot diagram in Figure \ref{dotdiagram} contains no information, but as the paper progresses dot diagrams are used to convey information about the structure and construction of finite number walls. 
\begin{figure}[H]
\centering
\includegraphics[width=0.6\linewidth]{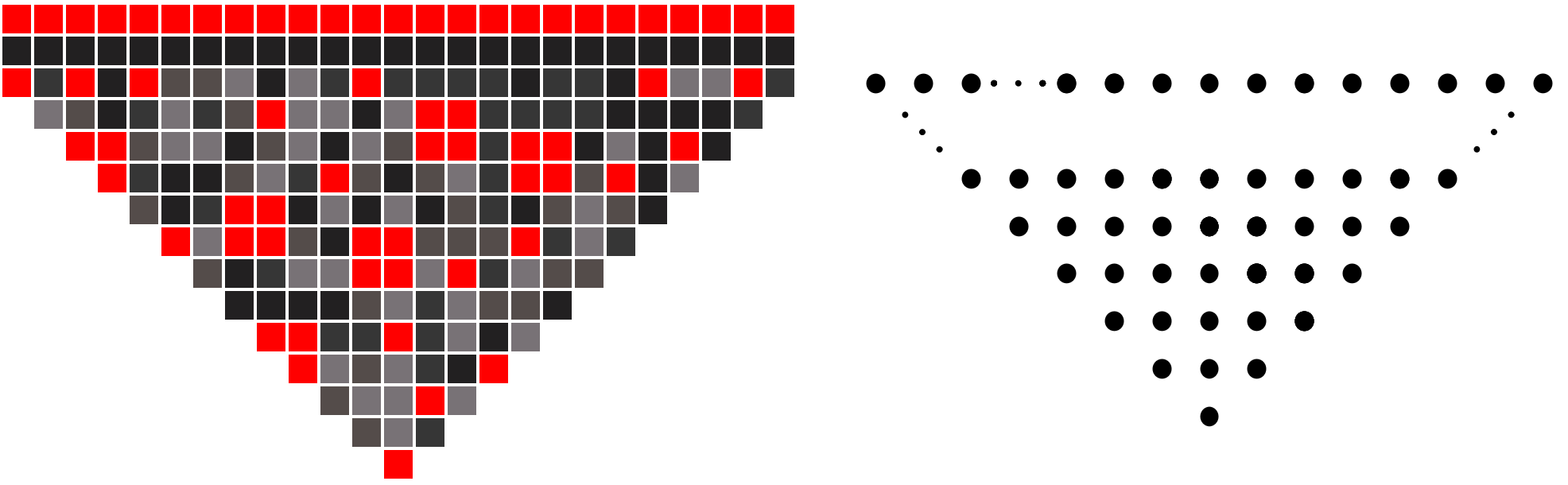}
\caption{\textbf{Left:} The number wall of a sequence of length 25 generated uniformly and randomly over $\F_5$. The zero entries are coloured in red, with the nonzero values assigned a shade of grey, with 1 being the darkest and 4 being the lightest. The top row (red) has index $-2$. \textbf{Right:} A blank dot diagram.}
\label{dotdiagram}
\end{figure}
\noindent All illustrations of \textit{specific} number walls have top row with index -2, as this illustrates all three cases of Definition \ref{nw}. This is different from dot diagrams, where the top row has index zero (unless stated otherwise).

\subsubsection*{Windows in Finite Number Walls}
\noindent Theorem \ref{window} states that every zero portion in a number wall is in the shape of a square. However, the left image in Figure \ref{dotdiagram} clearly shows windows (red) that are not square shaped. This is because Theorem \ref{window} is referring to \textit{infinite} number walls. When dealing with \textit{finite} sequences, it is possible that zero portions may appear in shapes other than a square. This can only happen on the edges of the defined part of the finite number wall. This does not violate Theorem \ref{window}, as in these cases \textit{any} continuation of the finite sequence that generated the finite number wall will eventually result in the non-square zero portion becoming square, even if that square has infinite side length. \\ 

\noindent The following three definitions categorise the types of window that can appear in a finite number wall. Throughout, a circle with a cross in it appearing in a dot diagram means the corresponding entry is nonzero. Recall that the inner frame of a window is the nonzero entries surrounding a square of zero entries.
\begin{definition}\label{complete}
    A \textbf{complete window} is a square of zeroes in a finite number wall where every entry of the inner frame is also contained within the defined part of the finite number wall. 
\end{definition}
\begin{figure}[H]
    \centering
    \includegraphics[width=0.5\linewidth]{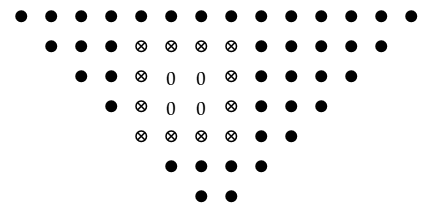}
    \caption{The two by two window above is complete, as the inner frame (denoted by the circles with crosses) is fully defined.}
\end{figure}
\begin{definition}\label{closed}  
A \textbf{closed window} is any zero portion of a finite number wall such that the ratios of the geometric sequences comprising its inner frame are calculable, but the inner frame itself is not completely within the defined part of the finite number wall. A \textbf{right-side closed window} is any non-square zero portion of the number wall such that the ratio of the inner frame on the right side is determined. A \textbf{left-side closed window} is defined similarly.\end{definition}
    \begin{figure}[H]
        \centering
        \includegraphics[width=1\linewidth]{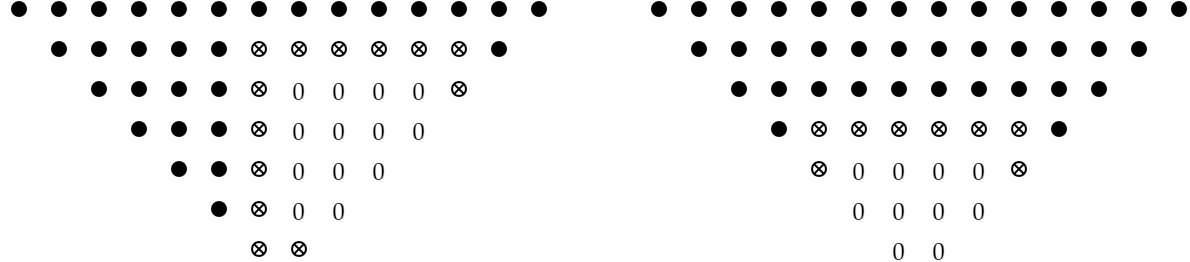}
        \caption{Two examples of closed windows. \textbf{Left:} at least two entries of each side of the inner frame are known, and therefore all four of the ratios of the geometric sequences comprising the inner frame are calculable. \textbf{Right:} all four ratios of the inner frames are calculable by Theorem \ref{ratio ratio}.}
        \label{fig:closedwindow}
    \end{figure}
    \vspace{-0.5cm}
    \begin{definition}\label{open_wind} An \textbf{open window} is any zero portion of the finite number wall that is not closed or complete. Similarly, a \textbf{right-side open window} (\textbf{left-side open window}, respectively) is a zero portion of a number wall where the ratio of the geometric sequence comprising the right-side (left-side, respectively) of the inner frame is not calculable. 
    \begin{figure}[H]
        \centering
        \includegraphics[width=1\linewidth]{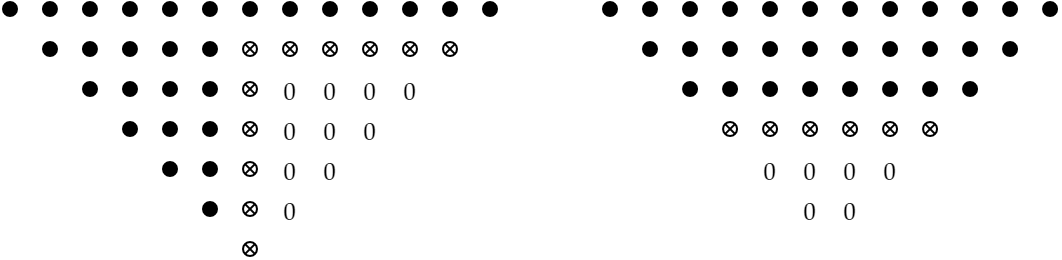}
        \caption{Two examples of open windows. The first example is both a left-side closed window and right-side open window. The second example is both a right-side open window and left-side open window. }
        \label{fig:openwindow}
    \end{figure}
   
\end{definition}

 \noindent A window is described as \textbf{incomplete} if it is not complete and it is not important to specify if it is an open window or a closed window. The word \textbf{window} will be used in isolation only when it does not matter if it is closed, open or complete.\\

 \noindent To help the reader commit these definitions to memory, links to them are found in the glossary. 

\subsubsection*{Windows and Square Portions}

\noindent The converse of Theorem \ref{growth} only demands that a particular square portion of the number wall should be zero; not that the infinite number wall contains a complete window of exactly the given size that begins exactly in the given place. To write about square portions more succinctly, the following notation is introduced:\\

\noindent \textbf{\hypertarget{sq_prt}{Notation}:} Given $l\in\N$ and $m,n\in\Z$, let $\Box(l,m,n)$ denote all the entries of a number wall (finite or infinite) within the square portion with side length $l$ and that has top left corner in row $m$ and column $n$.  \\

\noindent Two square portions \textbf{overlap} (or \textbf{intersect}) if they share entries. Recall that a finite number wall is an infinite array where only finitely many entries have been defined. Indeed, a square portion may then include (or be entirely comprised of) undefined variables. In fact, one often uses $\Box(l,m,n)$ to refer to a square portion that is completely outside of the defined portion of the finite number wall. See the red square in Figure \ref{fin_nw} for an example of this.\\

\noindent In general, there is no restriction on the values of entries inside a square portion. However, given a fixed (finite or infinite) sequence $\mathbf{S}$, some features in the number wall $W_q(\mathbf{S})$ (windows, inner frames, etc) force entries in certain square potions to take specific values. In particular, it is useful to have a way to describe if a (complete or incomplete) window that appears in $W_q(\mathbf{S})$ forces the values of a given square portion to be zero. Whence, the following definition.

 \begin{definition}\label{contain_def}
\noindent Let $\mathbf{S}$ be a fixed finite sequence over $\F_q$ and let $W_q(\mathbf{S})$ be its finite number wall. A (complete or incomplete) window $\mathcal{W}$ in $W_q(\mathbf{S})$ is said to \textbf{contain} a square portion $\Box(l,m,n)$ if for \textit{any} extension of $\mathbf{S}$ into a two-dimensional infinite sequence, $\mathcal{W}$ forces that every entry in $\Box(l,m,n)$ is equal to zero. 
\end{definition}
\noindent See Figure \ref{fig: contain} for an example illustrating Definition \ref{contain_def}. 
\begin{remark}
    At first glance, the use of the word ``contain'' in the above definition may seem misleading, since the square portion $\Box(l,m,n)$ may not be fully within the defined part of the finite number wall. However, this choice of terminology makes sense in view of Definition \ref{def: fin_nw} and the Square Window Theorem (Theorem \ref{window}). Indeed, if $\Box(l,m,n)$ is not fully within the defined part of the number wall but is contained within some incomplete window $\mathcal{W}$, then the entries of the finite number wall that are currently undefined but are within the minimal square containing $\mathcal{W}$ are forced to be zero. With this in mind, $\Box(l,m,n)$ is ``contained'' within $\mathcal{W}$. 
\end{remark}
 \begin{figure}[H]
    \centering
    \includegraphics[width=0.45\linewidth]{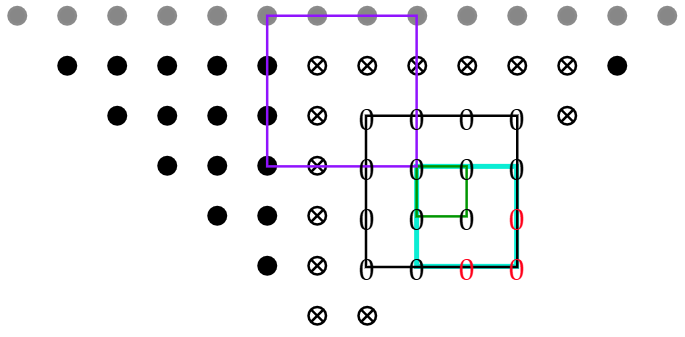}
    \caption{The incomplete window $\mathcal{W}$ is depicted by the entries denoted with 0 (in black). The parts of its inner frame that are in the defined part of the finite number wall are depicted by the crossed dots. Indeed, $\mathcal{W}$ contains the square portion (black square), since the red entries will be zero regardless of how the generating sequence (grey dots, top row) is continued. The same incomplete window also contains the blue and green square portions, despite being a different size and beginning in different positions. The purple square portion is not contained within $\mathcal{W}$, as it contains nonzero entries (the crossed dots).}
    \label{fig: contain}
\end{figure}

\subsubsection*{Extending Finite Sequences}
\noindent Many results in this paper revolve around counting the number of ways one can extend a given finite sequence to ensure the resulting sequence has certain properties in its number wall. To this end, the following notation is introduced which formalises the concept of continuing or extending a finite sequence.\\

\noindent \textbf{\hypertarget{seq_extend}{Notation}:} Given $l_{\mathbf{S}},l_R\in\N$ and given two sequences $\mathbf{S}=(s_i)_{0\le i \le l_{\mathbf{S}}}$ and $\textbf{R}=(r_i)_{0\le i \le l_R}$, define $ \mathbf{S}\oplus \mathbf{R}$  as the concatenation of $\mathbf{S}$ and $\textbf{R}$: \begin{equation}
        \textbf{S}\oplus\textbf{ R} = (k_i)_{0\le i \le l_{\mathbf{S}}+l_\textbf{R}+1} ~~~\text{ where }~~~ k_i=\begin{cases}s_i &\text{ if } 0\le i \le l_{\mathbf{S}}\\ r_{i-l_{\mathbf{S}}-1} &\text{ if } l_{\mathbf{S}} < i \le l_{\mathbf{S}}+l_\textbf{R}+1.\end{cases}
    \end{equation}

\subsubsection*{Minimal Generators}
\noindent Given a square portion $\Box(l,m,n)$ of a number wall, one is able to find the minimal length of a finite sequence $\mathbf{S}$ such that $W_q(\mathbf{S})$ has a window containing $\Box(l,m,n)$. 

\begin{definition}\label{hat} Let $\Box=\Box(l,m,n)$ be a square portion of a number wall defined as above. A \textbf{generator} of $\Box$ is a sequence $\mathbf{S}=(s_i)_{n_1\le i \le n_2}$ such that, regardless of the choice of additional values for $s_i$ when $i<n_1$ and $i>n_2$, the infinite number wall $W_q((s_i)_{i<n_1}\oplus \textbf{S}\oplus (s_i)_{i>n_2})$ has zeroes in every entry of $\Box(l,m,n)$. A \textbf{minimal generator} for $\Box$ is a generator for $\Box$ of minimal length. 
\end{definition}
\begin{figure}[H]
    \centering
    \includegraphics[width=0.5\linewidth]{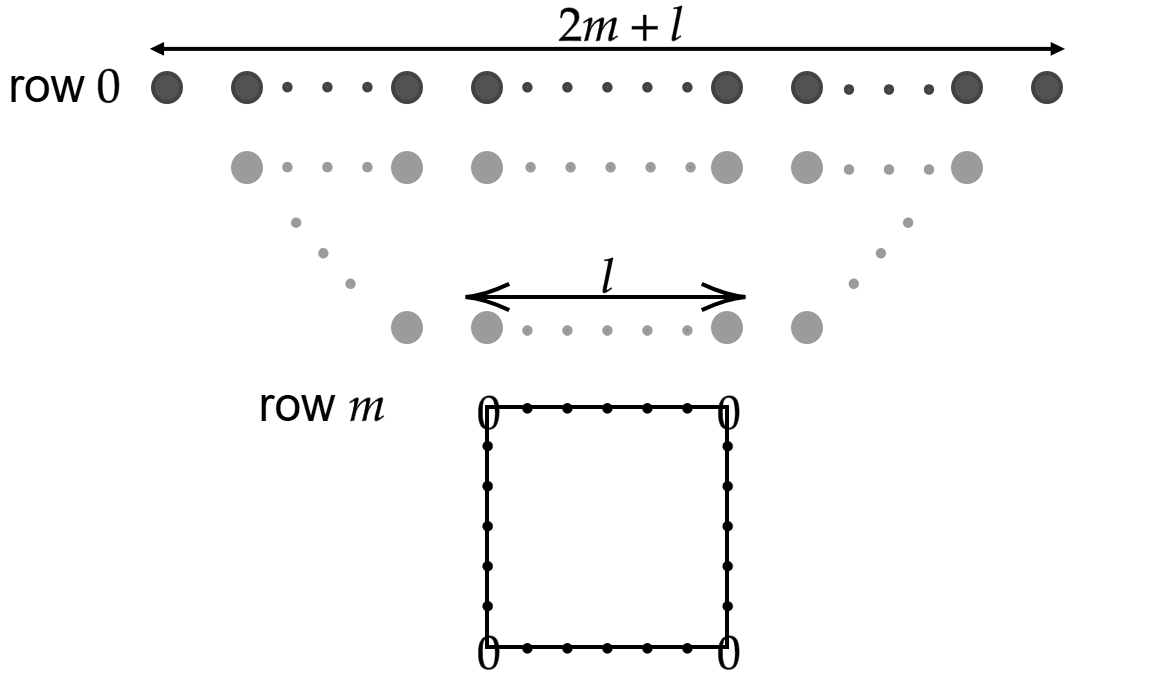}
    \caption{The square portion $\Box(l,m,n)$ is denoted by the black square and a minimal generator of $\Box(l,m,n)$ is the dots coloured in dark grey.}
\end{figure}
\noindent There are many different minimal generators for any given square portion $\Box(l,n,m)$. However, it is also clear that each of these minimal generators has length $2m+l$. Furthermore, for a sequence $\mathbf{S}=(s_i)_{1\le i \le r}$ such that $W_q(\mathbf{S})$ has a (complete or incomplete) window containing $\Box(l,n,m)$, a minimal generator of $\Box(l,m,n)$ is given by the subsequence $(s_i)_{n-m\le i \le n+m+l-1}$ (note here, that $n-m\ge1$ from the triangular shape of the defined part of the number wall, and that $n+m+l-1\le r$ from the assumption that $W_q(\mathbf{S})$ has a window containing $\Box$). \\

\noindent For the purposes of this paper, the specific choice of minimal generator is not important. Instead, given a square portion $\Box(l,m,n)$, the length and position of a minimal generator (which only depend on the choice of square portion) of $\Box(l,m,n)$ are what matter. 

\subsubsection*{Diagonals of Finite Number Walls}

\noindent Due to the triangular shape of the defined part of a finite number wall, it is natural for one to consider the entries that appear not on a specific row or column, but instead in a given diagonal. This is defined below. Recall that, for a finite sequence $\mathbf{S}$ over a field $\F_q$, the finite number wall $W_q(\mathbf{S})$ is still an infinite array, but only finitely many of the entries are known. Therefore, the entries of $W_q(\mathbf{S})$ are indexed as $(W_{m,n})_{n,m\in\Z}$.
\begin{definition}
    \label{diagonal} Let $\mathbf{S}=(s_i)_{1\le i \le r}$ be a finite sequence of length $r$, and let $W_q(\mathbf{S})=(W_{m,n}(\mathbf{S}))_{m,n\in\Z}$ be the number wall generated by $\mathbf{S}$ over $\F_q$. For $k\in\N$, the $k^\text{th}$-\textbf{diagonal} is all the elements of the finite number wall that have column index $k-i$ and row index $i$ for $0\le i \le \left\lfloor\frac{k-1}{2}\right\rfloor$.
\end{definition}

\section{Lemmata on the Size of Zero Entries in Number Walls}\label{Sect:Lemmata}

\noindent This section covers all the lemmata that are required for the proofs of Theorems \ref{metric} and \ref{log2}. The proofs of these lemmata relies on a deep theory of combinatorics on number walls, which is postponed to Section \ref{Sect: proofs}.

\subsection{Increasing the Size of an Open Window}
\noindent Open windows in finite number walls have the property that they can increase in size depending on how the finite sequence that generates the finite number wall is extended. The following lemma counts how many ways one can extend a given finite sequence such that a chosen open window in its finite number wall increases in size. 
\begin{prop}\label{wind_extend_lem}
    Let $l$ be a natural number and let $\mathbf{S}=(s_i)_{0\le i \le r}$ be a finite sequence over the field $\F_q$. Assume that the finite number wall of $\mathbf{S}$, denoted $W_q(\mathbf{S})$, has a right-side open window $\mathcal{W}$ of size $l$. Then, given another natural number $l'>l$, there exists a unique sequence $ \mathbf{S}'=(s_i)_{r+1\le i \le r+l'-l}$ over $\F_q$ such that $\mathcal{W}$ is extended to become a right-side open window of size $l'$ in the finite number wall $W_q( \mathbf{S}\oplus\mathbf{S}')$.
\end{prop}
\noindent Proposition \ref{wind_extend_lem} is illustrated below.

\begin{figure}[H]
    \centering
    \includegraphics[width=0.65\linewidth]{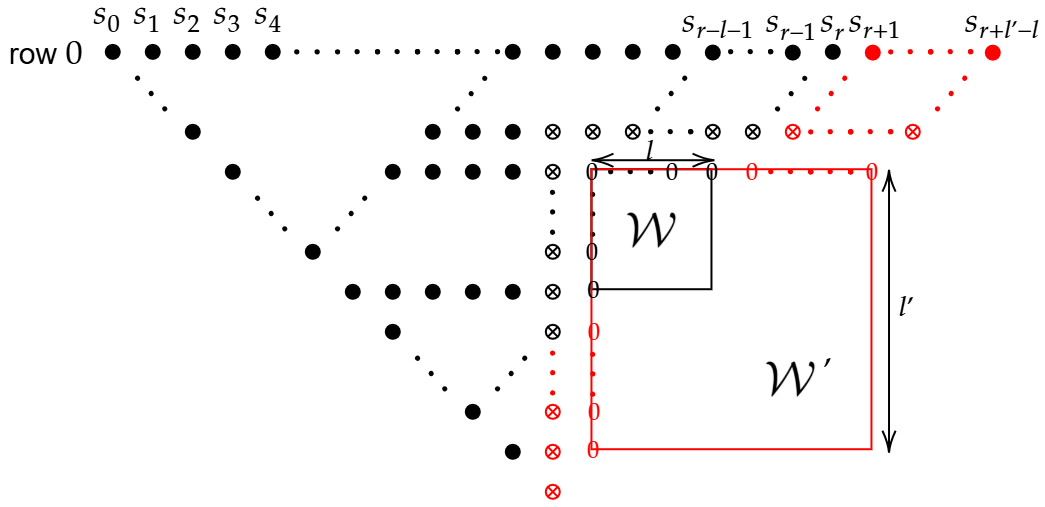}
    \caption{The number wall $W_q(\mathbf{S})$ (black dots) has a right-side open window $\mathcal{W}$ (black square), whose inner frame is denoted by the black crossed dots. There is a unique sequence $ \mathbf{S}'=(s_i)_{r+1\le i \le r+l'-l}$ (red dots, top row) that extends $\mathbf{S}$ (black dots, top row) such that $W_q( \mathbf{S}\oplus\mathbf{S}')$ has the window shown by the large red square.  }
\end{figure}
\subsection{Extensions to Finite Number Walls}
\noindent The following lemma shows how many ways a finite sequence can be extended so that its number wall has a (complete or incomplete) window containing a given square portion.
\begin{lemma} \label{contain}
Let $r$ be a natural number and let $\mathbf{S}$ be a sequence of length $r$. Additionally, let $m',n'$ and $l$ be natural numbers such that any minimal generator of the square portion $\Box=\Box(l,m',n')$ is not a subsequence of $\mathbf{S}$ but overlaps with the final $2k+i$ entries of $\mathbf{S}$, for some $k\in\N$ and $i\in\{1,2\}$. Call the sequence defined by this overlap $ \mathbf{S}_k$, and assume that there are no windows in $W_q( \mathbf{S}_k)$ that intersect or contain $\Box$.  Finally, define $m:=m'-k$ (in such a way that $2m+1-i$ is the number of entries one must extend $ \mathbf{S}_k$ by so that the greatest row index containing defined entries of the number wall of the resulting sequence is $m'$). \\

\noindent Then, the number of extensions to $ \mathbf{S}'=(s_j)_{r\le j \le r+2m+l-i}$ of length $2m+l+1-i$ such that $W_q( \mathbf{S}\oplus\mathbf{S}')$ has a window containing $\Box$ is equal to $q^{2m+1-i}$.
\end{lemma}
\noindent The set up in the $i=1$ case is illustrated below.
\begin{figure}[H]
    \centering
    \includegraphics[width=0.70\linewidth]{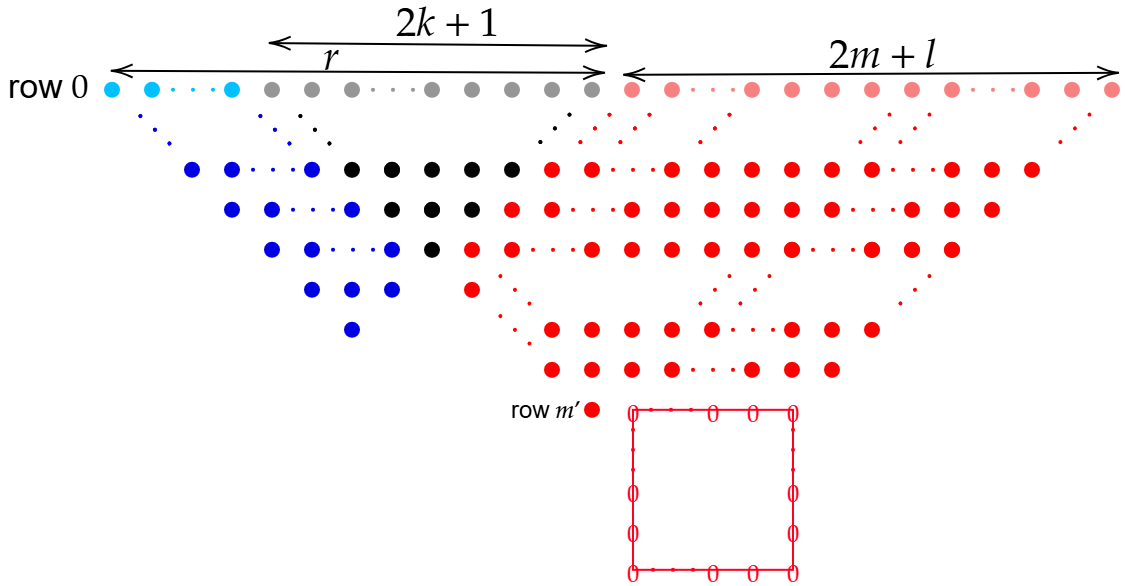}
    \caption{The finite sequence $\mathbf{S}$ (light blue and grey dots) of length $r$ generates a finite number wall (light blue, dark blue, grey and black dots). The subsequence $ \mathbf{S}_k$ (grey dots) generates the finite number wall (grey and black dots) $W_q( \mathbf{S}_k)$. The sequence $ \mathbf{S}'$ (light red dots) extends $ \mathbf{S}_k$ in such a way that $W_q( \mathbf{S}_k\oplus \mathbf{S}')$ (grey, black, light and dark red dots) has a window containing the square portion $\Box(l,m',n')$ (red square). }
    \label{contain_fig1}
\end{figure}

\noindent Both Lemma \ref{contain} and the following corollary play a key role in the proof of Theorems \ref{metric} and \ref{log2}.
\begin{cor}\label{contain full}
Let $q$ be a prime power, $r,n,m,l$ be natural numbers and let $\Box=\Box(l,n,m)$ be a square portion in a number wall over $\F_q$. The number of sequences of length $r$ over $\F_q$ whose number wall has a window containing $\Box$ is $q^{r-l}$ if $r>m+n+l$ and zero otherwise.
\end{cor} 

\subsection{Rectangular Zero Portions}

\noindent Whilst windows are always in a square shape, determining how many finite sequences have a (complete or incomplete) window that contains a given rectangular portion is essential for proving Theorem \ref{metric}. To this end, the following definition is made:

\begin{definition}\label{Rect_prtn}
Let $r,n,m,l\in\N$, $d\in\Z$ satisfying $d<l$ and let $q$ be a prime power. Define $R_{r,q}(l,d,m,n)$ as the set of finite sequences of length $r$ over $\F_q$ whose number walls have zero entries comprising every part of the rectangular portion with horizontal length $l$, height $l-d$ and top left corner in row $m$ and column $n$.
\end{definition}
\noindent Explicit formulae for $R_{r,q}(l,d,m,n)$ are proved in Lemma \ref{Rect} (see Section \ref{Sect:Rect}). However, for the purposes of proving Theorem \ref{metric}, the following proposition of the aforementioned lemma provides an upper bound that is sufficient. 
\begin{prop}\label{bound}
    Let $l\in\N$, $d\in\Z\backslash\{0\}$ and define $\widetilde{l}=\max\{l,l-d\}$. Then \[\#R_{r,q}({l,d,n,m})\ll_q |d|\cdot q^{r-\widetilde l}.\] 
\end{prop}

\subsection{Pairs of Windows}\label{Sect: Pair of Wind}
\noindent This subsection provides the final result needed to prove Theorem \ref{metric}. Given square portions $\Box_1=\Box(l_1,n_1,m_1)$ and $\Box_2=\Box(l_2,n_2,m_2)$, define the set $W_{r,q}({\Box_1,\Box_2})$ as all the sequences of length $r$ over $\F_q$ whose number wall has (complete or incomplete) windows that contain $\Box_1$ and $\Box_2$. 
\begin{lemma}
   Let $\Box_1=(l_1,n_1,m_1)$ and $\Box_2=(l_2,n_2,m_2)$ be non-overlapping square portions. Furthermore, let $r$ be a natural number that is suitably large so that a sequence of length $r$ generates a number wall where all the entries in both $\Box_1$ and $\Box_2$ are defined. Then, \begin{equation}\#W_{r,q}(\Box_1,\Box_2)\ll_q q^{r-l_1-l_2}.\label{twowindbound}\end{equation} \label{twowind}
\end{lemma}
\section{Hausdorff Dimension of the Set of Counterexamples to $t$-LC with Additional Growth Function}\label{Sect:HD}

\noindent Recall the set $M(t,f)$ from equation \ref{M_q(P(t),f)}. Theorem \ref{growth} shows that if a Laurent series $\Theta(t)$ is in $M(t,f)$, then there exists an $l\in\N$ such that for any $k\in\N$, the square zero portions with top left corner on the $k^\nth$ column and $m^\nth$ row (with $m<k$) of  $W_q(\Theta(t))$ have size less than or equal to $l+\lceil \log_q(f(q^k))\rceil -1$. For the duration of this section, a sequence satisfying this property about the size of its windows as $k$ tends to infinity is said to satisfy the \textbf{window growth property with respect to function} $f$, or just the ``window growth property" for short. The strategy of the proof is to construct a Cantor set of sequences with number walls satisfying this window growth property and to show it has full dimension.\\

\noindent To this end, the statement below (from \cite[Lemma2]{HillVelani}) is introduced to find the Hausdorff dimension of Cantor sets. To state it, define the \textbf{diameter} of  $X\subset\F_q\!\left(\!\left(t^{-1}\right)\!\right)$ as \[\diam(X)=\sup\{|\Theta(t)-\Phi(t)|:\Theta(t),\Phi(t)\in X\}.\]
\begin{lemma}[Mass Distribution Principle]\label{MDP}
  Let $\mu$ be a probability measure supported on a subset X of $\F_q\!\left(\!\left(t^{-1}\right)\!\right)$. Suppose there are positive constants $a$, $s$ and $\rho$ such that for any ball $B$ such that $\diam(B)\le \rho$, one has\[\mu(B)\le a\cdot\diam(B)^s.\] Then, $\dim(X)\ge s$.
\end{lemma}

 \noindent Before the proof of Theorem \ref{log2} begins in earnest, note that the Cantor set constructed by removing all sequences with a window of size \begin{equation}L_k:=l+\lceil \log_q(\log^2_q(q^k))\rceil -1\label{L_k}\end{equation} with their top left corner on \textit{diagonal} $k$ (instead of column $k$) is also a subset $M_q(t,\log^2)$. Indeed, every entry on diagonal $k$ is in column at least $\frac{k}{2}$. Since $f=\log_q^2$, the size a window in column $k/2$ needs to be to violate the window growth property is $l+\lceil \log_q(k^2)-\log_q\left(4\right)\rceil-1$. The constant $-\log_q\left(4\right)$ is absorbed into $l$, as $l$ can be any fixed natural number and so is assumed to be greater than 1. Due to this shift in focus from columns to diagonals, the following notation is introduced.\\
 
 \noindent \hypertarget{diag_prtn}{\textbf{Notation:}}  For natural numbers $l,k$ and $m$, define $\widehat\Box(l,m,k)$ as the square portion of side length $l$ with its top left corner on row $m$ and diagonal $k$. That is, $\widehat\Box(l,m,k)=\Box(l,m, k-m)$.\\

 \noindent Finally, in order to build the aforementioned Cantor set, one must tightly control the size of windows that appear in the final diagonals of a finite number wall.\begin{definition}\label{tapered}
     Let $\mathbf{S}$ be a sequence of length $N\in\N$ over $\F_q$ and let $l\in\N$. The finite number wall $W_q(\mathbf{S})$ is $l$\textbf{-tapered} if there is no incomplete window containing a square portion $\widehat\Box(l,m,N-l)$ for any $0\le m \le \lceil l/2\rceil$.
 \end{definition}
 \begin{figure}[H]
     \centering
     \includegraphics[width=0.5\linewidth]{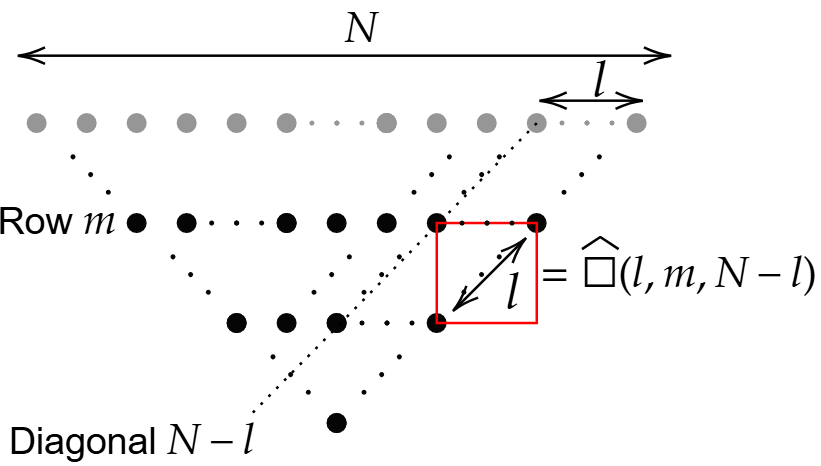}
     \caption{A finite sequence (grey dots, top row) is $l$-tapered if its number wall $W_q(\mathbf{S})$ (all dots) has no incomplete window that contain square portions of size $l$ with top left corner on diagonal $N-l$ (dotted line). One such square portion is illustrated by the red square.}
 \end{figure}

 \noindent In other words, a sequence being $l$-tapered is equivalent to there being no strings of $l$ consecutive zero entries along the right-most diagonal of its number wall. \\
 
 \noindent Given a sequence of length $q^n$, that satisfies the window growth property, the goal is to count how many ways one can extend it into a sequence of length $q^{n+1}$ that also satisfies the window growth property. To this end, recall the definition of $L_k$ (equation (\ref{L_k})) and abbreviate the ``Window Growth Property with respect to $\log^2$'' to ``$\log^2$-WGP''. Then, for $n\in\N$, define the following sets: \begin{itemize}
     \item $X_n=\{\mathbf{S}:=(s_i)_{1\le i \le q^n}: s_i\in\F_q,~S\text{ is }(L_{q^n}/2)\text{-tapered and }W_q(\mathbf{S})\text{ satisfies }\log^2\text{-WGP}\}; $
     \item $Y_n$ a subset of $X_n$ that is defined inductively on $n$ by a process detailed in Section \ref{Sect: Y_n}, the base case being given by $Y_0 = \{y\in\F_q: y\neq0\};$
     \item $\mathcal{J}(Y_n)=\left\{B\left(\sum_{i=1}^{q^n}y_it^{-i},q^{-q^n}\right): (y_i)_{1\le i \le q^n}\in Y_n\right\}$ (recall $B(\cdot,\cdot)$ from Definition \ref{ball});
     \item $\mathcal{I}(Y_{n})=\left\{y\oplus z: y\in Y_n, z\in\F_q^{q^{n}(q-1)}\right\}.$
 \end{itemize} The aforementioned Cantor set $\mathcal{C}$ that lives within $M_q(t,\log^2)$ will be the intersection of $\mathcal{J}(Y_n)$ over all $n\in\N$. The elements of $Y_n$ are used to construct a set $Y_{n+1}\subset X_{n+1}$ and hence the corresponding set $\mathcal{J}(Y_{n+1})$. To do this, each sequence in $Y_n$ is extended into \begin{equation}R_n:=q^{q^n(q-1)}\label{R_n}\end{equation} different sequences of length $q^{n+1}$. For convenience, define $R_{-1}:=q$. The union of these $R_n$ sequences over every $y\in Y_n$ collectively makes the set $\mathcal{I}(Y_{n})$. In other words, \begin{equation*}
      \#\mathcal{I}(Y_{n})=\#Y_n\cdot R_n.
 \end{equation*} Every sequence in $\mathcal{I}(Y_{n})$ that has a window of size larger than $L_k$ with its top left corner on diagonal $k$ in its number wall needs to be removed in order to build $Y_{n+1}$.\\

 \noindent The proof of Theorem \ref{log2} is split into two subsections, with the first calculating how many sequences in $\mathcal{I}(Y_{n})$ are to be thrown away to construct $Y_{n+1}$, and the second calculating the Hausdorff dimension of the resulting Cantor set. The former subsection implements the Proposition \ref{wind_extend_lem} and Lemma \ref{contain}, whilst the latter is more standard in its calculations.
\subsection{Definition of the Level Sets}\label{Sect: Y_n}
\noindent This subsection details the construction of $Y_{n+1}$ and establishes that the cardinality of $Y_{n+1}$ satisfies \[\#Y_{n+1}\ge \#Y_n \cdot R_n\cdot \left(1-q^{2-l}-q^{1-l/2}-q^{-l/2}\right).\]

 \noindent As the number wall generated by the first $q^n$ entries already satisfies $\log^2$-WGP by induction, only the diagonals with index $q^n<k\le q^{n+1}$ are checked. \\

\noindent Let $S\in Y_{n}$ be a finite sequence of length $q^{n}$. There are precisely $R_n$ elements of $\mathcal{I}(Y_n)$ that are extensions of $\mathbf{S}$. Consider the finite number wall generated by one such sequence. 

\begin{figure}[H]
    \centering
    \includegraphics[width=0.6\linewidth]{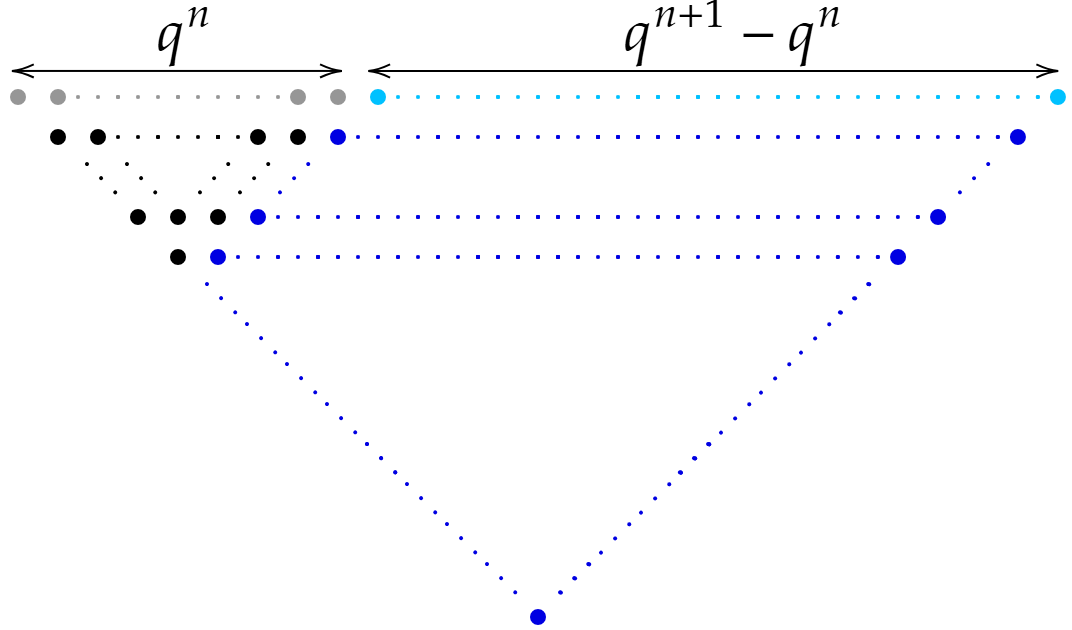}
    \caption{The sequence $\mathbf{S}$ (grey dots) generates a finite number wall (black and grey dots). When $\mathbf{S}$ is extended by $q^{n+1}-q^n$ entries (light blue dots), a large finite number wall is generated (whole picture).}\label{In_wall}
\end{figure}

\noindent The goal is to attain a lower bound for how many of the $R_n$ possible extensions to $\mathbf{S}$ exist whose number wall does not have a window that violates the window growth property. To achieve this, for every natural number $q^n<k\le q^{n+1}$, one removes every extension to $\mathbf{S}$ that generates a finite number wall with a window of size $L_k$ with top left corner on each entry of anti-diagonal $k$. To this end, denote by $\mathbf{S}'$ an extension to $\mathbf{S}$ of length $q^{n+1}$.

\subsubsection*{Partitioning $W_q(\mathbf{S}')$ into Four Parts}

\noindent  The number wall generated by $\mathbf{S}'$ is partitioned into four parts ($P_1, P_2,$ $P_3,P_4$), illustrated by the black, blue, green and red parts in Figure \ref{Cantorpic1}. For now, ignore the fact that some entries are coloured in darker shades.

\begin{figure}[H]
     \centering
     \includegraphics[width=0.6\linewidth]{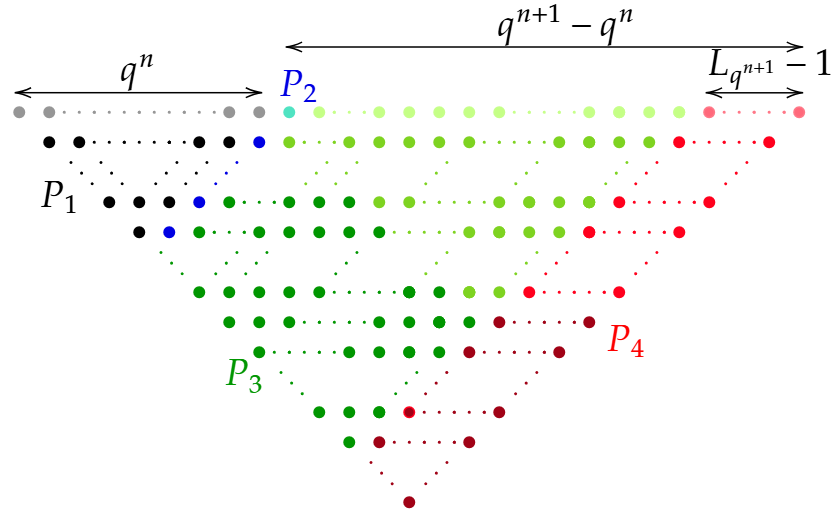}
     \caption{The sequence $\mathbf{S}$ (grey dots, top row) is extended into a sequence $\mathbf{S}'$ (top row, all colours). This figure shows the number wall $W_q(\mathbf{S}')$ split into four parts. }
     \label{Cantorpic1}
 \end{figure}

\noindent The first section of this partition, labelled $P_1$, is $W_q(\mathbf{S})$ (black dots in Figure \ref{Cantorpic1}). This is assumed to already satisfy $\log^2$-WGP. The second part of $P_2$ (blue dots in Figure \ref{Cantorpic1}) is defined as anti-diagonal $q^n+1$ of $W_q(\mathbf{S}')$. The third section, $P_3$ (green dots in Figure \ref{Cantorpic1}) is all the anti-diagonals with index $q^n+1<k\le q^{n+1}-L_{q^{n+1}}$. Finally, the fourth segment, $P_4$, (red dots in Figure \ref{Cantorpic1}) is given by the anti-diagonals with indices $q^{n+1}-L_{q^{n+1}}+1\le k\le q^{n+1}$.

\subsubsection*{Sequences with Large Windows whose Top Left Corner is in $P_2$}

\noindent First, one throws away every sequence $\mathbf{S}'\in\mathcal{I}(Y_n)$ that has a window violating $\log^2$-WGP in $W_q(\mathbf{S}')$ with top left corner in $P_2$. Recall that, by induction, it is assumed that $\mathbf{S}$ is $L_{q^n}/2$-tapered, where $L_{q^n}/2$ was defined in equation (\ref{L_k}).\\

\noindent As one is interested in windows with top left corner in $P_2$, one has that $k=q^n+1$. For each value of row index $m$ in the range $0\le m \le (q^n-1)/2$, there are three possibilities for how one could extend $\mathbf{S}$ into a sequence with a window containing $\widehat\Box(L_k,m,k)$. \begin{itemize}
    \item \textbf{Case 1:} there is no window in $W_q(\mathbf{S})$ that intersects with the square portion $\widehat\Box(L_k,m,k)$. The plan here is to apply Lemma \ref{contain}. To do this, note that any minimal generator of $\widehat\Box(L_k,m,k)$ is not a subsequence of $\mathbf{S}$ but intersects the final $2m$ entries of $\mathbf{S}$, illustrated below. \begin{figure}[H]
        \centering
        \includegraphics[width=0.7\linewidth]{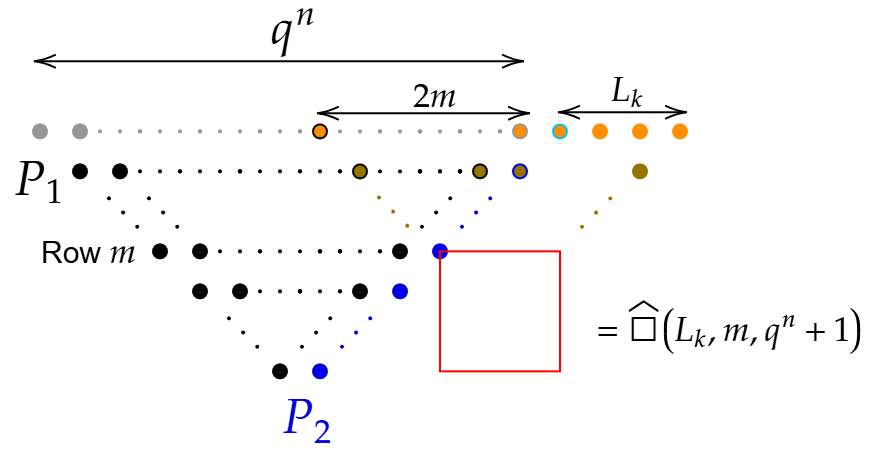}
        \caption{Some dots in this diagram have two colours, with one colour in the interior and the other being the boundary of a dot. The sequence $\mathbf{S}$ (grey dots) generates the number wall $W_q(\mathbf{S})$ (black and grey dots). There is no window in $W_q(\mathbf{S})$ that intersects $\widehat\Box(L_k,m,k)$ (red square), which has top left corner in $P_2$ (blue dots). The entries that make up a minimal generator of $\widehat\Box(L_k,m,k)$ are in orange. Precisely $2m$ of these entries overlap with the end of $\textbf{S}$, whilst $L_k$ of them do not.}
    \end{figure} These final $2m$ entries of $\mathbf{S}$ are called $\mathbf{S}_m$. Therefore, by the assumption of Case 1, the conditions of the $i=2$ case of Lemma \ref{contain} are satisfied with the following values. The left-side of each correspondence below refers to notation from Lemma \ref{contain} and the right-side from this proof. \begin{align*}
        &\mathbf{S}_k\leftrightarrow \mathbf{S}_m& &2k+2\leftrightarrow 2m& &m'\leftrightarrow m& &m\leftrightarrow1& &l\leftrightarrow L_k&
    \end{align*} Recall $R_n$ from equation (\ref{R_n}). Hence, of the possible $R_n$ extensions to $\mathbf{S}$ to a sequence of length $q^{n+1}$, $R_n\cdot q^{-L_k}$ of them have a window with top left corner on row $m$ and anti-diagonal $k=q^n+1$ that violates $\log^2$-WGP. Therefore, $R_n\cdot q^{-L_k}$ entries $\mathcal{I}(Y_n)$ are thrown away. 
    \item \textbf{Case 2:} there is a right-side closed window in $W_q(\mathbf{S})$ that intersects $\widehat\Box(L_k,m,k)$. In this case, there are no possible extensions to $\mathbf{S}$ that will result in a window that contains $\widehat\Box(L_k,m,k)$. Note, that since $\mathbf{S}$ is $L_{q^n/2}$ tapered, any such window in $W_q(\mathbf{S})$ cannot contain $\widehat\Box(L_k,m,k)$.
    \item \textbf{Case 3:} there is a right-side open window $\mathcal{W}$ in $W_q(\mathbf{S})$ that intersects $\widehat\Box(L_k,m,k)$. Since this window is right-side open, it can be extended using Proposition \ref{wind_extend_lem} to contain $\widehat\Box(L_k,m,k)$. However, as $\mathcal{W}$ is size $L_k/2 - i$, for some $0\le i \le L_k/2-1$, $\mathcal{W}$ only needs to be extended by $L_k/2+i$ entries in order to violate the window growth property. In this case, this results in throwing away $R_n\cdot q^{-(L_k/2+i)}$ entries of $\mathcal{I}(Y_n)$. This value is maximised when $i=0$. \end{itemize}
\noindent Out of these three cases, Case 3 throws away the largest number of extensions to $\mathbf{S}$. Recall the definition of $L_k$ from equation (\ref{L_k}). For any $q^n<k\le q^{n+\frac{1}{2}}$, $L_k=l-1+(2n+1)$. Therefore, summing over the $(q^n+1)/2$ choices for $m$, at most \begin{align}
    R_n\cdot q^{-L_k/2}\cdot \frac{q^n+1}{2}&=R_n\cdot q^{(-l+1)/(2)-n}\cdot \frac{q^n+1}{2}\nonumber\\&=R_n\cdot \frac{q^{-l/2}+q^{(-l+1)/2-n}}{2}\nonumber\\&<R_n\cdot q^{-l/2}\label{blue}
\end{align} of the possible $R_n$ extensions to $\textbf{S}$ are removed from $\mathcal{I}(Y_n)$. 
\subsubsection*{Sequences with Large Windows whose Top Left Corner is in $P_3$}

\noindent The next section of this partition to be dealt with is $P_3$. Recall, these are all the entries of $W_q(\mathbf{S}')$ on anti-diagonal $k$ with $k$ in the range $q^n+2\le k\le q^{n+1}-L_{q^{n+1}}$ and row $m\in\N$ satisfying $0\le m \le \left\lceil \frac{k}{2}\right\rceil-1$. The plan is to count how many extensions to $\mathbf{S}$ can have a window containing the square portion $\widehat\Box(L_k,m,k)$. \\

\noindent To this end, $P_3$ is partitioned into two further sections, depicted by the medium and dark shades of green in Figure \ref{Cantorpic1}. Specifically, $P_{3,1}$ (medium green in Figure \ref{Cantorpic1}) is defined as all the entries of $P_3$ such that any minimal generator of any square portion with top left corner at this entry would not intersect $\mathbf{S}$. Then, $P_{3,2}$ is the remaining elements of $P_3$.\\

\noindent If $m,k\in\N$ are such that the entry of the number wall in row $m$ and anti-diagonal $k$ is in $P_{3,1}$, then one can apply Corollary \ref{contain full} to count all the $R_n\cdot q^{-L_k}$ possible extensions of $\mathbf{S}$ that have a window containing $\widehat\Box(L_k,m,k)$.\\

\noindent On the other hand, if the entry of the number wall in row $m$ and anti-diagonal $k$ is in $P_{3,2}$, then the plan is to apply Lemma \ref{contain} to show there are $R_n\cdot q^{-L_k}$ possible extensions to $\mathbf{S}$ that have number walls with windows containing $\Box'(L_k,m,k)$. However, Lemma \ref{contain} assumes that there is no window in $W_q(\mathbf{S})$ that contains the entry of $W_q(\mathbf{S}')$ in anti-diagonal $k$ and row $m$. Therefore, one now verifies that this condition is satisfied.\\

\noindent Say there were some window in $W_q(\mathbf{S})$ that contained the entry in row $m$ and anti-diagonal $k$. Then, the same window necessarily contains an entry from $P_2$. Therefore, any such sequence $\mathbf{S}$ has already been dealt with in the previous case, and, as an upper bound, it is sufficient to assume that Lemma \ref{contain} can be applied to every entry of $W_q(\mathbf{S}')$ in $P_3$.\\

\noindent One is now able to count how many entries of $\mathcal{I}(Y_n)$ are being thrown away. The number of entries of $W_q(\mathbf{S}')$ in $P_3$ is trivially bounded from above by $(q^{n+1})^2$. Furthermore, when $q^n<k\le q^{n+1}$, $L_k$ is equal to $$L_k=l+\lceil\log_q(k^2)\rceil -1=\begin{cases}
    l+2n+1 &\text{ if } q^n<k\le q^{n+1/2},\\ l+2n &\text{ if }q^{n+1/2}<k\le q^{n+1}.
\end{cases}$$ Therefore, the maximum one is throwing away comes from the second of the above cases, and is at most \begin{equation}R_n\cdot q^{2n+2-(l+2n)}=R_n\cdot q^{2-l}\label{green}\end{equation} sequences from $\mathcal{I}(Y_n)$.

\subsubsection*{Sequences with Large Windows whose Top Left Corner is in $P_4$}

\noindent Now, assume that $m,k\in\N$ are such that the entry of $W_q(\mathbf{S}')$ in row $m$ and anti-diagonal $k$ falls in $P_4$. As this section has width $L_k-1$, there does not exist any continuation to $\mathbf{S}$ (denoted $\mathbf{S}'$) that has a window with top left corner in $P_4$ violating the window growth property with respect to $\log^2$. However, it is required that every sequence in $Y_{n+1}$ is $L_k/2$-tapered. Therefore, by the same method as in $P_3$, there are $R_n\cdot q^{-\frac{L_k}{2}}$ such continuations that have a window of size $L_k/2$ in a given entry of anti-diagonal $q^{n+1}-L_k/2$. Furthermore, there are less than $q^{n+1}/2$ such rows, and since $L_k=2n+l$, one must throw away at most \begin{equation}\label{red}
    R_n\cdot q^{-(l+2n)/2}\cdot \frac{q^{n+1}}{2}<R_n\cdot q^{-l/2+1}
\end{equation}
entries of $\mathcal{I}(Y_n)$.\\

\noindent It may appear like one must also throw away all the extensions to $\mathbf{S}$ that contain a window of size $L_k/2+j$ with top left corner on anti-diagonal $q^{n+1}-L_k/2-j$. However, upon closer inspection, each such window necessarily contains a square portion of size $L_k/2$ with top left corner on anti-diagonal $q^{n+1}-L_k/2$, and therefore these extensions have already been removed.

\subsubsection*{The Total to be Removed from $\mathcal{I}(Y_n)$}

\noindent Combining equations (\ref{green}), (\ref{red}) and (\ref{blue}) shows that there are at least \begin{equation}t_n:=R_n\left(1-q^{2-l}-q^{1-l/2}-q^{-l/2}\right)\label{tn}\end{equation}extensions to $\textbf{S}$ that fall within $Y_{n+1}$. This is always positive if $l$ is sufficiently large. 
\subsection{Constructing the Cantor Set}
\noindent The Hausdorff dimension of the Cantor set $\mathcal{C}\subset \textbf{Mad}(t,\log^2,q)$ is now calculated. The method closely follows both \cite[Lemma 1]{cantor} and \cite[Lemma 3.1]{Bad_Harr}, but neither of these results are applicable here since \cite[Lemma 1]{cantor} is only valid for Cantor sets over $\R$, and \cite[Lemma 3.1]{Bad_Harr} requires that for any $\delta>0$, $\prod_{i=0}^nR_i^\delta>R_n$ holds for sufficiently large $n$. 
 \begin{proof}[Proof of Theorem \ref{log2}]
 Define \[r_n:=R_n\cdot\left(q^{2-l}+q^{1-l/2}+q^{-l/2}\right).\] Let $\mathbf{K}(\mathbb{I},\mathbf{R},\mathbf{r})$ be a Cantor set with $\mathbf{R}=(R_n)_{n\ge0}$ and $\mathbf{r}=(r_n)_{n\ge0}$: it is constructed by starting with the unit cylinder $\I$ and splitting each cylinder comprising a level set into $R_n$ sub-cylinders and throwing at most $r_n$ of them away. Let $\varepsilon>0$ be small. Then, as the value of $r_n/R_n$ is independent of $n$, it follows that there exists some $l\in\N$ sufficiently large and some $n'\in\N$ depending on $\varepsilon$ such that for all $n>n'$ \begin{equation}\label{cantor1}
R_n^{1-\varepsilon}\le t_n(l).
\end{equation} 

\noindent A probability measure $\zeta$ is defined on $\mathbf{K}(\mathbb{I},\mathbf{R},\mathbf{r})$ recursively: let \[\zeta(\I):=1\]
\noindent and for $J\in\mathcal{J}(Y_n),$ $J'\in\mathcal{J}(Y_{n-1})$ such that $J\subset J'$ let \begin{equation}\label{Cantor2}
    \zeta(J):=\frac{\zeta(J')}{\#\{\widetilde J\in\mathcal{J}(Y_n):\widetilde J\subset J'\}},
\end{equation}
\noindent It is shown in \cite[Proposition 1.7]{mass} that $\zeta$ can be extended to all Borel sets $F$ of $\F_q\left(\!\left(t^{-1}\right)\!\right)$ by setting \[\zeta(F):=\zeta(F\cap \mathbf{K}(\mathbb{I},\mathbf{R},\mathbf{r})) = \inf \left\{\sum_{J\in\mathcal{J}} \zeta(J)\right\},\] where the infimum is over all coverings $\mathcal{J}$ of $F\cap\mathbf{K}(\mathbb{I},\mathbf{R},\mathbf{r})$ by cylinders $J\in\bigcup_{n\ge0}\mathcal{J}(Y_n).$ For any pair of cylinders $J\in\mathcal{J}(Y_n)$ and $J'\in\mathcal{J}(Y_{n-1})$ such that $J\subset J'$, equation (\ref{Cantor2}) and the definition of $t_n(l)$ in (\ref{tn}) imply that \begin{equation*}
    \zeta(J)\le t_{n-1}^{-1}\zeta(J').
\end{equation*}
\noindent Inductively, this yields \begin{equation}\label{Cantor3}
    \zeta(J)\le\prod_{i=0}^{n-1}t^{-1}_i.
\end{equation}
 \noindent Next, let $\delta_n$ denote the Haar measure of a cylinder $J\in\mathcal{J}(Y_n)$. It is clear that \begin{equation}\label{Cantor4}
     \delta_n=\prod_{i=0}^{n-1}R_i^{-1}.
 \end{equation} 
 \noindent Let $C\subset \mathbb{I}$ be an arbitrary ball in $\I$. There exists $j\in\N$ such that \begin{equation}
     \delta_{j+1}<\diam(C)\le \delta_j.\label{delta_C}
 \end{equation}Therefore, one has that \begin{align}
 \zeta(C)&\le \sum_{\substack{J\in\mathcal{J}(Y_{j+1})\\J\cap C\neq\emptyset}} \zeta(J)~~~\substack{(\ref{Cantor3})\\\le }~~~\left(\frac{\diam(C)}{\delta_{j+1}}\right)\left( \prod_{i=0}^j t^{-1}_i\right)\\&\substack{(\ref{Cantor4})\\=}~~~ \diam(C)^\varepsilon \left(\prod_{i=0}^j \frac{R_i}{t_i}\right)\cdot \diam(C)^{1-\varepsilon}\nonumber
\\ &\substack{(\ref{delta_C})\\<}~~~ \left(\delta_{j}\right)^\varepsilon\cdot  \left(\prod_{i=0}^j \frac{R_i}{t_i}\right)\cdot \diam(C)^{1-\varepsilon}
 ~~~\nonumber\\&\substack{(\ref{Cantor4})\\= }~~~ \frac{R_j}{t_j} \cdot\left(\prod_{i=0}^{j-1} \frac{R_i^{1-\varepsilon}}{t_i} \right)\cdot \diam(C)^{1-\varepsilon}.\label{Cantor 6}
 \end{align}
Using that $d:=\frac{R_j}{t_j}$ is a constant depending only on the fixed number $l$, 
 \begin{align}
 (\ref{Cantor 6})~~~
 < ~~~d\cdot \prod_{i=0}^{j-1} \frac{R_i^{1-\varepsilon}}{t_i} \cdot \diam(C)^{1-\varepsilon}.\label{cantor_fin}
 \end{align}
\noindent If $j>n'$, equation (\ref{cantor1}) implies that 
 \begin{equation*}
     (\ref{cantor_fin})<d\cdot \prod_{i=0}^{n'} \frac{R_i^{1-\varepsilon}}{t_i} \cdot \diam(C)^{1-\varepsilon}.
 \end{equation*}
 \noindent Hence, by the Mass Distribution Principle (Lemma \ref{MDP}), $\mathbf{K}(\mathbb{I},\mathbf{R},\mathbf{r})$ has dimension greater than or equal to $1-\varepsilon$ for any $\varepsilon>0$. Therefore, $\mathbf{K}(\mathbb{I},\mathbf{R},\mathbf{r})$ has dimension 1. This completes the proof. 
 \end{proof}

\section{A Khintchine-Type Result in $P(t)$-adic Approximation}\label{Sect:Khint}
\noindent This section contains the proof of Theorem \ref{metric}. This is a corollary of the following theorem, proved using the combinatorial statements stated in Section \ref{Sect:Lemmata}. \begin{theorem}
   \label{metric2} Let $f:\{q^k\}_{k\ge0}\cup\{0\}\to\mathbb{R}_{>0}$ be a function and $t\in\F_q[t]$ be an irreducible polynomial. For $\mu$-almost every Laurent series $\Theta\in\mathbb{I}$, the inequality \begin{equation}
        f(|N(t)|)\cdot|N(t)|\cdot \left|\left\langle N(t)\cdot\Theta(t)\right\rangle\right|\cdot |N(t)|_{t}\le \frac{1}{q}\label{metinq}
    \end{equation} has infinitely (finitely, respectively) many solutions if \begin{equation}
        \sum_{k\ge0}\frac{k}{f\left(q^k\right)}\label{metsum2}
    \end{equation}
    diverges (converges, respectively) and, in the divergence case, $f$ is non-decreasing.
\end{theorem}\noindent Theorem \ref{metric2} implies Theorem \ref{metric}.\begin{proof}[Proof of Theorem \ref{metric} modulo Theorem \ref{metric2}]
    If the sum (\ref{metsum2}) converges, the proof is clear. For the divergence case, let \begin{equation}
    A_f:=\{\Theta(t)\in\I: \text{there exist infinite many solutions } N(t)\text{ to equation (\ref{metinq})}\}.\label{A_f}
\end{equation} Next, note that multiplying the function $f$ by $q^i$ for any fixed $i\in\N$ does not effect the convergence of (\ref{metsum2}). That is, if $A_f$ has full measure then so does $A_{q^if}$ for all $i\in\N$. Next, the set $W_q(t,f)$ (defined in equation (\ref{W_q(P(t),f)})) is expressed as \[W_q(t,f)=\bigcap_{i=0}^\infty A_{q^{i}f}.\] As this is a countable intersection of full measure sets, $W_q(t,f)$ has full measure.
\end{proof}

\subsection{Preliminary Lemmata}
\noindent Before the proof of Theorem \ref{metric2} begins in earnest, one further lemma is required. 
\subsubsection*{A Zero-One Law for Diophantine Approximation over Function Fields}
The first reduces the task of proving Theorem \ref{metric2} to only showing that $W_q(t,f)$ has positive measure.
\begin{lemma}[Zero-One Law - Inoue and Nakada, \cite{DSFF}]\label{fullmeas}
   Let $\phi:\F_q[t]\to\R_{>0}$ be a function and let\[E_{N(t)}=\left\{\Theta(t)\in\mathbb{I}: \left|\langle N(t)\cdot \Theta(t)\rangle\right|\le \phi(N(t))\right\}.\] Then $\mu(\limsup_{|N(t)|\to\infty}E_{N(t)})\in\{0,1\}$.
\end{lemma}

\noindent The proof of the above lemma is a direct adaptation of the proof in the real case, provided by Gallagher in \cite{zeroone}.

\subsection{Proof of Theorem \ref{metric2}: The Convergent Case}
\begin{proof}[Proof of Theorem \ref{metric2}]

\noindent Similarly to the set $A_f$ (introduced in equation (\ref{A_f})), define the set \begin{equation}\label{AN}
    A(N(t),t):=\left\{\Theta(t)\in\I:  N(t)\cdot f(|N(t)|)\cdot|N(t)|_{P(t)}\cdot \left|\left\langle|N(t)|\cdot\Theta(t)\right\rangle\right|\le q^{-1}\right\}.
\end{equation}
\noindent Note that $A(N(t),t)$ depends also on $f$, but this is dropped from the notation. \\

\noindent Recall the notation of a ball in the field of Laurent series, introduced in (\ref{ball}). Given a polynomial $N(t)\in\F_q[t]$, let $d(N):=\deg(N(t))$. Then, the set $A(N(t),t)$ is decomposed as follows: \begin{align}
    A(N(t),t)&=\bigcup_{R(t)\in(\mathbb{F}_q[t])_{d(N)}} \left\{\Theta(t)\in\I: \left|N(t)\cdot\Theta(t)-R(t)\right|\le\frac{q^{-1}}{f(|N(t)|)\cdot |N(t)|\cdot |N(t)|_{t}}\right\}\nonumber\\
    &= \bigcup_{R(t)\in(\mathbb{F}_q[t])_{d(N)}} B\left(R(t)N(t)^{-1}, \frac{q^{-1}}{f(|N(t)|)\cdot|N(t)|^2\cdot |N(t)|_{t}}\right).\label{An}
\end{align} As there are $q^{d(N)+1}$ such polynomials $R(t)$, the measure of $A(N(t),t)$ is then calculated as\begin{equation*}
    \mu(A(N(t),t))\le q^{d(N)+1}\frac{q^{-1}}{f(|N(t)|)\cdot|N(t)|^2\cdot |N(t)|_{t}}=\frac{1}{f(|N(t)|)\cdot|N(t)|\cdot |N(t)|_{t}}\cdotp
\end{equation*}\noindent Let $M(t)\in\F_q[t]$ be coprime to $t$ and let $k\in\N$ such that $N(t)=M(t)\cdot t^k$. Then,
\begin{align}
\sum_{N(t)\in\F_q[t]\backslash\{0\}}\mu(A(N(t),t)) &\le \sum_{M(t)\in\F_q[t]\backslash\{0\}}\sum_{k\ge0} \frac{1}{f(|M(t)\cdot t^k|)\cdot|M(t)|}\nonumber\\
&=\sum_{d\ge0}~\sum_{\substack{M(t)\in\F_q[t]\backslash\{0\}\\\deg(M(t))=d}}~\sum_{k\ge0} \frac{1}{f(|M(t)\cdot t^k|)\cdot|M(t)|}\cdotp\label{BC_conv1}
\end{align}
\noindent As there are $(q-1)q^{d(M)}$ polynomials $M(t)\in\F_q[t]$ of degree $d$, \begin{align}
(\ref{BC_conv1})&=\sum_{d\ge0}\sum_{k\ge0} \frac{q-1}{f(q^{d(M)+k})}\cdotp\label{BC_conv2}
\end{align}
\noindent Finally, let $i:=d(M)+k$. For each value of $i\in\N$, there are $ i +1$ combinations of different values for $k$ and $d(M)$ that attain it. Hence, \begin{equation*}
    (\ref{BC_conv2})=\sum_{i\ge0} \left(i+1\right)\cdot \frac{q-1}{f(q^i)} \ll_{P(t),q}\sum_{i\ge0}\frac{i}{f(q^i)}\cdotp
\end{equation*}Applying the Borel-Cantelli lemma proves the convergence case. 
\end{proof}

\subsection{Proof of Theorem \ref{metric2}: The Divergent case. }  

\begin{proof}The proof is split into distinct sections, each of which builds on the previous and is focussed on a specific idea. To begin, the problem is rephrased in terms of number walls.

\subsubsection*{Square Portions Represented by Polynomials}

\noindent Let $N(t)\in\F_q[t]$ be a polynomial that is decomposed as $N(t)=M(t)\cdot t^k$ for $(M(t),t)=1$. Furthermore, let $d(M):= \deg(M(t))$. If $\Theta(t)\in A(N(t),t)$ (as defined in equation (\ref{AN})), Theorem \ref{growth} (applied with $l=1$) implies there is a window in $W_q(\Theta(t))$ that contains the square portion \[\Box(N(t)):=\Box\left(\left\lfloor\log_q(f(q^{d(M)+k}))\right\rfloor,d(M),d(M)+k+1\right).\] By Corollary \ref{contain full}, for $r$ a large natural number, there are $q^{r-\left\lfloor\log_q(f(q^{d(M)+k}))\right\rfloor}$ possible sequences of length $r$ over $\F_q$ whose number wall has a (complete or incomplete) window containing $\Box(N(t))$. However, it is clear that $\Box(N(t))$ only depends on $k$ and $d(M)$, and not on the specific choice of $M(t)$. For $d(M),k\in\N$, define then the set \[A_{d(M),k}:=\left\{\Theta(t)\in A(\widetilde N(t),t): \widetilde N(t)=\widetilde M(t)\cdot t^k, ~\deg(\widetilde M(t))=d(M), \text{ and } (\widetilde M(t),t)=1\right\}.\]
Explicitly, $A_{d(M),k}$ is the union of all the sets $A(N(t),t)$ for which can expressed as $N(t)=M(t)\cdot t^k$ for some polynomial $M(t)$ of degree $d(M)$. In particular, $A_{d(M),k}$ is the union of finitely many sets $A({N(t)},t)$. Therefore, if $\Theta(t)$ is in the set $A({N(t)},t)$ for infinitely many $N(t)\in\F_q[t]$, then $\Theta(t)$ is also in infinitely many sets $A_{d(M),k}$. The converse is also clear, implying \[\limsup_{|N(t)|\to\infty}A({N(t)},t)=\limsup_{\max(d(M),k)\to\infty}A_{d(M),k}.\]
Similarly as above, define $\Box(d(M),k)$ as the unique square portion represented by any $N(t)=M(t)\cdot t^k$, where $\deg(M(t))=d(M)$ and $M(t)$ and $t$ are coprime. \\

\noindent The proof now moves on to finding the measure of the above limsup sets.

\subsubsection*{The Divergence Borel-Cantelli Lemma}

\noindent The following classical lemma is used to attain positive measure of $W_q(t,f)$ from the divergence of the series (\ref{metsum2}). The reader is referred to \cite{Divbc} for the proof and further reading.

\begin{lemma}[Divergence Borel-Cantelli]\label{divbc} Let $\{E_n\}_{n\in \N}$ be a sequence measureable sets in a space with measure $\mu$. Define $E_\infty:=\limsup_{n\to\infty}E_n$. Assume $\sum_{n\ge0}\mu(E_n)=\infty$ and that there exists a constant $C>0$ such that the inequality \[\sum_{s,t\le r}\mu(E_s\cap E_t)\le C\left(\sum_{n\le r}\mu(E_n)\right)^2\] holds for infinitely many $r\in\N$. Then $\mu(E_\infty)\ge1/C$.\end{lemma}

\noindent This lemma is now used to show that $A_\infty:=\limsup_{\min(d(M),k)\to\infty}A_{d(M),k}$ has positive measure. The sum \[\sum_{d(M_1)\le r}~\sum_{k_1\le r-d(M_1)}~\sum_{d(M_2)\le r}~\sum_{k_2\le r-d(M_2)}\mu(A_{d(M_1),k_1}\cap A_{d(M_2),k_2})\] is partitioned into two pieces: \begin{align}
    S_1&=\sum_{d(M_1)\le r}~\sum_{k_1\le r-d(M_1)}~\sum_{d(M_2)\le r}~\sum_{\substack{k_2\le r-d(M_2)\\\Box(d(M_1),k_1)\text{ intersects }\Box(d(M_2),k_2)}}\mu(A_{d(M_1),k_1}\cap A_{d(M_2),k_2}),\nonumber\end{align}
    \noindent and\begin{align}
    S_2&=\sum_{d(M_1)\le r}~\sum_{k_1\le r-d(M_1)}~\sum_{d(M_2)\le r}~\sum_{\substack{k_2\le r-d(M_2)\\\Box(d(M_1),k_1)\text{ does not intersect }\Box(d(M_2),k_2)}}\mu(A_{d(M_1),k_1}\cap A_{d(M_2),k_2}).\label{S2A}
\end{align}

\subsubsection*{The Sum $S_1$}

\noindent The sum $S_1$ is dealt with first. The goal is to show that \[S_1\ll \left(\sum_{d(M)\le r}\sum_{k\le r-h}\mu(A_{d(M),k})\right)^2.\]To achieve this, fix $d(M_1)$ and $k_1$, then consider all the values of $d(M_2)$ and $k_2$ such that $\Box(d(M_1),k_1)$ intersects $\Box(d(M_2),k_2)$. As the size of $\Box(d(M),k)$ is determined only by $d(M),k$ and the function $f$, there are only a finite number of square portions $\Box(d(M_2),k_2)$ intersecting $\Box(d(M_1),k_1)$. These values of $M_2$ and $k_2$ are partitioned into disjoint sets $\{P_i\}_{i\ge0}$ depending on how much larger than $\Box(d(M_1),k_1)$ a square portion must be to contain $\Box(d(M_1),k_1)$ and $\Box(d(M_2),k_2)$. Explicitly, $P_i$ contains all the pairs $(d(M_2),k_2)$ such that the minimum size of any square portion containing $\Box(d(M_1),k_1)$ and $\Box(d(M_2),k_2)$ is $i$ more than $\Box(d(M_1),k_1)$. For example, $P_0$ contains only the pairs $(d(M_2),k_2)$ such that $\Box(d(M_2),k_2)=\Box(d(M_1),k_1)$ (that is, $d(M_2)=d(M_1)$ and $k_2=k_1$). Similarly, $P_1$ contains all the pairs $(d(M_1),k_1)$ such that the square portion containing $\Box(d(M_1),k_1)$ and $\Box(d(M_2),k_2)$ is only one larger than the size of $\Box(d(M_1),k_1)$. Below, the diagram used in the calculation of the measure of $P_1$: \begin{figure}[H]
    \centering
    \includegraphics[width=0.75\linewidth]{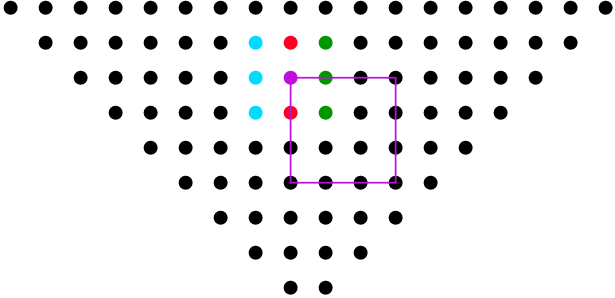}
    \caption{The purple square represents $\Box(d(M_1),k_1)$ and the coloured dots represent the pairs $(d(M),k)\in P_1$.}
    \label{div Pic1}
\end{figure} \noindent Theorem \ref{growth} states that the size of the square portion $\Box(d(M),k)$ is equal to $\left\lfloor\log_q(f(q^{d(M)+k}))\right\rfloor$, and that $d(M)+k+1$ is the column index of the top left entry of the square portion $\Box(d(M),k)$. Therefore, a square portion $\Box(d(M_2),k_2)$ with top left corner in the red positions would have exactly the same size as $\Box(d(M_1),k_1)$. However, a square portion $\Box(d(M_2),k_2)$ with top left corner in the blue (green, respectively) positions would have a size less than or equal to (greater than or equal to, respectively) the size of $\Box(d(M_1),k_1)$. Each starting position determines the shape of the minimal rectangular portion of the number wall that contains both $\Box(d(M_1),k_1)$ and $\Box(d(M_2),k_2)$. \\

\noindent For example, if $d(M_2)$ and $k_2$ are such that $\Box(d(M_2),k_2)$ starts in the top left blue position in Figure \ref{div Pic1}, then $\Box(d(M_2),k_2)$ has size less than or equal to $\Box(d(M_1),k_1)$ and hence $\Box(d(M_1),k_1)$ and $\Box(d(M_2),k_2)$ are contained in a single square portion of length $\left\lfloor\log_q(f(q^{d(M)+k}))\right\rfloor+1$ with top left corner in the blue position of Figure \ref{div Pic1}. One now needs to calculate $\mu(A_{d(M_1),k_1}\cap A_{d(M_2),k_2})$. To this end, let $\Box$ be a square portion of size $l\in\N$. Using Corollary \ref{contain full} and the ultrametric property of the absolute value gives \[\mu(\{\Theta(t)\in\F_q\!\left(\!\left(t^{-1}\right)\!\right): W_q(\Theta(t))\text{ has a window containing }\Box\})=q^{r-l}\cdot q^r=q^{-l}.\] \noindent Hence, in this case, $\mu(A_{d(M_1),k_1}\cap A_{d(M_2),k_2})=q^{-l_1-1}.$\\

\noindent Similarly, if $d(M_2)$ and $k_2$ are such that $\Box(d(M_2),k_2)$ starts in the top middle red position, then $\Box(d(M_1),k_1)$ and $\Box(d(M_2),k_2)$ are contained inside a rectangle with height one greater than its width. In this case, Proposition \ref{bound} implies that
$\mu(A_{d(M_1),k_1}\cap A_{d(M_1)-1,k_1})\ll q^{-l_1-1}.$ \\

\noindent Doing a similar calculation for the other six pairs $(d(M_2),k_2)\in P_1$ shows that \begin{align}\sum_{(d(M_2),k_2)\in P_1}\mu(A_{d(M_1),k_1}\cap A_{d(M_2),k_2})\le8\cdot\mu(A_{d(M_1),k_1}\cap A_{d(M_1)-1,k_1})\ll q^{-l_1-1}.\nonumber\end{align}
\noindent Similarly, the following diagram depicts starting positions for $\Box(d(M_2),k_2)$ when $(d(M_2),k_2)\in P_i$. \begin{figure}[H]
    \centering
    \includegraphics[width=1\linewidth]{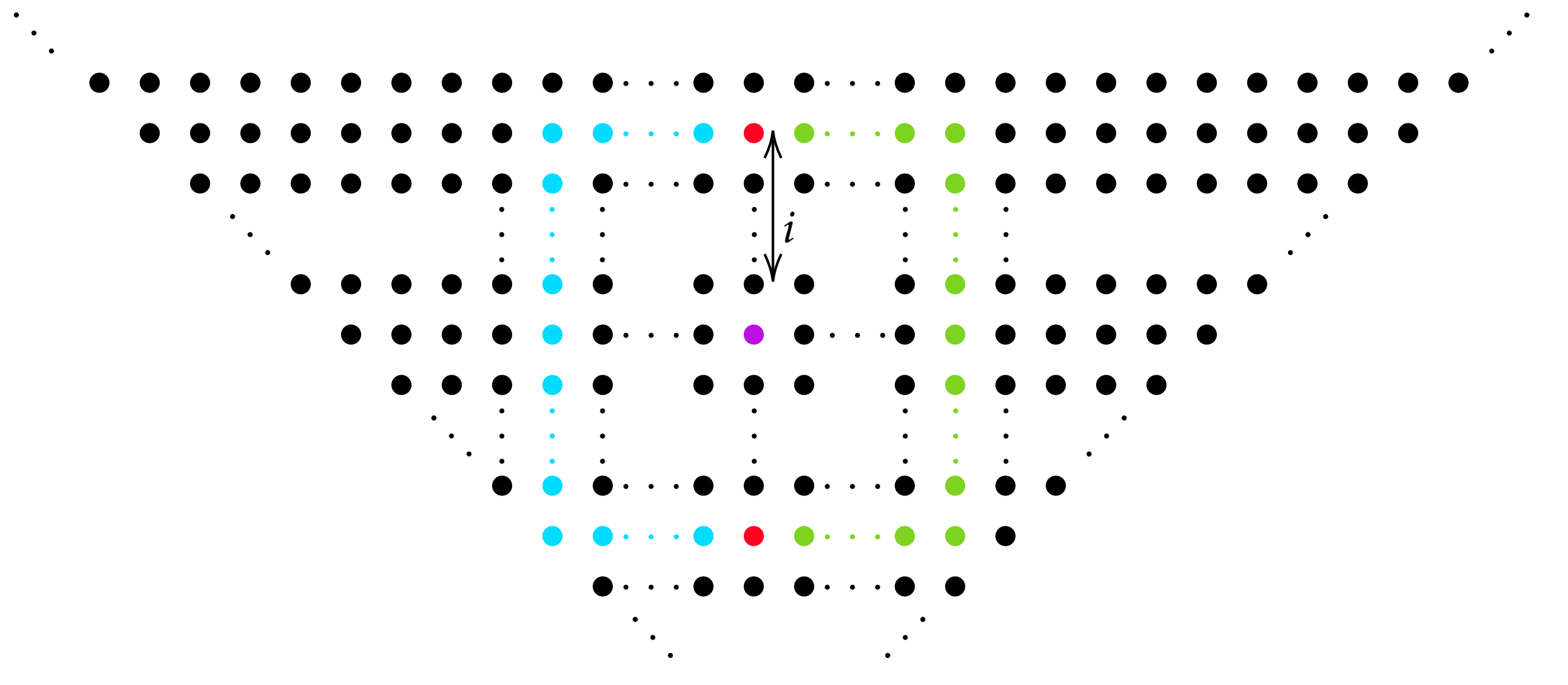}
    \caption{The coloured dots (with the exception of the purple dot) represent the pairs in $P_i$.}
    \label{div pic2}
\end{figure} \noindent As $P_i$ is the boundary of a square where each side is $2i+1$ entries long, there are $8i$ pairs $(d(M_2),k_2)\in P_i$. If $d(M_2)$ and $k_2$ are such that $\Box(d(M_2),k_2)$ starts on the top middle red position, then $\Box(d(M_1),k_1)$ and $\Box(d(M_2),k_2)$ are contained in a rectangle with height $i$ more than its width. Every other pair $(h-i,k-i)\in P_i$ has $\Box(d(M_1),k_1)$ and $\Box(d(M_2),k_2)$ contained in a rectangle with difference between width and height at most $i$. Once again, combining this with Proposition \ref{bound} implies that \begin{align}
    \sum_{(d(M_2),k_2)\in P_i}\mu(A_{d(M_1),k_1}\cap A_{d(M_2),k_2})\le8i\cdot\mu(A_{d(M_1),k_1}\cap A_{d(M_1)-i,k_1})
    \ll 8i \cdot (i+1)\cdot q^{-l_1-i}. \nonumber
\end{align} 
\noindent Depending on the function $f$, the size of the window $\Box(d(M_2),k_2)$ could be different from $\Box(d(M_1),k_1)$. This leads to some pairs $(d(M_2),k_2)\in P_i$ being replaced with $(d(M_2)',k_2')$ representing the starting positions being moved up and to the left. In some cases, pairs in $P_i$ are ignored entirely. Moving a starting position does not effect the calculation and it suffices to assume none are ignored when calculating an upper bound. Furthermore, when $i=l_1+1,$ the sets $P_i$ become empty, as the windows $\Box(d(M_1),k_1)$ and $\Box(d(M_2),k_2)$ no longer intersect. Hence, \begin{align}
    \sum_{d(M_2)\le r}\sum_{\substack{k_2\le r-d(M_2)\\\Box(d(M_1),k_1)\text{ intersects }\Box(d(M_2),k_2)}}\mu(A_{d(M_1),k_1}\cap A_{d(M_2),k_2})&\ll_q q^{-l_1}\cdot\left(1+8\cdot\sum_{i=1}^\infty i^2q^{-i}\right)\nonumber\\
    &=\mu(A_{d(M_1),k_1})\cdot\left(1+8\cdot\sum_{i=1}^\infty i^2q^{-i}\right). \label{divsum1}
\end{align}
\noindent Therefore, from the convergence of the sum in equation, (\ref{divsum1})\begin{align*}
    S_1 \ll_q \sum_{d(M_1)+k_1\le r}\mu(A_{d(M_1),k_1}) \ll_q \left(\sum_{d(M_1)+k_1\le r}\mu(A_{d(M_1),k_1})\right)^2,
\end{align*}with the second inequality using the assumption that $\sum_{d(M_1)\ge0}\sum_{k_1\ge0}\mu(A_{d(M_1),k_1})=\infty,$ which implies there exists some value $r$ such that $\sum_{d(M_1)+k_1\le r}\mu(A_{d(M_1),k_1})>1$. 
\subsubsection*{The Sum $S_2$}
\noindent Next, $S_2$, as defined in (\ref{S2A}), is dealt with.  If $\Box(d(M_1),k_1)$ does not intersect with $\Box(d(M_2),k_2)$, then $\mu(A_{d(M_1),k_1}\cap A_{d(M_2),k_2})$ is given by the number of sequences that have windows containing $\Box(d(M_1),k_1)$ and $\Box(d(M_2),k_2)$ multiplied by the measure of the ball centred around them. By Theorem \ref{twowind}, this is\[\mu(A_{d(M_1),k_1}\cap A_{d(M_2),k_2}) \ll q^{r-l_1-l_2}\cdot q^{-r}=q^{-l_1-l_2}.\]Hence, \begin{align}
    S_2&\ll \sum_{d(M_1)+k_1\le r}\sum_{\substack{d(M_2)+k_2\le r\\\Box(d(M_1),k_1)\text{ does not intersect }\Box(d(M_2),k_2)}} q^{-\left\lfloor\log_q(f(q^{d(M_1)+k_1})\right\rfloor-\left\lfloor\log_q(f(q^{d(M_2)+k_2})\right\rfloor}.\label{divinq1}\end{align}
\noindent To acquire an upper bound on this sum, the condition that $\Box(d(M_1),k_1)$ and $\Box(d(M_2),k_2)$ do not intersect can be dropped:
    \begin{align*}(\ref{divinq1})&\le\sum_{d(M_1)+k_1\le r}\sum_{d(M_2)+k_2\le r} q^{-\left\lfloor\log_q(f(q^{d(M_1)+k_1})\right\rfloor-\left\lfloor\log_q(f(q^{d(M_2)+k_2})\right\rfloor}\\
    &= \left(\sum_{d(M_1)+k_1\le r}q^{-\left\lfloor\log_q(f(q^{d(M_1)+k_1})\right\rfloor}\right)\cdot \left(\sum_{d(M_2)+k_2\le r}q^{\left\lfloor\log_q(f(q^{d(M_2)+k_2})\right\rfloor}\right)\\
    &=\left(\sum_{h+k\le r}\mu(A_{h,k})\right)^2.
\end{align*} Hence, by the Divergence Borel-Cantelli lemma (Lemma \ref{divbc}), the set $A_\infty$ has positive measure. By applying Lemma \ref{fullmeas} with the sets $A_{N(t)}$ from (\ref{AN}) in place of the sets $E_N$, one obtains that $A_\infty$ has full measure, completing the proof.
\end{proof}
\section{Combinatorics on Number Walls}\label{Sect: proofs}
\noindent This section is dedicated to proving all the lemmata stated in Section \ref{Sect:Lemmata}, with each lemma being given its own subsection. 
\subsection{The Free-Determined Lemma}\label{7.1}
This section is dedicated to proving Proposition \ref{wind_extend_lem}.
\noindent The following definitions revolve around extending a finite sequence by one entry and categorising the changes in its finite number wall.
\begin{definition}\label{determined} Let $\mathbf{S}=(s_i)_{1\le i \le r}$ be a finite sequence of length $r$ over a field $\F_q$, and let $W_q(\mathbf{S})=(W_{m,n}(\mathbf{S}))_{m,n\in\Z}$ be the finite number wall generated by $\mathbf{S}$. Assume that $s_{r+1}$ is to be added to the end of the sequence $\mathbf{S}$. Then:\begin{itemize}
    \item An entry on the $(r+1)^\nth$ diagonal of the number wall $W_q( \mathbf{S}\oplus\{s_{r+1}\})$ is \textbf{determined} if it is independent of the choice of $s_{r+1}$.
    \item An entry $W_{i,r+1-i}( \mathbf{S}\oplus\{s_{r+1}\})$ of the number wall on the $(r+1)^\nth$ diagonal is \textbf{free} if for any $x\in\mathbb{F}_q$, there exists a value of $s_{r+1}$ making $W_{i,r+1-i}( \mathbf{S}\oplus\{s_{r+1}\})=x$.
\end{itemize}
\end{definition}
\noindent Given a finite number wall generated by a sequence of length $r$, the following lemma describes how many degrees of freedom there are for the values of the $(r+1)^\nth$ diagonal.
\begin{lemma}[Free-Determined Lemma]\label{detlem}
Let $W_q(\mathbf{S})=(W_{m,n}(\mathbf{S}))_{n,m\in\mathbb{Z}}$ be the finite number wall generated by the sequence $\mathbf{S}=(s_i)_{0\le i \le r}$ over a field $\F_q$. Assume that $s_{r+1}\in\F_q$ is added to the end of $\mathbf{S}$. Then, for $0\le i \le\left \lfloor\frac{r}{2}\right\rfloor$, the value of $W_{i,r+1-i}( \mathbf{S}\oplus\{s_{r+1}\})$ is determined if and only if $W_{i-1,r-i}(\mathbf{S})=0$. Furthermore, if $W_{i,r+1-i}( \mathbf{S}\oplus\{s_{r+1}\})$ is not determined, then it is free. Also, picking a value for any non-determined entry on the $(r+1)^\nth$ diagonal uniquely decides the value for every remaining non-determined entry on the $(r+1)^\nth$ diagonal. 
\end{lemma}
\noindent Lemma \ref{detlem} is illustrated below. \begin{figure}[H]
    \centering
    \includegraphics[width=0.75\linewidth]{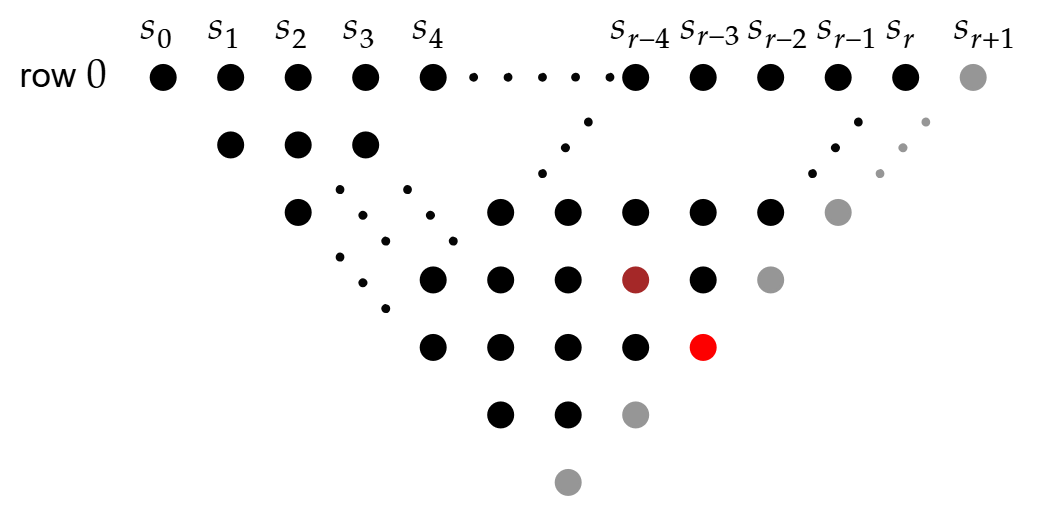}
    \caption{The number wall generated by $ \mathbf{S}\oplus\{s_{r+1}\}$. Each entry on row zero has the corresponding entry of $\mathbf{S}$ written above it. The $(r+1)^\nth$ diagonal is in grey. Lemma \ref{detlem} states that the bright red entry is determined if and only if the dark red entry is zero, and it is free otherwise.}
    \label{fig:enter-label}
\end{figure}
\begin{proof}
\noindent The value of $W_{i,r+1-i}(\mathbf{S})$ is calculated from the definition of a number wall: \begin{equation}
W_{i,r+1-i}(\mathbf{S})=\det(T_{\mathbf{S}}(i,r+1-i))=\det\begin{pmatrix}s_{r+1-i}&\dots&s_{r+1}\\
\vdots&\ddots&\vdots\\
s_{r+1-2i}&\dots&s_{r+1-i}\end{pmatrix}.\nonumber
\end{equation}
\noindent Due to the Toeplitz structure, $s_{r+1}$ only occurs in the top right corner of $T_{\mathbf{S}}(i,r+1-i)$. By expanding along the top row, \begin{equation}\label{ID_div}
W_{i,r+1-i}(\mathbf{S})=s_{r+1}\cdot \det(T_{\mathbf{S}}(i-1,r-i))+Y,
\end{equation}
\noindent where $Y$ is a number that depends only on $\{s_i:1\le i \le r\}$. Hence, the value of $W_{i,r+1-i}(\mathbf{S})$ is independent of $s_{r+1}$ if and only if $\det(T_{\mathbf{S}}(i-1,r-i))=W_{i-1,r-i}(\mathbf{S})=0$. If $W_{i-1,r-i}(\mathbf{S})\neq0$, then for any chosen value of $W_{i,r+1-i}(\mathbf{S})$ there exists a unique choice of $s_{r+1}$ that achieves it. Namely, $s_{r+1}=(W_{i,r+1-i}(\mathbf{S})-Y)/\det(T_{\mathbf{S}}(i-1,r-i))$. Note that this division is well defined by equation (\ref{ID_div}). As the value of $s_{r+1}$ has now been fixed, the value of every non-determined entry on diagonal $r+1$ is also fixed.
\end{proof}
\noindent The Free-Determined Lemma (Lemma \ref{detlem}) has the following simple corollary:

\begin{cor} \label{diag}
Let $\mathbf{S}=(s_i)_{0\le i \le r}$ be a finite sequence over a field $\F_q$ and let $s_{r+1}', s_{r_1}''\in\F_q$. Furthermore, define the sequences $ \mathbf{S}'= \mathbf{S}\oplus\{s_{r+1}'\}$ and $ \mathbf{S}''= \mathbf{S}\oplus\{s_{r+1}''\}$. If $0\le j\le\left\lfloor\frac{r}{2}\right\rfloor$ is such that $W_{j,r+1-j}( \mathbf{S}')$ and $W_{j,r+1-j}( \mathbf{S}'')$ are not determined, then $W_{j,r+1-j}( \mathbf{S}')=W_{j,r+1-j}( \mathbf{S}'')$ if and only if $s_{r+1}'=s_{r+1}''$. 
\end{cor}

\begin{proof}
\noindent Clearly, if $s_{r+1}'=s_{r+1}''$, then the entire $(r+1)^\nth$ diagonal is identical. Assume conversely that $W_{j,r+1-j}( \mathbf{S}')= W_{j,r+1-j}( \mathbf{S}'')$ but that one has not decided on the values of $s_{r+1}'$ and $s_{r+1}''$. Applying Lemma \ref{detlem}, shows that this also fixes the value for every non-determined entry in the diagonal. Furthermore, by the definition of the number wall the values on the row zero cannot be determined (as every entry in row $-1$ is equal to 1), which completes the proof. 
\end{proof}

\begin{remark}
    That $t$-LC (and hence $P(t)$-LC from Theorem \ref{mainresult}) is false when $\F_q$ is infinite (originally established by de Mathan and Teulié in \cite{padic}) is an immediate consequence of Lemma \ref{detlem}. Indeed, if the sequence $(s_i)_{1\le i \le r}$ over $\F_q$ has no zero entries in its finite number wall, then every entry on the $(r+1)^\nth$ diagonal is free. Now, Lemma \ref{detlem} shows there are infinite choices for $s_{r+1}$ that do not create a zero entry on the $(r+1)^\nth$ diagonal. A simple induction beginning with any nonzero sequence of length one generates a counterexample to $t$-LC. 
\end{remark}
\noindent Given an open window in a finite number wall, denoted $\mathcal{W}$, Corollary \ref{diag} implies there is a unique extension to the generating sequence that extends $\mathcal{W}$ into a window of a given size.
\subsubsection*{Proof of Proposition \ref{wind_extend_lem}}
\noindent Proposition \ref{wind_extend_lem} is a direct corollary of the Free-Determined Lemma (Lemma \ref{detlem}).
\begin{proof}[Proof of Proposition \ref{wind_extend_lem}]
    The proof is achieved by repeatedly applying Lemma \ref{detlem}. Indeed, let $W_q(\mathbf{S})=(W_{m,n}(\mathbf{S}))_{m,n\in\Z}$ and say $\mathcal{W}$ has its top left corner on row $m\in\N$ and column $n\in\N$. Then, $W_{m,n+l-1}(\mathbf{S})=0$ and this entry is on the $r^\nth$ diagonal. Furthermore, $W_{m-1,n+l-1}(\mathbf{S})$ is nonzero since it is in the north side of the inner frame of $\mathcal{W}$. Therefore, by Lemma \ref{detlem} $W_{m,n+l}(\mathbf{S})$ is free and there is a unique value of $s_{r+1}$ that makes $W_{m,n+l}(\mathbf{S})=0$. By the Square Window Theorem \ref{window}, $\mathcal{W}$ has been extended into a right-side open window of size $l+1$. This argument is repeated $l'-l$ times to complete the proof. 
\end{proof}

\subsection{The Blade of a Finite Number Wall}\label{Sect:Blade}

\noindent The goal of this section is to prove Lemma \ref{contain}. The first step towards this is to describe the possible formations of entries that can occur at the bottom of the defined part of a finite number wall. 
\subsubsection*{Left and Right-side Blades}
\begin{definition}\label{blade_def}
     Given a finite number wall, its \textbf{blade} is the values on the bottom two rows that contain defined entries. Explicitly, if the number wall is generated by a sequence $\mathbf{S}$ of length $2r+1$ for $r\in\N$, then the blade is the entries in rows $r$ and $r-1$ which are defined, and if the sequence has length $2r$ then the blade is comprised of the entries in row $r-1$ and $r-2$ which are defined. Furthermore, define the \textbf{right-side (left-side}, respectively)\textbf{ blade} as entries in the two right-most (left-most, respectively) columns of the blade. That is, the right-side blade is the entries of the blade of $W_q(\mathbf{S})$ in columns $r$ and $r+1$, regardless of whether $\mathbf{S}$ has length $2r$ or $2r+1$.
     \end{definition}
\noindent These definitions are illustrated below: \begin{figure}[H]
    \centering
    \includegraphics[width=0.75\linewidth]{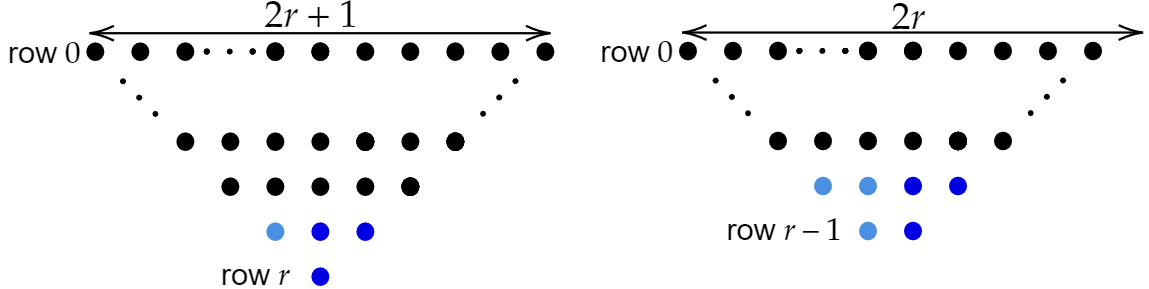}
    \caption{The blades of these number walls are coloured in blue, with the right-side blade in a darker shade. }
\end{figure}
 \noindent Note, that if the finite sequence has odd length the right-side blade and the left-side blade overlap. \\
\subsubsection*{The Symmetry of Finite Number Walls}
\noindent Studying the right-side blade is sufficient for most purposes, but all of the results in this subsection are also true for the left-side blade by symmetry from the following Lemma.

\begin{lemma}
    Let $W_q(\mathbf{S})$ be the finite number wall of the finite sequence $\mathbf{S}=(s_1,\dots,s_r)$ over a field $\F_q$. Define $ \mathbf{S}'=(s_r,\dots,s_1)$ as the `reflection' of $\mathbf{S}$. Then the number wall of $ \mathbf{S}'$ is the number wall of $\mathbf{S}$, reflected vertically. More precisely, let $m$ and $n$ be natural numbers satisfying $0\le m\le \left\lfloor\frac{r-1}{2}\right\rfloor$ and $0\le n\le r-1-2m$, and define as usual $W_{m,n}(\mathbf{S})$ and $W_{m,n}( \mathbf{S}')$ to be the entry on column $n$ and row $m$ of the number walls of $\mathbf{S}$ and $ \mathbf{S}'$, respectively. Then, $W_{m,r-m-n}( \mathbf{S}')=W_{m,m+n+1}(\mathbf{S})$.\label{reflect}
\end{lemma}

\noindent Lemma \ref{reflect} is illustrated by Figure \ref{reflectpic}:\begin{figure}[H]
    \centering
    \includegraphics[width=1\linewidth]{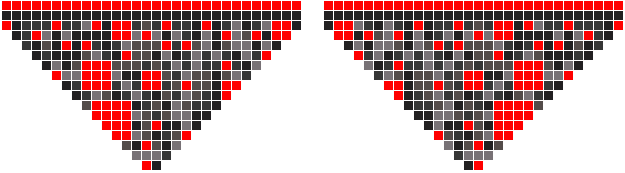}
    \caption{Left: The number wall of a sequence $\mathbf{S}$ over $\F_5$. Each entry of $\mathbf{S}$ is picked with equal probability assigned to every value of $\F_5$. Right: The number wall of $ \mathbf{S}'$, defined as in Lemma \ref{reflect}.}
    \label{reflectpic}
\end{figure}
\begin{proof}
    
    The entry of $W(\mathbf{S})$ in column $m+n+1$ and row $m$ is equal to the determinant of $T_{\mathbf{S}}(m,m+n+1)$. Similarly, the entry in column $r-m-n$ and row $m$ of $W( \mathbf{S}')$ is $T_{\mathbf{S}'}(m,r-m-n)$. The definition of $T_{\mathbf{S}'}(r-m-n,m)$ shows it is the transpose of $T_{\mathbf{S}}(m,m+n+1)$, completing the proof.
\end{proof}
\subsubsection*{The Profile of a Number Wall}
\noindent It is often only important that an entry in the number wall is zero or nonzero, with its exact value in the latter case being irrelevant, whence this definition:
\begin{definition}\label{profile_def}
    Let $W_q(\mathbf{S})$ be the number wall of a (finite or infinite) sequence $\mathbf{S}$ over a field $\F_q$. The \textbf{profile} of $W_q(\mathbf{S})$ is a two dimensional array with the same width and height as $W_q(\mathbf{S})$ where every non-zero entry has been replaced with the same symbol. In other words, the profile of a number wall shows only where its zero entries are located.
\begin{figure}[H]
    \centering
    \includegraphics[width=0.75\linewidth]{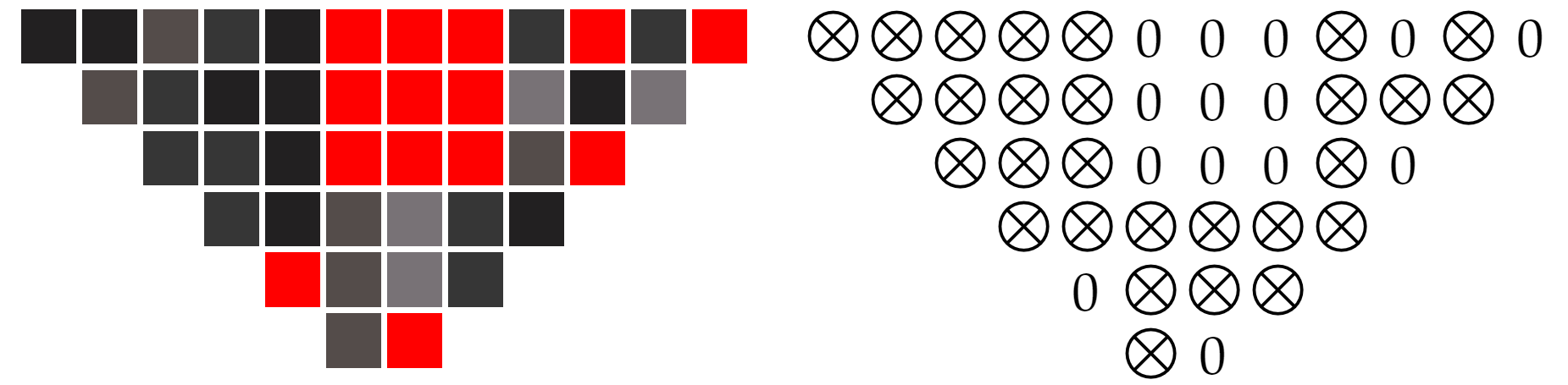}
    \caption{Left: The number wall $W(\mathbf{S})$ for the sequence $\mathbf{S}=(1,1,3,2,1,0,0,0,2,0,2,0)\in\F_5^{12}$. The windows are coloured in red. Right: The profile of $W(\mathbf{S})$.}
\end{figure}
\end{definition}
\noindent This definition allows for dot diagrams to distinguish between zero and nonzero entries without worrying about specific values. 
\subsubsection*{The Profiles of the Right-Side Blades}
\noindent There are seven possible profiles for the right-side blade, as the profile with a nonzero entry in the top left and zeros in the top right and bottom left is not possible due to the Square Window Theorem (Theorem \ref{window}). \begin{figure}[H]
    \centering
    \includegraphics[scale=0.3]{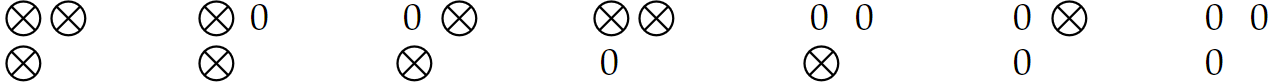}
    \caption{The seven possible right-side blade profiles. }
    \label{bladeshapes}
\end{figure}
\noindent In text, these right-side blade profiles are denoted by $\XXX,~\XXO,~\XOX,~\OXX,~\XOO,~\OOX,~\OOO.$ The right-side blade comprised of all zeros (furthest right in Figure \ref{bladeshapes}) is called the \textbf{zero right-side blade}. All other blades are described as \textbf{nonzero}. \\

\subsubsection*{Extending one Right-Side Blade into Another}

\noindent Let $\mathbf{S}=(s_i)_{1\le i \le r}$ be a finite sequence of length $r$ in $\F_q$ whose number wall has a nonzero right-side blade. There are $q^2$ different sequences $\{s_{r+1},s_{r+2}\}$, such that $ \mathbf{S}':= \mathbf{S}\oplus\{s_{r+1},s_{r+2}\}$ has length $r+2$. These $q^2$ continuations can be partitioned by the different blades they create in $W_q( \mathbf{S}')$. Indeed, let $\{u_i\}_{i=1,2,3}$ and $\{v_i\}_{i=1,2,3}$ be the entries in the right-side blade of the sequence $\mathbf{S}$ and $ \mathbf{S}'$ respectively.\vspace{-0.25cm} \begin{figure}[H]
    \centering
    \includegraphics[width=0.25\linewidth]{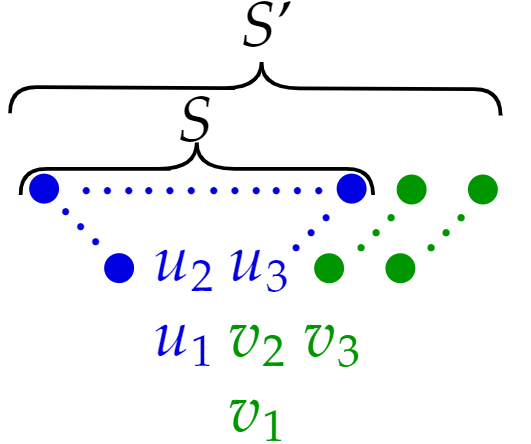}
    \caption{How $u_1,u_2,u_3,v_1,v_2$ and $v_3$ appear in the number wall of $ \mathbf{S}'$ when $r$ is odd.}
    \label{bp}
\end{figure}

\noindent Using Lemma \ref{detlem}, the following diagrams are created. Each shows a given right-side blade profile for $\mathbf{S}$ (in blue, representing $u_1,u_2$ and $u_3$) and all the possible right-side blade profiles for $ \mathbf{S}'$ (in green, representing $v_1, v_2$ and $v_3$), with the number of possible continuations from $\mathbf{S}$ to $ \mathbf{S}'$ giving this right-side blade written underneath. These are referred to as the \textbf{Tree Diagrams}. There are six in total, and two of them are explained in detail. The rest follow similarly. 
\begin{figure}[H]
    \centering
    \includegraphics[width=0.5\linewidth]{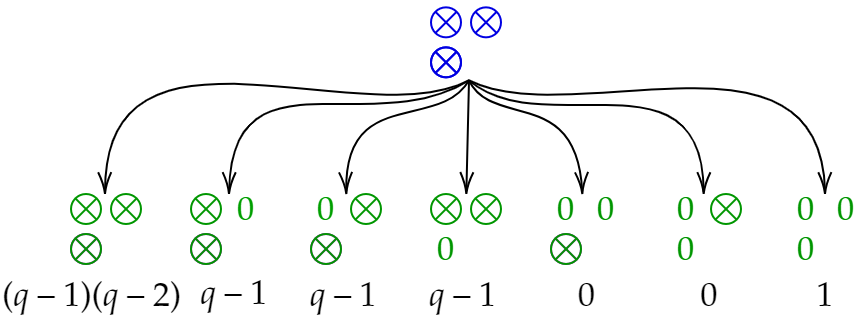}
    \caption{The $\hspace{-7pt}\substack{\p{X}XX\\X}$-Tree Diagram with the corresponding frequency of the possible extensions.}
    \label{XXXTD}
\end{figure}
\noindent In this case, using the notation from Figure \ref{bp}, all the values of $\{u_i\}_{i=1,2,3}$ are nonzero, and hence Lemma \ref{detlem} implies all the values $\{v_i\}_{i=1,2,3}$ are free. Lemma \ref{detlem} also implies there is only one choice of $s_{r+1}$ that would give $v_2=0$, and hence there are $q-1$ choices that give $v_2\neq0$. In the same way, there are unique choices of $s_{r+2}$ that give $v_1=0$ and $v_3=0$. The Square Window Theorem (Theorem \ref{window}) implies these choices coincide if and only if $v_2=0$, explaining why the right-most green right-side blade in Figure \ref{XXXTD} has only one way of occurring and hence why there are $q-1$ choices for the third right-side blade. Furthermore, if $v_2\neq0$ then it is impossible that $v_1=0=v_3$ by the Square Window Theorem (Theorem \ref{window}). Hence, there are $(q-2)$ choices for $s_{r+2}$ making neither $v_1$ or $v_3$ zero, giving $(q-1)(q-2)$ choices for $s_{r+1}$ and $s_{r+2}$ which result in the left-most green right-side blade in Figure \ref{XXXTD}. The same reasoning explains the second and fourth right-side blades. Because the sequence $\mathbf{S}$ has right-side blade $\XXX$, it is impossible for $ \mathbf{S}'$ to have right-side blade $\OOX$ or $\XOO$, as this would result in a non-square window.
\begin{figure}[H]
    \centering
    \includegraphics[width=0.5\linewidth]{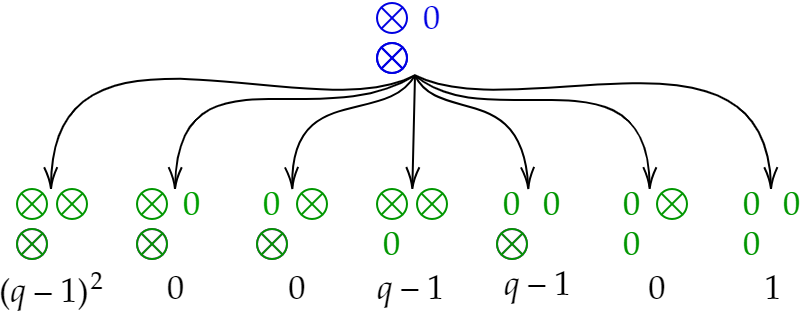}
    \caption{The $\XXO$-Tree Diagram.}
    \label{XXOTD}
\end{figure}
\noindent This case is much the same as the previous Tree Diagram. Using notation from Figure \ref{bp}, one has that $u_3$ is zero and hence $v_3$ is determined. However, it still depends on the value of $s_{r+1}$, and $v_2$ is free. Hence, there is one choice for $s_{r+1}$ making $v_2=0$ and $q-1$ choices for $s_{r+1}$ giving $v_2\neq0$. In the latter case, it is impossible to have $v_3=0$ as this would create two windows intersecting diagonally, violating the Square Window Theorem (Theorem \ref{window}). This explains the second right-side blade. However, $v_1$ is free and hence there are $q-1$ choices for $s_{r+1}$ making it nonzero and a single choice making it zero, explaining the first and fifth right-side blade. The third and sixth right-side blades are impossible due to the Square Window Theorem (Theorem \ref{window}). The third and seventh right-side blades follow exactly as in Figure \ref{XXXTD}.  \\

\noindent The remaining Tree Diagrams are shown below. The reader is encouraged to verify them to gain a better understanding.
\begin{figure}[H]
    \centering
    \includegraphics[width=1\linewidth]{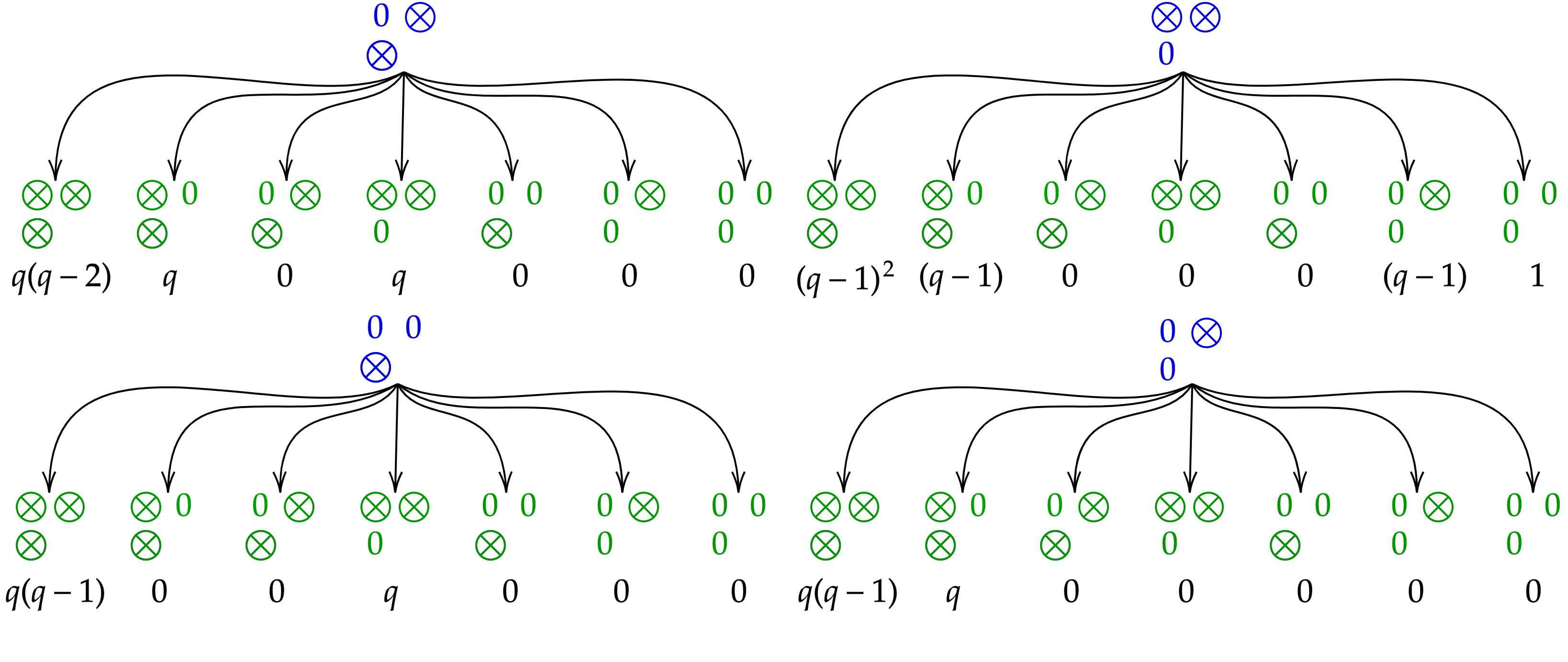}
    \caption{Top left: The $\XOX$-Tree Diagram. Top right: The $\OXX$-Tree Diagram. Bottom left: The $\XOO$-Tree Diagram. Bottom right: The $\OOX$-Tree Diagram.}
    \label{XOXTD}
\end{figure}
\noindent Note that there is no $\OOO$-Tree Diagram. This is because, in this case more context is required to calculate the values of $\{v_i\}_{i=1,2,3}$.
\subsubsection*{Extending Finite Sequences by an Even Number of Entries }
\noindent Given a finite sequence $\mathbf{S}$ with a nonzero right-side blade, the Tree Diagrams show how many length-two continuations of $\mathbf{S}$ give any desired right-side blade. The function introduced in the following definition generalises this to continuations of any arbitrary even length.

\begin{definition}\label{Qdef}
      Define the function \begin{align*}C:\{\text{nonzero right-side blades}\}\times\{\XXX,\XXO,\OXX\}\times(\N\cup\{0\})&\to\N\cup\{0\},\\
     (B_1,B_2,m)&\mapsto C(B_1,B_2,m)\end{align*}as the number of ways a finite sequence whose number wall has right-side blade $B_1$ can be extended by $2m$ entries such that the number wall of the resulting sequence has right-side blade $B_2$.\end{definition} 

\noindent For example, $C(\XXX,\XXO,1)$ is the number of ways a finite sequence whose number wall has right-side blade $\XXX$ can be extended by two digits to have a right-side blade $\XXO$. The $\XXX$-Tree Diagram shows that $C(\XXX,\XXO,1)=q-1$. When $m=0$, it will be convenient later to define $C(B_1,B_2,0)=1$ when $B_1$ is equal to $B_2$, and 0 otherwise.\\

\noindent Although it is not clear \textit{a priori}, the map $C$ is well defined. That is, the value of $C(B_1,B_2,m)$ depends only on the stated variables and not on any other specifics about the choice of sequence. This is justified below, in Lemma \ref{blades}. The same lemma gives explicit values of $C(B_1,B_2,m)$. These formulas do not need to be committed to memory and are not individually important. They are used later in the proof of Lemma \ref{contain}.

\begin{lemma}\label{blades}
   Let $B_1$ and $B_2$ be nonzero right-side blades and $m$ be a natural number. Then the function $C(B_1,B_2,m)$ is well-defined. Explicitly, the value of $C(B_1,B_2,m)$ only depends on the stated variables, and not any other values in the number wall. Furthermore, $C(B_1,B_2,m)$ takes the following values for $m\ge1$:
   \begin{enumerate}[leftmargin=2cm,rightmargin=2cm]
       \item[\large$B_2=\XXX:$ ]
\begin{enumerate}
    \item[(1.a)] $C(\XOO,\XXX,m)=C(\OOX,\XXX,m)=\begin{cases}\frac{q(q-1)^2}{q+1}(q^{2m-2}-1)&\text{ if }m\ge 2\\q(q-1)&\text{ if }m=1.\end{cases}$
    \item[(1.b)] $C(\XOX,\XXX,m)=\begin{cases}\frac{q(q-1)}{q+1}(q^{2m-1}-q^{2m-2}+2)&\text{ if } m\ge2\\q(q-2)&\text{ if } m=1.\end{cases}$
    \item[(1.c)] $C(\XXX,\XXX,m)=\frac{q-1}{q+1}(q^{2m}-q^{2m-1}-2)$.
    \item[(1.d)] $C(\XXO,\XXX,m)=C(\OXX,\XXX,m)=\frac{(q-1)^2}{q+1}(q^{2m-1}+1)$.
    \item[]
\end{enumerate}
       \item[\large$B_2=\XXO:$ ]
\begin{enumerate}
    \item[(2.a)]$C(\XXX,\XXO,m)=C(\OXX,\XXO,m)=\frac{q-1}{q+1}(q^{2m-1}+1)$
       \item[(2.b)] $C(\XXO,\XXO,m)=\frac{q(q-1)}{q+1}(q^{2m-2}-1)$
       \item[(2.c)] $C(\XOX,\XXO,m)=C(\OOX,\XXO,m)=\begin{cases}\frac{q(q-1)}{q+1}(q^{2m-2}-1)&\text{ if }m\ge2\\q&\text{ if }m=1\end{cases}$
       \item[(2.d)] $C(\XOO,\XXO,m)=\begin{cases}\frac{q^2(q-1)}{q+1}(q^{2m-3}+1)&\text{ if }m\ge2\\0&\text{ if }m=1.\end{cases}$
       \item[]
\end{enumerate}
       \item[\large$B_2=\OXX:$ ]
\begin{enumerate}
    \item[(3.a)] $C(\XXX,\OXX,m)=C(\XXO,\OXX,m)=\frac{q-1}{q+1}(q^{2m-1}+1)$
       \item[(3.b)] $C(\XOX,\OXX,m)=C(\XOO,\OXX,m)=\begin{cases}\frac{q(q-1)}{q+1}(q^{2m-2}-1)&\text{ if }m\ge2\\q&\text{ if }m=1.\end{cases}$
       \item[(3.c)] $C(\OXX,\OXX,m)=\frac{q(q-1)}{q+1}(q^{2m-2}-1)$
       \item[(3.d)] $C(\OOX,\OXX,m)=\begin{cases}\frac{q^2(q-1)}{q+1}(q^{2m-3}+1)&\text{ if }m\ge2\\0&\text{ if }m=1.\end{cases}$
\end{enumerate}
   \end{enumerate}
\end{lemma}
\noindent In the upcoming proof and throughout the paper, numbers written above an entry (or collection of entries in a brace) on the top row of a dot diagram indicate how many choices there are in $\F_q$ for that entry resulting in the indicated diagram. Unless specified otherwise, these choices are made from left to right. 
\begin{proof}
\noindent The lemma is proved by induction on $m$. The base cases are given by the Tree Diagrams (Figures \ref{XXXTD}, \ref{XXOTD} and \ref{XOXTD}), and it is clear these are well-defined. Assume that these formulas hold and that the function is well-defined up to $m-1$. The following dot diagram displays the situation:  \begin{figure}[H]
       \centering
       \includegraphics[width=0.5\linewidth]{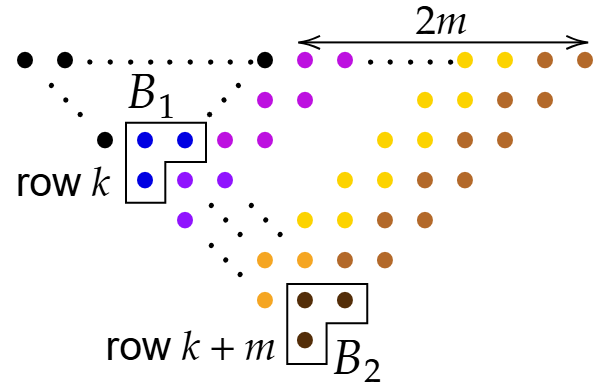}
       \caption{The black dots to the top left represent the starting number wall, with its right-side blade in blue. Each pair of colours represents an additional two entries being added to the sequence, with the right-side blade created by these additions being in a darker shade. This dot diagram assumes the starting sequence has odd length. In the even length case, the first element of the starting sequence can be entirely ignored, reducing back to the odd case.}
   \end{figure}

\noindent Every formula is proved using the same idea. The Tree diagram for $B_1$ shows which possible profiles the right-side blade could take when the generating sequence is extended by two entries, and how many of those extensions result in each right-side blade profile. The induction hypothesis is then applied and some basic algebra is completed.\\

\noindent  The proof of Lemma \ref{blades} is split into two parts. Within each part, the formulae are proved in triples: given a choice of $B_1$, the value of $C(B_1,B_2,m)$ is derived for a generic choice of $B_2$. To obtain an individual formula, a specification of the profile of $B_2$ is made and a calculation is performed. Each formulae will be referred to with the labels from Lemma \ref{blades}. 
\subsubsection*{Proof of Equations (1.a), (1.b), (2.c), (2.d), (3.b) and (3.d)}

    The $\XOO$-Tree Diagram (Figure \ref{XOXTD}) shows that when extending the sequence by two, there are $q(q-1)$ ($q$, respectively) ways to get a right-side blade of profile $\XXX$ (of profile $\OXX$, respectively). With the notations of Definition \ref{Qdef}, it is clear that \[C(\XOO,B_2,m)=q(q-1)\cdot C(\XXX,B_2,m-1)+q\cdot C(\OXX,B_2,m-1).\] Picking a specific right-side blade profile for $B_2$, using the induction hypothesis and doing some simple rearranging gives formulae $(1.a)$, $(2.d)$ and $(3.b)$. It is also clear from this recurrence formula that, if $C(\XXX,B_2,m-1)$ and $C(\OXX,B_2,m-1)$ are well-defined, then $C(\XOO,B_2,m)$ is also. In a similar fashion, the $\OOX$- Tree Diagram (Figure \ref{XOXTD}) shows that\[C(\OOX,B_2,m)=q(q-1)\cdot C(\XXX,B_2,m-1)+q\cdot C(\XXO,B_2,m-1)\] and the $\XOX$-Tree Diagram (Figure \ref{XOXTD}) that $C(\XOX,B_2,m)$ is equal to\[q(q-2)\cdot C(\XXX,B_2,m-1)+q\cdot C(\XXO,B_2,m-1)+ q\cdot C(\OXX,B_2,m-1).\] Once again, picking a specific profile for $B_2$, using the induction hypothesis and rearranging proves the two equalities in $(1.a)$, $(2.c)$, and also the equalities $(3.d)$, $(1.b)$, and $(3.b)$.
\subsubsection*{Proof of Equations (1.c), (1.d), (2.a), (2.b), (3.a) and (3.c)}
    \noindent The formulae for $B_1=\XXX,~ \XXO$ and $\OXX$ are not as simple. The first to be proved is $C(\XXX,B_2,m)$. Using the~$\XXX$-Tree Diagram (Figure \ref{XXXTD}), \begin{align}&C(\XXX,B_2,m)=\nonumber\\&(q-1)(q-2)\cdot C(\XXX,B_2,m-1)+(q-1)\cdot C(\XXO,B_2,m-1)\nonumber\\&+(q-1)\cdot C(\XOX,B_2,m-1)+(q-1)\cdot C(\OXX,B_2,m-1)\nonumber\\&+\widetilde{C}_{\XXX}(\OOO,B_2,m-1).\label{prodind}\end{align} 
    \noindent Here, $\widetilde{C}_{\XXX}(\OOO,B_2,m-1)$ is the number of ways to add $2(m-1)$ digits to a generating sequence with right-side blade $\XXX$, with the additional constraint that the first two digits result in the zero right-side blade. This is illustrated below: \begin{figure}[H]
        \centering
        \includegraphics[width=0.75\linewidth]{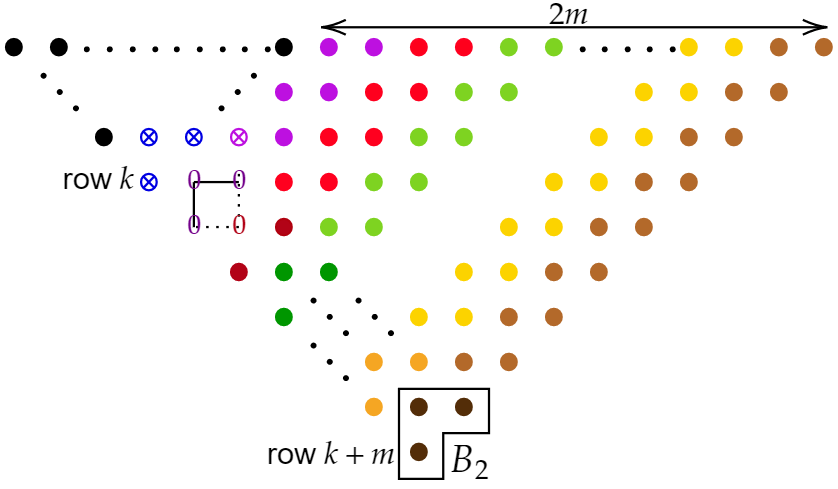}
        \caption{$\widetilde{C}_{\XXX}(\OOO,B_2,m-1)$ is equal to the number of continuations of length $2m$ to the generating sequence (black dots, top row), such that the finite number wall of the generating sequence extended by the first two digits (purple dots, top row) has the zero right-side blade and the full continuation (ending in brown dots, top row) gives the right-side blade $B_2$. The dotted line for the window with top left corner on row $k$ indicates its size has not been determined.}
    \end{figure}
    
    \noindent Let $\mathcal{W}$ be the window created on row $k$. That is, the window that overlaps with the zero right-side blade created by the 2-digit extension of the generating sequence. Evaluating $\widetilde{C}_{\XXX}(\OOO,B_2,m-1)$ amounts to counting the number of continuations of the generating sequence given every possible size of $\mathcal{W}$. For example, the minimum possible size of $\mathcal{W}$ is $2\times 2$, since it contains the zero right-side blade. If $\mathcal{W}$ has side length 2, then the number wall looks as the left one below:\begin{figure}[H]
        \centering
        \includegraphics[width=1\linewidth]{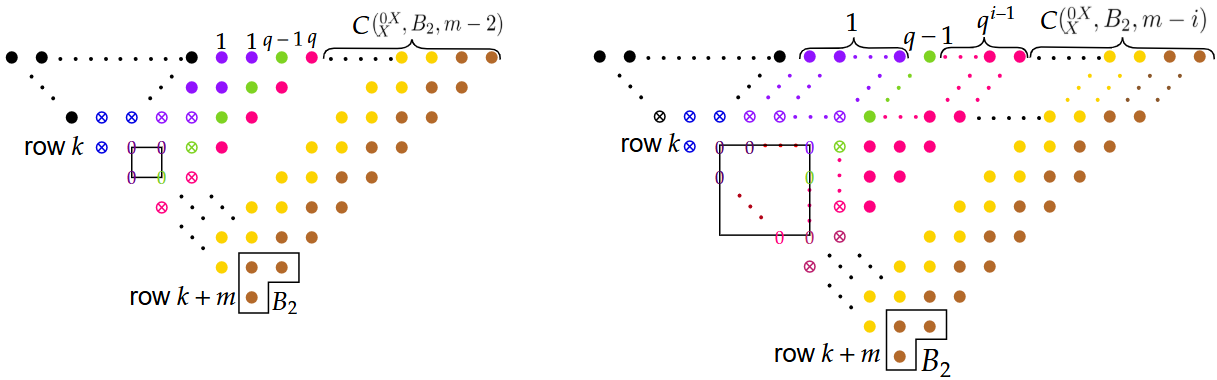}
        \caption{\textbf{Left:} If $\mathcal{W}$ has minimal size, then Lemma \ref{detlem} shows that the first of the two red entries in row zero can take any of the $q-1$ different values that would close the window. The second red entry in row zero is free to take any of the $q$ possible values. \textbf{Right:} The dot diagram when the window on row $k$ is size $i$.}
        \label{bi}
    \end{figure}
    \noindent The case when $\mathcal{W}$ is of size $2\times 2$ contributes $q(q-1)\cdot C(\XOX,B_2,m-2)$ to the total sum. In general, if $\mathcal{W}$ has size $i\times i$, it contributes $(q-1)q^{i-1}C(\XOX,B_2,m-i)$ to the total. This is illustrated in the right-hand side of Figure \ref{bi}. Summing over all values of $i$ gives \begin{align}
        C(\XXX,B_2,m)&=(q-1)(q-2)\cdot C(\XXX,B_2,m-1)+(q-1)\cdot C(\XXO,B_2,m-1)\nonumber\\&+(q-1)\cdot C(\OXX,B_2,m-1)+\sum_{i=1}^{m-1}(q-1)q^{i-1}\cdot C(\XOX,B_2,m-i).\label{indprod}
    \end{align}
\noindent Note that the $C(\XOX,B_2,m-1)$ term from (\ref{prodind}) has been absorbed into the sum above. At this point, a specific choice for $B_2$ is made, the induction hypothesis is used and the sum is evaluated using the geometric sequence formula to obtain relations $(1.c)$, $(2.b)$ and $(3.a)$.\\

\noindent The values of $C(\OXX,B_2,m)$ and $C(\XXO,B_2,m)$ are derived using an identical method. For this reason, the full detail of these proofs is omitted. Instead, only the recurrence relations akin to equation (\ref{indprod}) are provided:\vspace{-0.2cm}\begin{align*}
    C(\XXO,B_2,m)=&(q-1)^2\cdot C(\XXX,B_2,m-1)+(q-1)\cdot C(\OXX,B_2,m-1)\nonumber\\&+\sum_{i=1}^{m-1}(q-1)q^{i-1}\cdot C(\XOO,B_2,m-i).\nonumber\\
    C(\OXX,B_2,m)=&(q-1)^2\cdot C(\XXX,B_2,m-1)+(q-1)\cdot C(\XXO,B_2,m-1)\nonumber\\&+\sum_{i=1}^{m-1}(q-1)q^{i-1}\cdot C(\OOX,B_2,m-i).\nonumber
\end{align*}\vspace{-0.7cm}
\end{proof}

\subsubsection*{Establishing Lemma \ref{contain}}
\noindent Recall the statement of Lemma \ref{contain}. Figure \ref{contain_fig1} is recreated with additional details that will aid in the proof. \begin{figure}[H]
    \centering
    \includegraphics[width=1\linewidth]{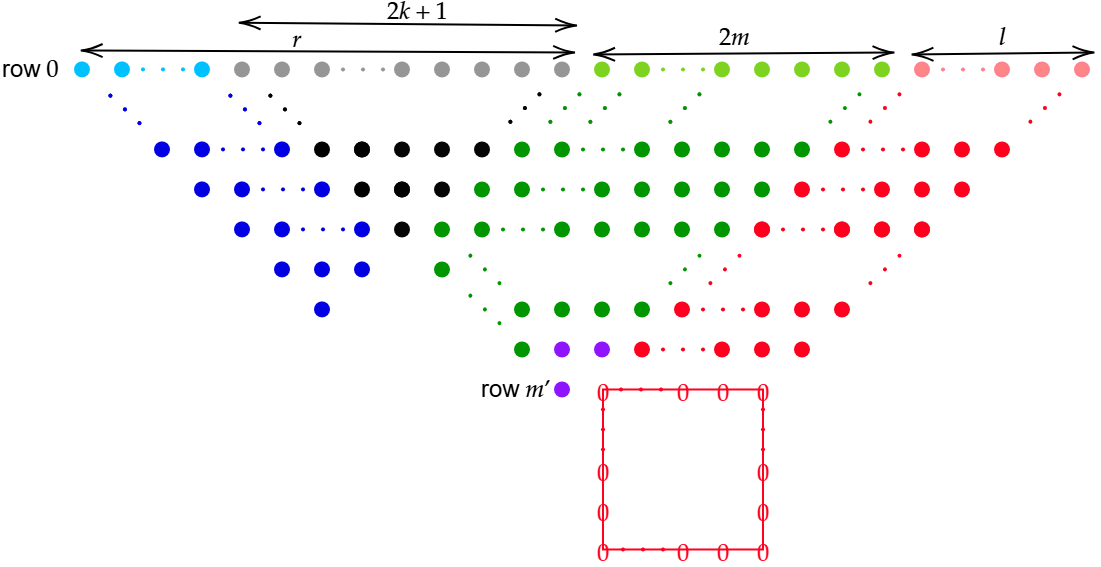}
    \caption{The finite sequence $\mathbf{S}$ (light blue and grey dots) of length $r$ generates a finite number wall (light blue, dark blue, grey and black dots). The subsequence $ \mathbf{S}_k$ (grey dots) generates the finite number wall (grey and black dots) $W_q( \mathbf{S}_k)$. The sequence $ \mathbf{S}'$ (light red dots) extends $ \mathbf{S}_k$ in such a way that $W_q( \mathbf{S}_k\oplus \mathbf{S}')$ (grey, black, light and dark red dots) has a window containing the the square portion $\Box(l,m',n')$ (red square).}
    \label{Contain_proof_fig1}
\end{figure}

\noindent Using notation from Figure \ref{Contain_proof_fig1}, the goal is to calculate how many extensions of length $2m+l$ (since, in this case, $i=1$) there are to $\mathbf{S}$ such that the resulting number wall has a (complete or incomplete) window containing $\Box$.

\begin{proof}[Proof of Lemma \ref{contain}]
\noindent Firstly, it is sufficient to assume that $r=2k+i$. If this is not the case, the first $r-2k-i$ entries of $\mathbf{S}$ can be completely ignored, which returns the proof to the case $r=2k+i$. This has no effect on the calculation, since these entries are not part of any minimal generators of $\Box$. \\

\noindent Let $B$ be the right-side blade of $W_q( \mathbf{S}_k)$. The proof now splits into multiple cases.\begin{itemize}
    \item{Case 1:} $i=1$ and $B$ is not the zero right-side blade;
    \item{Case 2:} $i=1$, $B$ is the zero right-side blade and the incomplete window that contains $B$ is right-side closed;
    \item{Case 3:} $i=1$, $B$ is the zero right-side blade and the incomplete window that contains $B$ is right-side open;
    \item{Case 4:} $i=2$ and the entry of $B$ with the greatest row index is nonzero.
    \item{Case 5:} $i=2$, the entry of $B$ with the greatest row index is zero and the incomplete window that contains this entry of $B$ is right-side closed;
    \item{Case 6:} $i=2$, the entry of $B$ with the greatest row index is zero and the incomplete window that contains this entry of $B$ is right-side open.
\end{itemize} 

\noindent The majority of the work takes place in cases 1 and 3. Furthermore, cases 5 and 6 are extremely similar to cases 2 and 3 and therefore only the differences in the proofs are mentioned. 
\subsubsection*{Case 1}
\noindent In Case 1, it is assumed that $W_q( \mathbf{S}_k)$ has a nonzero right-side blade. The diagram below illustrates the set up. \begin{figure}[H]
    \centering
    \includegraphics[width=0.75\linewidth]{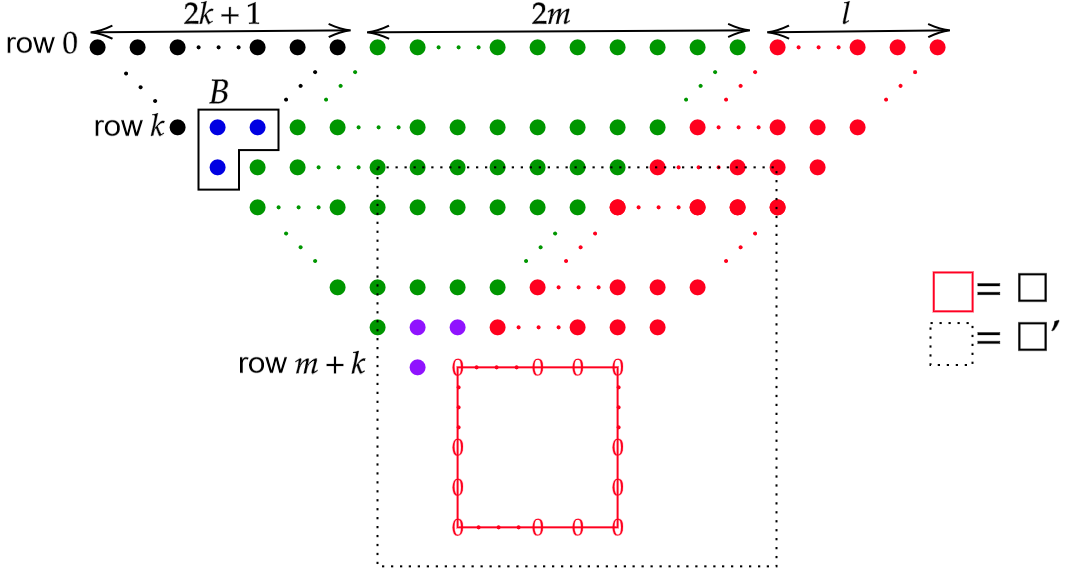}
    \caption{Case 1. The sequence $ \mathbf{S}_k$ (top row, black) has nonzero blade (dark blue). It is extended by $2m+l$ entries (top row, green and red) so that the extended number wall has an incomplete window (dotted square) containing the square portion $\Box$ (red square). Here, the dotted lines mean that this is just one of many possible windows. The right-side blade of the sequence of length $2k+2m+1$ is in purple. }\label{contain pic}
\end{figure}\noindent The result is proved by summing over all the sequences whose first $2k+1$ digits agree with $ \mathbf{S}_k$ and whose number walls have (incomplete or complete) windows that contain $\Box$. These sequences are partitioned by the row and column that their windows begin in. For example, the first section of the partition contains all the sequences of length $2k+2m+l+1$ that are extensions of $ \mathbf{S}_k$ and whose number wall has a window containing $\Box$ that has its top left corner in either row $m+k$ or column $n$. This partition is illustrated by the left dot diagram below. The right diagram is used later in the proof and should be ignored for now.\begin{figure}[H]
    \centering
    \includegraphics[width=1\linewidth]{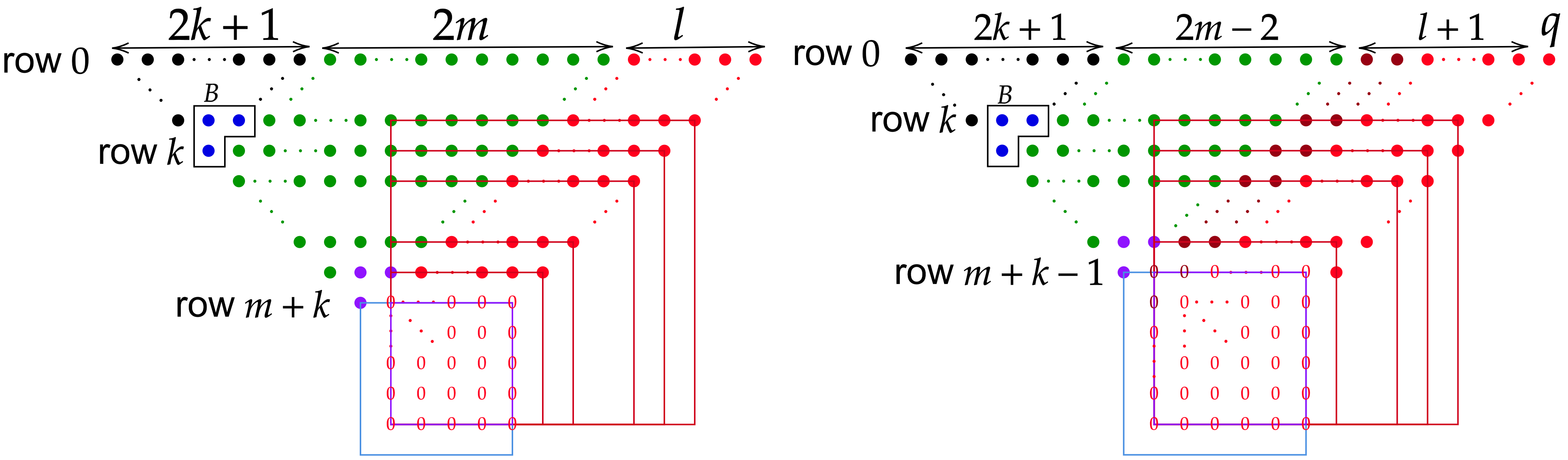}
    \caption{\textbf{Left:} The purple square denotes the incomplete window that contains $\Box$ and which has its top left corner in the same row and column as $\Box$; that is, $\Box$ itself. The blue square represents the incomplete window that contains $\Box$ which starts on the same row as $\Box$ but not the same column. All the red squares represent windows that start in the same column as $\Box$. \textbf{Right:} When the starting row or column has been decreased by one, the square portions in this partition are one row higher. This shortens the green extension to the black sequence by two terms. The dark red entries are the ones that were in the green extension to $ \mathbf{S}_k$ in the first partition but are now part of the unique extension required to generate the window.}
    \label{owco}
\end{figure}

\noindent \noindent With notation from Definition \ref{Qdef}, there are $C(B,\XXX,m)$ ways to extend $\textbf{S}_k$ by $2m$ entries such that the purple right side blade (left image of Figure \ref{owco}) has profile $\XXX$. Then, the Free-Determined Lemma (Lemma \ref{detlem}) implies there is a unique continuation of length 1 that creates a zero in the top left corner of $\Box$. Then, Proposition \ref{wind_extend_lem} shows that there is a unique extension of length $l-1$ that would create a right-side open incomplete window that contains the square portion $\Box$. This right-side open incomplete window is the purple square of Figure \ref{owco}. Now consider the window denoted by the blue square in the left image of Figure \ref{owco}. In this case, the purple right-side blade must have profile $\OXX$, and hence by applying Proposition \ref{wind_extend_lem}, there are $C(B,\OXX,m)$ ways of obtaining this right-side open incomplete window that contains $\Box$. Next, note that a sequence cannot have more than one of the red windows from Figure \ref{owco} in its number wall, as they intersect. Also note that each of the red windows requires that the purple right-side blade has profile $\XXO$. Hence, the sequences counted by $C(B,\XXO,m)$ are precisely the sequences whose number walls contain one of the red windows. Each of these sequences must then be extended $l$ times by applying Proposition \ref{wind_extend_lem}. Therefore, the contribution of this whole partition to the total sum is \[C(B,\XXX,m)+C(B,\OXX,m)+C(B,\XXO,m).\] 

\noindent The second section of the partition is similar, but contains all the sequences of length $2k+2m+l+1$ that have windows containing $\Box$ that have their top left corner in row $m+k-1$ or column $n-1$. This is illustrated in the right-hand side of Figure \ref{owco}. As before, begin with the window drawn in purple in the right image of Figure \ref{owco}. The purple right-side blade has moved in position up one row and left one column and it once again must have profile $\XXX$. There are $C(B,\XXX,m-1)$ ways to attain a continuation to $ \mathbf{S}_k$ that satisfies this. However, these continuations are of length $2m-2$, instead of $2m$ as in the previous case. Furthermore, the size of the windows in this partition has only increased by one. Therefore, the final entry in the a sequence of length $2k+2m+l+1$ is free to take any value and hence there are $qC(B,\XXX,m-1)$ such sequences. The argument is the same for the red and blue windows containing the square portion $\Box$ in Figure \ref{owco}, and so this case contributes \[q\big(C(B,\XXX,m-1)+C(B,\OXX,m-1)+C(B,\XXO,m-1)\big)\] to the total sum. Each partition is done in a similar way, leading to a total value of \begin{equation}\sum^{m}_{i=0}q^i\big(C(B,\XXX,m-i)+C(B,\OXX,m-i)+C(B,\XXO,m-i)\big).\label{containtotal}\end{equation} From Lemma \ref{blades}, it can be seen that for $m>1$,\begin{equation}C(B,\XXX,m)+C(B,\OXX,m)+C(B,\XXO,m)=q^{2m-1}(q-1).\label{trouble}\end{equation} If $B=\XXX,\XXO,$ or $\OXX$, this formula also applies for $m=1$ and the sum equals 1 for $m=0$. If $B=\OOX,\XOO$ or $\XOX$, then at $m=1$, the left-hand side of (\ref{trouble}) has value $q^2$.\\

\noindent In the former case when $B=\XXX,\XXO$ or $\OXX$, \begin{align*}(\ref{containtotal})&=\sum_{j=0}^{m-1} q^{2m-1-2j}(q-1)\cdot q^j + q^{m}
    =q^{m}((q-1)(q^{m-1}+q^{m-2}+\dots+q+1)+1)=q^{2m}.
\end{align*}

\noindent In the latter case when $B=\XOO,\OOX,$ or $\XOX$, 
\begin{align*}
(\ref{containtotal})&=\sum_{j=0}^{m-2} q^{2m-1-2j}(q-1)\cdot q^j + q^{m-1}\cdot q^2
    =q^{m+1}((q-1)(q^{m-2}+q^{m-3}+\dots+q+1)+1)=q^{2m}.
\end{align*}
This concludes Case 1.1. 

\subsubsection*{Case 2}
\noindent In this case, it is assumed that the window that contains the right-side blade $B$ is right-side closed. For brevity, call this window $\mathcal{W}$. Figure \ref{3.5.6} illustrates this case. Because $\mathcal{W}$ is closed, there exists a maximal $s\in\N$ such that when $ \mathbf{S}_k$ is extended to be a sequence of length $2k+1+s$, denoted $ \mathbf{S}_k'$, the right-side blade of $W_q( \mathbf{S}_k')$, denoted $B'$, is part of $\mathcal{W}$ or its inner frame and is therefore independent of the choice of extension. Furthermore, $B'$ is not the zero right-side blade, since if it were then it would be contained within the window $\mathcal{W}$ and hence by increasing $s$ by two would result in a new right-side blade that was contained in the inner frame of $\mathcal{W}$, contradicting the maximality of $s$. \begin{figure}[H]
    \centering
    \includegraphics[width=0.75\linewidth]{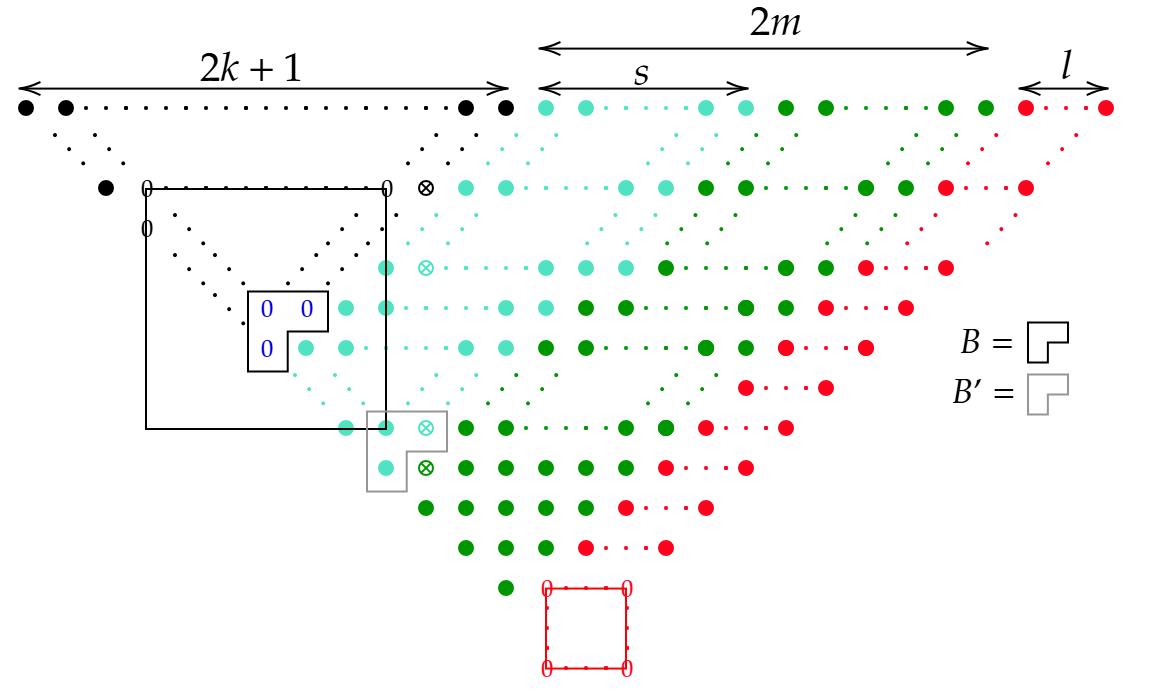}
    \caption{The number wall $W_q( \mathbf{S}_k)$ (black dots, zeroes, crossed dot and rotated L-shape) has a window (black square) that contains the zero right-side blade $B$ (dark blue zeroes). This window is right-side closed. The right-side blade $B'$ of the extended sequence is outlined in grey.  }
    \label{3.5.6}
\end{figure} \noindent Applying Case 1 to any of the $q^s$ possible extensions completes the proof in this case.

\subsubsection*{Case 3}
\noindent This case occurs when the window $\mathcal{W}$ in $W_q( \mathbf{S}_k)$ that contains the right-side blade $B$ is right-side open. The proof is achieved by summing over all the possible sizes $\mathcal{W}$ could attain. To begin, let $u\in\N$ be such that $\mathcal{W}$ has its top left corner on row $k-u$, and let $v\in\N$ be such that $\mathcal{W}$ has its bottom right corner on row $k+v$. Hence, $\mathcal{W}$ has side length $u+v+1$. There are $q-1$ ways to extend $ \mathbf{S}_k$ by one digit (into a sequence of length $2k+2$) in such a way that $\mathcal{W}$ becomes right-side closed. Then, the number wall of \textit{any} extension of $ \mathbf{S}_k$ of length $2v+1$ (such that $ \mathbf{S}_k$ becomes a sequence of length $2k+2v+3$) will have a nonzero right-side blade $B'$. 

\begin{figure}[H]
    \centering
    \includegraphics[width=0.8\linewidth]{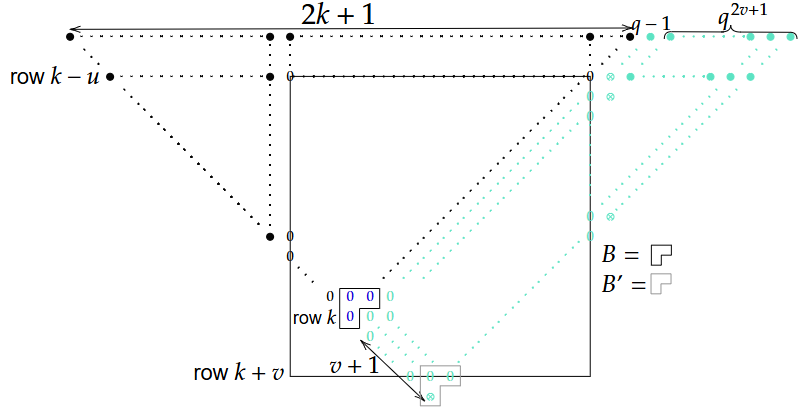}
    \caption{The window $\mathcal{W}$ is depicted by the black square. It can be closed in $q-1$ ways, represented by the first light blue dot on row zero, and then the next $q^{2v+1}$ entries of $ \mathbf{S}_k$ can be chosen freely to obtain the right-side blade $B'$}
\end{figure}
\begin{remark}In the figure, $B'$ is the $\XOO$ right-side blade. It could also be the $\XOX$ or the $\OOX$ right-side blade, depending on the position of $\mathcal{W}$. This does not effect the calculations.\end{remark}

\noindent From here, Case 1 implies that there are $q^{2m-2v-2}$ continuations of $ \mathbf{S}_k$ of length $2m-2v+2+l$ whose number walls have windows containing $\Box$, contributing $(q-1)q^{2m-1}$ to the total number of possible extensions to $ \mathbf{S}_k$ that result in a finite sequence of length $2k+2m+l+1$ and whose number wall has a window containing $\Box$. \\

\noindent Now, instead of closing $\mathcal{W}$, one could extend it by length $j$ (by Proposition \ref{wind_extend_lem}) for any $j\in\N$ such that $\mathcal{W}$ nor its inner frame (should it be closed at that size) intersect $\Box$. This is illustrated below. \begin{figure}[H]
    \centering
    \includegraphics[width=01\linewidth]{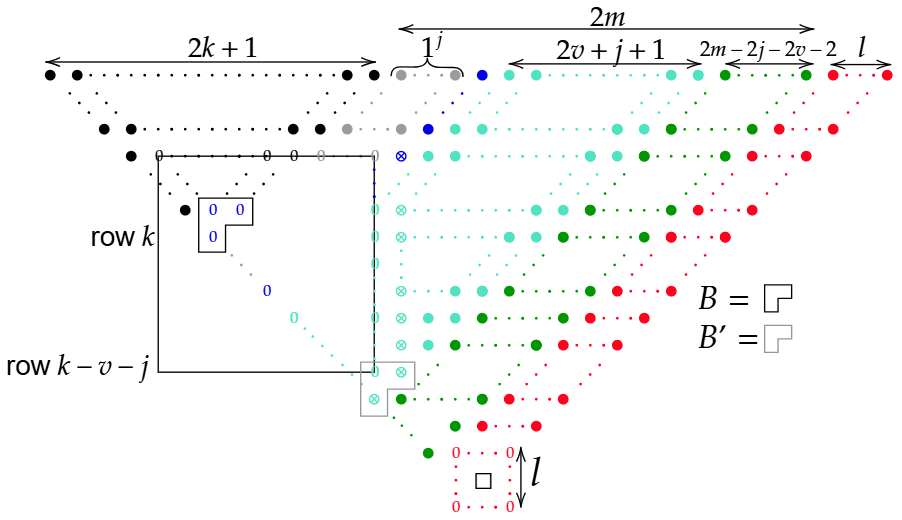}
    \caption{The window $\mathcal{W}$ (black square) has been extended by $j$ entries (grey dots in row zero) before being closed (top row, dark blue dot).}
\end{figure}

\noindent In this case, $ \mathbf{S}_k$ is uniquely extended by $j$ entries before one further entry (of which are there $q-1$ choices) is appended to the end in order to close $\mathcal{W}$. This increases the row index of the bottom row of $\mathcal{W}$ by $j$, meaning that $ \mathbf{S}_k$ can extended by $2v+j+1$ digits, all of which are chosen with complete freedom, and the resulting number wall will have the right-side blade profile $\XOO$. Then, as was just explained in the instance $j=0$, one applies Lemma \ref{contain} to obtain a contribution of $q^{2v+j+1}\cdot(q-1)q^{2m-2j-2v-2}=(q-1)q^{2m-j-1}$ to the total number of extensions of $ \mathbf{S}_k$ that have a window containing $\Box$.
\noindent The value of $j$ ranges from $0$ to $m-v-1$, and hence summing over all values of $j$ yields \begin{equation}\label{eqn: corections1}
    (q-1)\sum^{m-v-1}_{j=0}q^{2m-1-j} = q^{2m} - q^{m+v}
\end{equation}
\noindent Finally, one must consider the case where $\mathcal{W}$ is expanded to contain $\Box$. 
\begin{figure}[H]
    \centering
    \includegraphics[width=1\linewidth]{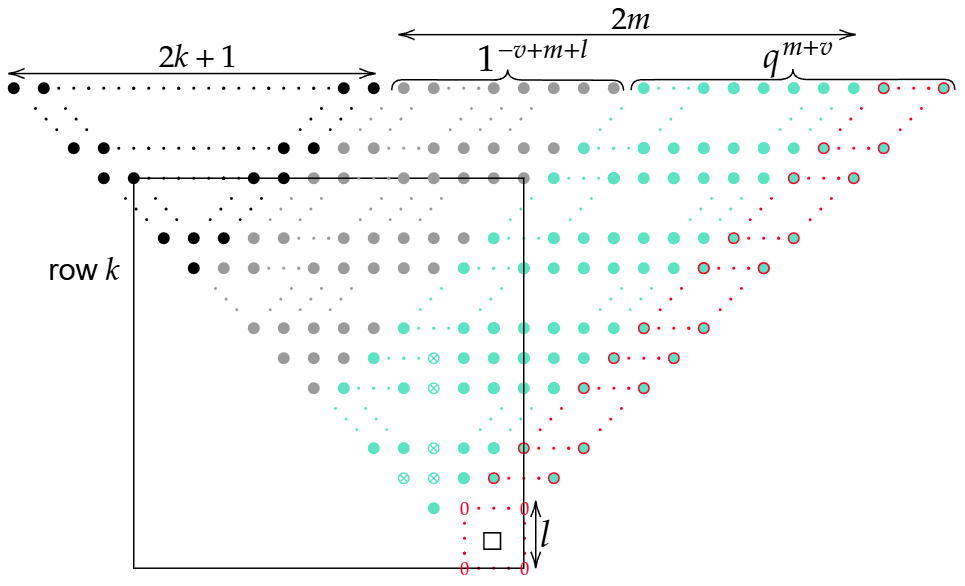}
    \caption{The dot diagram when window $\mathcal{W}$ (black square) now contains $\Box$ (red square). The entries that were chosen uniquely to create $\Box$ (red in Figure \ref{3.5.6}) can now be chosen arbitrarily. }
\end{figure}
\noindent That is, $\mathcal{W}$ has been extended by $m-v+l$ entries. As there is no longer any need to close $\mathcal{W}$, this contributes $q^{m+v}$ to the final sum. Combining this with equation (\ref{eqn: corections1}) implies that the total sum is\begin{align}
q^{m+v}+q^{2m}-q^{m+v}=q^{2m}.\label{wind cont sum}
\end{align}

\noindent This concludes Case 3.

\subsubsection*{Case 4}

\noindent Most of the hard work for this case has already been done. The left image in the following figure illustrates the set up. The right image is used later in the proof.\begin{figure}[H]
    \centering
    \includegraphics[width=1\linewidth]{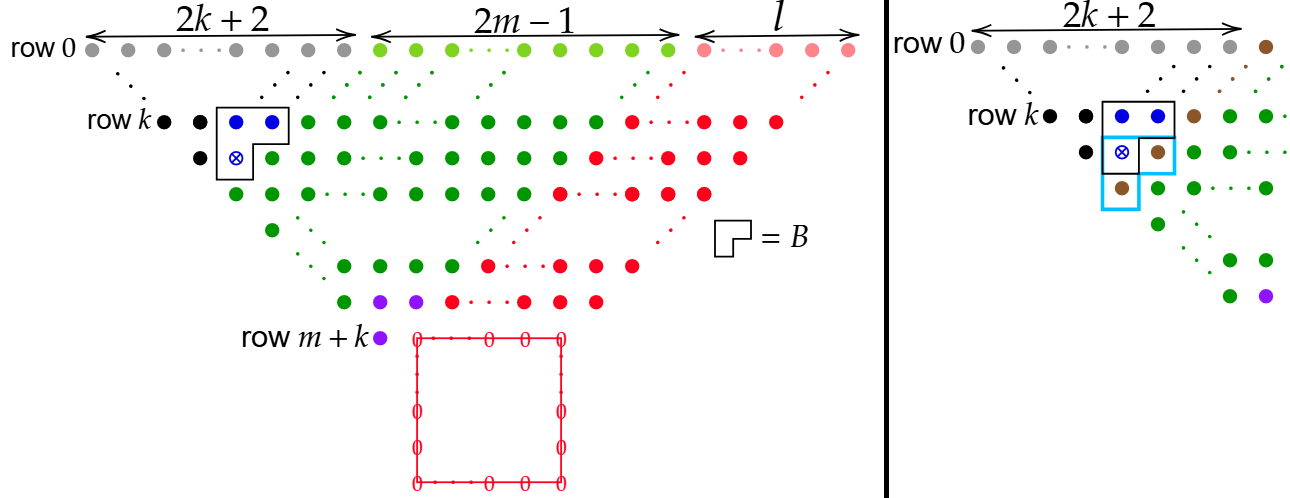}
    \caption{\textbf{Left}: The set up of Case 4, with the same colours as Figure \ref{Contain_proof_fig1}. \textbf{Right}: Assume the bottom entry of the right-side blade $B$ (crossed blue dot) is nonzero. Then any single digit extension (brown dot) to $ \mathbf{S}_k$ (grey dots) will result in a right-side blade $B'$ (light blue reversed L-shape figure) that is nonzero.}\label{fig:contain2.1}
\end{figure}

\noindent As the bottom entry of the right-side blade $B$ is nonzero, every single-digit extension to $ \mathbf{S}_k$ will result a sequence of length $2k+3$ whose finite number wall has a nonzero right-side blade. This is illustrated in the right image of Figure \ref{fig:contain2.1}. Thus, this reduces to the $i=1$ case and so there are $q\cdot q^{2m-2}$ possible extensions. Multiply this by $q$, which corresponds to the arbitrary single-digit extension to $ \mathbf{S}_k$, completes this case.

\subsubsection*{Case 5}
\noindent This case is almost identical to Case 2, and so only the differences are provided.  Indeed, if $B$ were the zero right-side blade, the only change from Case 2 would be that the value of $s$ had decreased by one. \\

\noindent If $B$ were not the zero right-side blade, then by the assumption that the window containing the lowest element of $B$ is right-side closed, $B=~\OOX$. Then, any three-digit extension to $ \mathbf{S}_k$ reduces this case back to Case 1 with $2k+1$ replaced with $2k+5$. 

\subsubsection*{Case 6}
\noindent This is very similar to Case 3. Indeed, the only significant difference is that (borrowing notation from Case 3) instead of $q^{2v+j+1}$ entries being chosen arbitrarily, one has $q^{2v+j}$ such entries instead. One other minor change is that the incomplete window $\mathcal{W}$ could be of size $1\times1$, which would make $v=0$.
\end{proof}
\subsubsection*{Establishing Corollary \ref{contain full}}
\begin{proof}[Proof of Corollary \ref{contain full}]
\noindent First, the condition that $r\ge m+n+l$ is required so that the finite sequence is long enough to determine if there is a window (of any kind) containing $\Box$. This, along with the rest of this proof, is depicted below. \begin{figure}[H]
    \centering
    \includegraphics[width=1\linewidth]{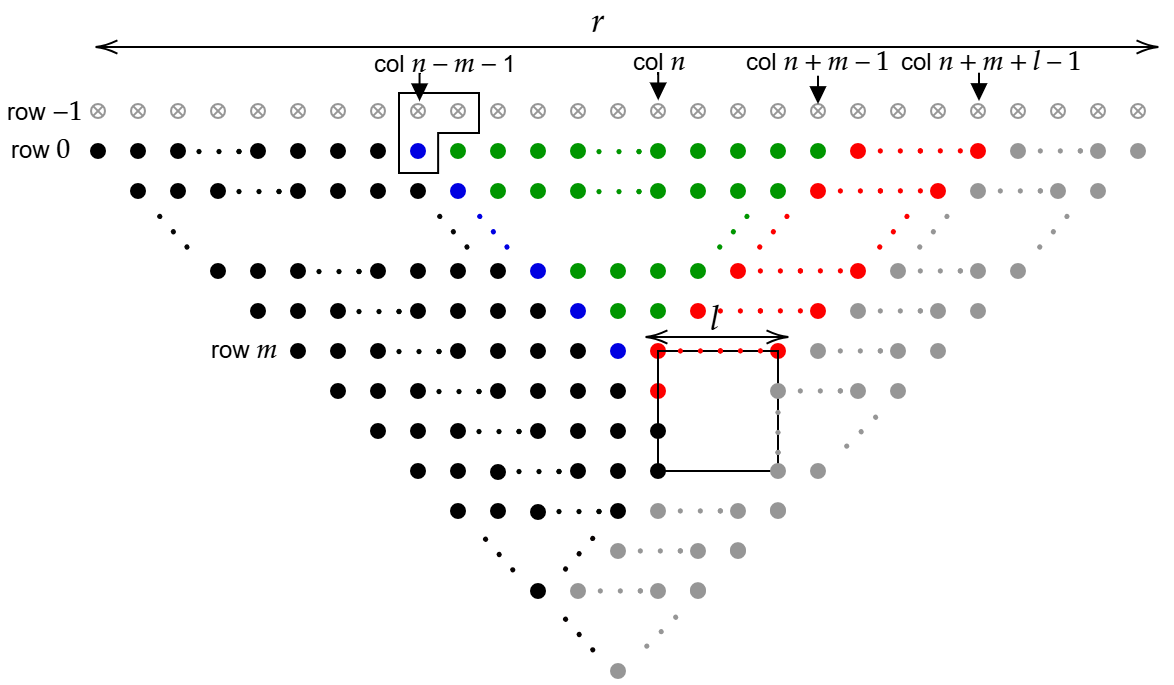}
    \caption{In order for the number wall of a finite sequence of length $r$ to have a window that contains $\Box$ (black square), it is necessary that $r>n+m+l$. The entries making up a minimal generator of $\Box$ are given by the green and red entries in row $0$.}
\end{figure} \noindent The entries of the sequence with index $n-m-1\le i \le n+m+l-1$ are calculated first. The first entry in this range is picked arbitrarily. Since row $-1$ is made of entirely 1s, this forms a right-side blade of profile $\XXX$ or $\OXX$. Either way, Lemma \ref{contain} is applied to show there are $q^{2m}$ extensions to this single digit sequence that result in a sequence whose number wall has a window (of any kind) containing $\Box$. The choices for the remaining entries of the sequence have no effect,  giving a total number of sequences of length $r$ as $$q^{r-2m-1-l}\cdot q^{2m+1}=q^{r-l}.$$ This concludes the proof.
\end{proof}
\subsection[Counting Sequences with Windows Containing \\Rectangular Portions]{Counting Sequences with Windows Containing Rectangular Portions}\label{Sect:Rect}
\noindent The goal of this section is to prove Lemma \ref{bound}, which is an immediate corollary of Lemma \ref{Rect} below which is a stronger result that counts precisely how many finite sequences of length $r\in\N$ exist whose finite number walls have (complete or incomplete) windows containing a given rectangular portion. Recall the notation $R_{r,q}(l,d,m,n)$ from Definition \ref{Rect_prtn}.

\begin{lemma}\label{Rect}
   Let $r,n,m,l\in\N$, $d\in\Z$ satisfying $m-n\le d\le l-1$ and let $q$ be a prime power. Assume also that $r>2(n+m+l)$. Then, \begin{enumerate}
        \item If $m-n\le d\le0$, then $\#R_{r,q}(l,d,m,n)=$ \begin{equation}\frac{q^{r-(l+2m-d+1)}}{q+1}\left((1-d)\cdot q^{2m+2}+(d+1)\cdot q^{2m+1}-d\cdot(q-1)\right).\label{vertrec}\end{equation}
       \item If $0< d\le m$, then $\#R_{r,q}(l,d,m,n)=$\begin{equation}\frac{q^{r-(l+2m)}}{q+1}\cdot \left((d+1)\cdot q^{2m+1} -(d-1)\cdot q^{2m}- \frac{q^{2d}-1}{q+1}\right).\label{horrec}\end{equation}
       \item If $d>m$, then equation (\ref{horrec}) holds upon replacing $d$ with $m$ in the right-hand side.  
   \end{enumerate} 
   \noindent 
\end{lemma}

\begin{remark}\label{rect_remark}
     The condition that $r/2>n+m+l$ ensures that at least one of the longest sides of the rectangle specified by $R_{r,q}({l,d,m,n})$ is fully contained in the number wall, as well as all the subsequent rectangular portions that are involved in the upcoming proof. In practice, Lemma \ref{Rect} is applied to calculate the measure of a set of Laurent series that have windows containing a given rectangular portion, and so $r$ is considered to be arbitrarily large.  
\end{remark}
\begin{remark}
    When $d=0$, equation (\ref{vertrec}) reduces to Corollary \ref{contain full}.
\end{remark} 

\noindent Corollary \ref{contain full} and Lemma \ref{Rect} provide a complete picture for counting how many finite sequences have a window containing any connected portion of zeros in their number wall. Even if the portion of zeros is not in a rectangular shape, a minimally sized rectangle can be drawn around it and the Square Window Theorem (Theorem \ref{window}) shows that this rectangle is also fully zero. 
\begin{proof}[Proof of Lemma \ref{Rect}]
\noindent Equations (\ref{vertrec}) and (\ref{horrec}) are proved by induction on $d$, where the base case of each occurs when $d=\pm1$. To complete the induction, recurrence relations are proved that allow one to calculate $R_{r,q}(l,d,m,n)$ in terms of $R_{r,q}(l,d',m,n)$ for $|d'|<|d|$. \\

\noindent The proof of Lemma \ref{Rect} is split into four parts. The first establishes recurrence relations that are used to complete the base case of the induction (part 2). The inductive step is completed in the third part, before the third statement of Lemma \ref{Rect} is proven in the final part.\subsubsection*{Calculating $R_{r,q}(l,d,m,n)$ with a Recurrence Relation in $d$.} 
\noindent  Assume first that $m-n\le d<0$. Let $\mathbf{S}$ be a sequence in $R_{r,q}(l,d,m,n)$. Then, there are three possibilities:\begin{enumerate}
\item[(1.1)] $\textbf{S}\in R_{r,q}({l+1,d+1,m,n-1})$ and $\textbf{S}\not\in R_{r,q}({l+1,d+1,m,n})$,
\item[(1.2)] $\textbf{S}\not\in R_{r,q}({l+1,d+1,m,n-1})$ and $\textbf{S}\in R_{r,q}({l+1,d+1,m,n})$,
\item[(1.3)] $\textbf{S}\in R_{r,q}({l+1,d+1,m,n-1})$ and $\textbf{S}\in R_{r,q}({l+1,d+1,m,n})$.
\end{enumerate}
\noindent These three cases are illustrated by the left image in Figure \ref{rect pic}. For now, ignore the right picture.
\begin{figure}[H]
    \centering
    \includegraphics[width=1\linewidth]{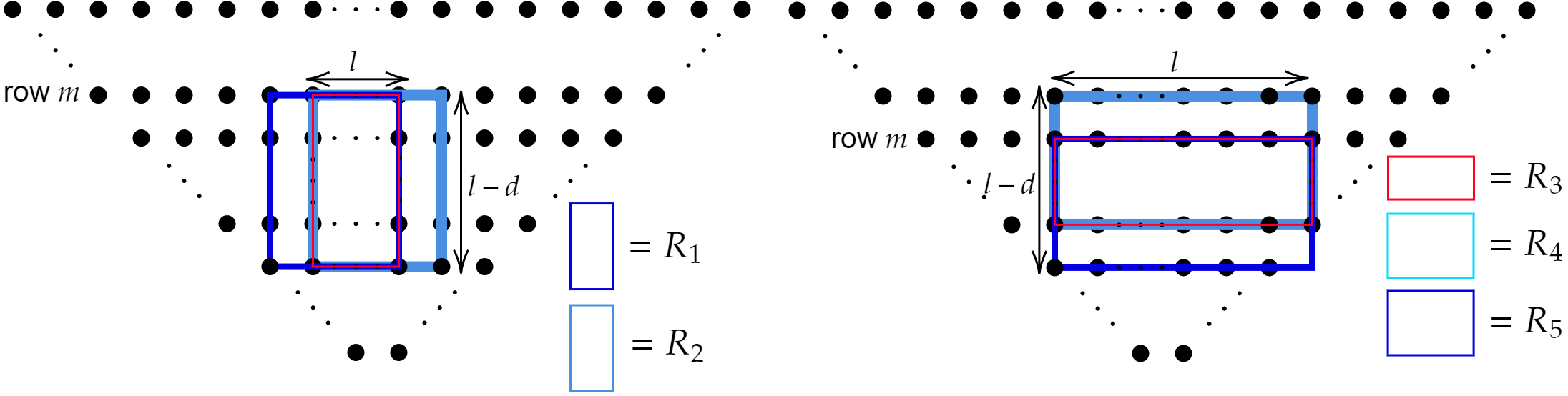}
    \caption{\textbf{Left:} When $d<0$, the sequences in $R_{r,q}(l,d,m,n)$ have a rectangular zero portion (represented by the red rectangle) in their number wall. The sequences in case (1.1) above have a window containing the rectangular portion $R_1$ and not $R_2$ in their number wall. Similarly, if a sequence is in case (1.2), its number wall has a window containing the rectangular portion $R_2$ and not $R_1$. If a sequence is in case (1.3), it has a single window containing both $R_1$ and $R_2$. \textbf{Right:} When $d>0$, the sequences in $R_{r,q}(l,d,m,n)$ generate a number wall such that every entry in the rectangular portion represented by the red rectangle $R_3$ is zero. This means the number walls either have a window containing only this rectangular portion $R_4$, only $R_5$ or both.}
    \label{rect pic}
\end{figure}
\noindent If $\mathbf{S}$ is in case (1.3), then $S\in R_{r,q}({l+2,d+2,m,n-1})$. Also, the number of sequences in case (1.1) is \[\#R_{r,q}({l+1,d+1,m,n-1})-\#R_{r,q}({l+2,d+2,m,n-1}).\] Similarly, the number of sequences in case (1.2) is \[\#R_{r,q}({l+1,d+1,m,n})-\#R_{r,q}({l+2,d+2,m,n-1}).\]
Summing over the three cases gives the formula \begin{align}\#R_{r,q}(l,d,m,n)&=\#R_{r,q}({l+1,d+1,m,n-1})+\#R_{r,q}({l+1,d+1,m,n})\nonumber\\&-\#R_{r,q}({l+2,d+2,m,n-1}).\label{vert_rec_eqn}\end{align}
\noindent Note that, if the top right corner of the specified rectangle is on the right-most diagonal (or the top left corner on the left-most diagonal) of the finite number wall, this method fails. However, in this case the specified rectangle is square (as its height is more than its width and the longest side is assumed to be fully on the number wall by the condition $r>2(n+m+l)$ and Remark \ref{rect_remark}) and consequently the problem reduces to Corollary \ref{contain full}.\\

\noindent The method is fundamentally identical for the proof of the case $0<d\le m$. By the same method as in the $d<0$ case, one then obtains  \begin{align}\#R_{r,q}(l,d,m,n)=\#R_{r,q}({l,d-1,m,n})+\#R_{r,q}({l,d-1,m-1,n})\nonumber\\-\#R_{r,q}({l,d-2,m-1,n}).\label{Rect6}\end{align} This is illustrated by the right-side diagram of Figure \ref{rect pic}.
\subsubsection*{The Base Case of the Induction}
\noindent The base case of the induction is now proven, starting with $d=-1$. Using equation (\ref{vert_rec_eqn}), one obtains that \begin{align}
    \#R_{r,q}({l,-1,m,n})=&\#R_{r,q}({l+1,0,m,n-1})+\#R_{r,q}({l+1,0,m,n})\nonumber\\&-\#R_{r,q}({l+2,1,m,n}).\label{rect_base_case_1}
\end{align}
\noindent Then, applying equation (\ref{Rect6}) to the final term of equation (\ref{rect_base_case_1}) gives \begin{align}
    (\ref{rect_base_case_1})=&\#R_{r,q}({l+1,0,m,n-1})+\#R_{r,q}({l+1,0,m,n})\nonumber\\&-\#R_{r,q}({l+2,0,m,n-1})-\#R_{r,q}({l+2,0,m-1,n-1})\nonumber\\&+\#R_{r,q}({l+2,-1,m-1,n-1}).\label{rect_base_case_2}
\end{align}
\noindent In four of the five terms in Equation (\ref{rect_base_case_2}), one has that $d=0$ and hence Corollary \ref{contain full} implies that $\#R_{r,q}({l,0,m,n})=q^{r-l}$. Therefore, one obtains \begin{equation}\label{rect_base_case_3}
    \#R_{r,q}({l,-1,m,n})= 2q^{r-(l+1)}-2q^{r-(l+2)} + \#R_{r,q}({l+2,-1,m-1,n-1}).
\end{equation}The final term of equation (\ref{rect_base_case_3}) is identical to $\#R_{r,q}({l,-1,m,n})$ but with $l$ increased by 2 and $m$ and $n$ decreased by 1. A simple induction on $m$ yields \begin{equation}
    (\ref{rect_base_case_3})=\sum_{i=0}^{m-1} \left(2q^{r-(l+2i+1)}-2q^{r-(l+2i+2)}\right) +\#R_{r,q}({l+2m,-1,0,n-m}).\label{rect_base_case_4}
\end{equation}
\noindent By the same method, one has that \begin{align*}\#R_{r,q}({l+2m,-1,n-m,0})=&\#R_{r,q}({l+2m+1,0,0,n-m-1})\\&+\#R_{r,q}({l+2m+1,0,0,n-m})\\&-\#R_{r,q}({l+2m+2,0,0,n-m-1})\\
=& 2q^{r-(l+2m+1)}-(q^{r-(l+2m+2)}+0),\end{align*} as $\#R_{r,q}({l,0,-1,n})=0$ for all $n,l\in\N$. Combining this with equation (\ref{rect_base_case_4}) gives\begin{align*}
    (\ref{rect_base_case_4})&=\sum_{i=0}^{m} \left(2q^{r-(l+2i+1)}-2q^{r-(l+2i+2)}\right) + q^{r-(l+2m+2)}\\&=\frac{q^{r-(l+2m+2)}}{q+1}(2q^{2m+2}+q-1),
\end{align*}
\noindent which proves the base case of equation (\ref{vertrec}). The $d=1$ case is fundamentally identical. Indeed, one then obtains that \begin{align*}
    \#R_{r,q}({l,1,m,n})=&2\sum_{i=0}^{m-1}\left(q^{r-(l+2i)}-q^{r-(l+2i+1)}\right) + \#R_{r,q}({l+2m,1,0,n-m})\\
    &=2\sum_{i=0}^{m-1}\left(q^{r-(l+2i)}-q^{r-(l+2i+1)}\right) + q^{r-(l+2m)}\\
    &= \frac{q^{r-(l+2m)}}{q+1}(2q^{2m+1}-(q-1))
\end{align*}
\noindent as required. Note, that to prove equation (\ref{horrec}) by induction, one also needs to establish the case where $d=2$, since equation (\ref{horrec}) does not hold when $d=0$. This is done just as before, by applying the recurrence relations from equation (\ref{Rect6}) along with Corollary \ref{contain full}. This gives \[\#R_{r,q}(l,2,m,n)=\frac{q^{r-(l+2m)}}{q+1}\left(3q^{2m+1}-q^{2m}-(q-1)(q^2+1)\right),\] which agrees with equation (\ref{horrec}).
\subsubsection*{The Induction Step}
\noindent The induction step is completed by using the induction hypothesis to apply equation (\ref{vertrec}) to equation (\ref{vert_rec_eqn}) (or equation (\ref{horrec}) to equation (\ref{Rect6})) and then undergoing a simple calculation. Indeed, when $d<0$, $\#R_{r,q}(l,d,m,n)~=~$\begin{align}&2\cdot \frac{q^{r-(l+2m-d+1)}}{q+1}\left(-d\cdot q^{2m+2}+(d+2)\cdot q^{2m+1}-(d+1)\cdot(q-1)\phantom{q^d}\hspace{-0.4cm}\right)\nonumber\\-&\frac{q^{r-(l+2m-d+1)}}{q+1}\left((-d-1)\cdot q^{2m+2}+(d+3)\cdot q^{2m+1}-(d+2)\cdot(q-1)\right)\nonumber\\=~&\frac{q^{r-(l+2m-d+1)}}{q+1}\left((1-d)\cdot q^{2m+2}+(d+1)\cdot q^{2m+1}-d\cdot(q-1)\right).\end{align} 

\noindent Equation (\ref{horrec}) is proved similarly. 

\subsubsection*{Proving Part 3 of Lemma \ref{Rect}}
\noindent Suppose $d>m$ and let $R_6$ be the rectangle starting in column $n$ and row $m$ with width $l$ and height $l-d$. This implies that the $d-m$ rows underneath $R_6$ are also zero. This is illustrated below:\begin{figure}[H]
    \centering
    \includegraphics[width=0.75\linewidth]{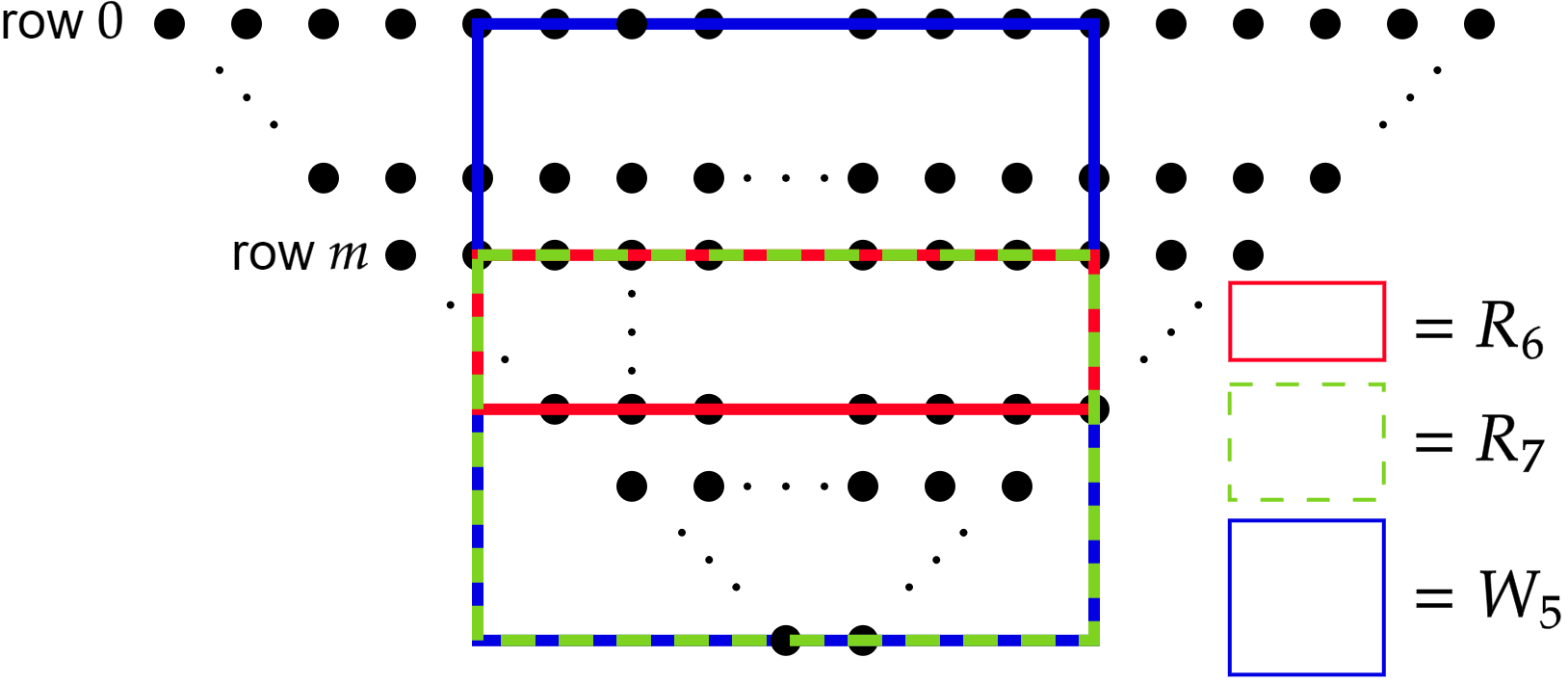}
    \caption{The difference in the height and width of $R_6$ (red rectangle) is greater than its starting depth. The window $W_5$ (blue square) starts on row 0 and contains $R_6$. The green dotted rectangle $R_7$ shows the rows which must be contained in any window containing $R_6$.}
    \label{Rect 4}
\end{figure}
\vspace{-0.4cm}
\noindent The Square Window Theorem (Theorem \ref{window}) implies the rectangular portion $R_6$ must be fully contained in a single square window. Furthermore, the window must have side lengths greater than or equal to $l$.  As a window cannot start higher than row zero, there is no window that contains $R_6$ that does not also contain the rectangular portion $R_7$; namely, the additional $m-d$ rows underneath $R_6$. The difference between the width and height of $R_7$ is $m$, reducing the question to the previous case. This concludes the proof.
\end{proof}\vspace{-0.6cm}

\subsection{Counting Finite Sequences with Windows Containing Pairs of Square Portions}
\noindent This subsection proves Lemma \ref{twowind}. Recall that, given square portions $\Box_1=\Box(l_1,n_1,m_1)$ and $\Box_2=\Box(l_2,n_2,m_2)$, the set $W_{r,q}(\Box_1,\Box_2)$ refers to all the sequences of length $r$ over $\F_q$ whose number wall has a pair of windows that each contain one of $\Box_1$ and $\Box_2$.  The proof of Lemma \ref{twowind} shows that the bound given by equation (\ref{twowindbound}) is sharp in many cases. \\\vspace{-0.4cm}

\subsubsection*{Minimal Generators of Two Square Portions}
\noindent Before the proof begins, the definition of a \textbf{minimal generator} of a square portion is extended to be defined for two square portions in the natural way. That is, given two square portions $\Box_1(l_1,m_1,n_1)$ and $\Box_2(l_2,m_2,n_2)$, a minimal generator of $\Box_1$ and $\Box_2$ is any finite sequence $\mathbf{M}(\Box_1,\Box_2)$ that satisfies following two properties. \begin{itemize}
    \item $W_q(\mathbf{M}(\Box_1,\Box_2))$ has a (complete or incomplete) window or pair of windows containing both $\Box_1$ and $\Box_2$.
    \item If any entry is removed from $\mathbf{M}(\Box_1,\Box_2)$, the previous condition fails.
\end{itemize}  
\noindent Recall that for an infinite sequence $\mathbf{S}=(s_i)_{i\ge0}$ whose number wall has a (complete or incomplete) window containing $\Box_1$, a minimal generator of $\Box_1$ is given by the entries $(s_i)_{n_1-m_1\le i \le n_1+m_1+l_1}$. Then, $\mathbf{M}(\Box_1,\Box_2)$ is a finite sequence $(s_i)$ with indices $i$ in the range\begin{equation}\label{min_gen_two}
    \min\{n_1-m_1,n_2-m_2\}\le i \le \max\{n_1+m_1+l_1, n_2+m_2+l_2\}
\end{equation} 
\begin{proof}[Proof of Lemma \ref{twowind}]

\noindent Let $M$ be the set of all minimal generators for $\Box_1$ and $\Box_2$ and let $$L_M=\max\{n_1+m_1+l_1, n_2+m_2+l_2\}-\min\{n_1-m_1,n_2-m_2\} +1$$ be the length of any finite sequence in $M$. The goal is to show that the cardinality of $M$ is bounded from above by $cq^{L_M-l_1-l_2}$ for some constant $c>0$ that does not depend on $\Box_1$ or $\Box_2$. If this is so, the remaining choices for the sequence of length $r$ are arbitrary and one has $$\#W_{r,q}(\Box_1,\Box_2)\ll_q q^{r-L_M}\cdot q^{L_M-l_1-l_2}.$$ 
\subsubsection*{Division into Cases}
\noindent  Without loss of generality, assume that\begin{equation}\vspace{-0.2cm}
    m_2\ge m_1. \label{low_to_right}
\end{equation}That is, $\Box_2$ is the `lower' of the two square regions. Define the finite sequence $\mathbf{M}(\Box_1)=(m_i^{(1)})_{0\le i \le 2m_1+l+1}$ as a minimal generator of $\Box_1$ and let $\mathbf{M}(\Box_2)=(m_i^{(2)})_{0\le i \le 2m_2+l+1}$ be a minimal generator of $\Box_2$. The proof is split into three cases, depending on the inequality that the natural numbers $n_1,n_2,m_1,m_2,l_1$ and $l_2$ satisfy:\begin{itemize}
    
\item[]{\textbf{Case 1:}} Minimal generators $\mathbf{M}(\Box_1)$ and $\mathbf{M}(\Box_2)$ have no overlap. That is,$$n_1-m_1<n_1+m_1+l_1<n_2-m_2<n_2+m_2+l_2;$$  

\item[]{\textbf{Case 2:}} For some $k\in\N$, the final $k$ entries of $\mathbf{M}(\Box_1)$ overlaps with the first $k$ entries of $\mathbf{M}(\Box_2)$ but neither $\mathbf{M}(\Box_1)$ or $\mathbf{M}(\Box_2)$ is a subsequence of the other. That is, $$n_1-m_1<n_2-m_2\le n_1+m_1+l_1< n_2+m_2+l_2;$$

\item[]{\textbf{Case 3:}} Minimal generator $\mathbf{M}(\Box_1)$ is a subsequence of $\mathbf{M}(\Box_2)$. That is, $$n_2-m_2\le n_1-m_1 < n_1+m_1+l_1<n_2+m_2+l_2.$$
\end{itemize}
\noindent Any other possible cases are either identical to one of the above by Lemma \ref{reflect} (vertical symmetry of finite number walls) or are ruled out by the assumption (\ref{low_to_right}). 

\subsubsection*{Case 1}
\noindent In this case, let $d$ be the number of entries in the generating sequence between any two given minimal generators of $\Box_1$ and $\Box_2$.\vspace{-0.35cm} \begin{figure}[H]
    \centering
    \includegraphics[width=1\linewidth]{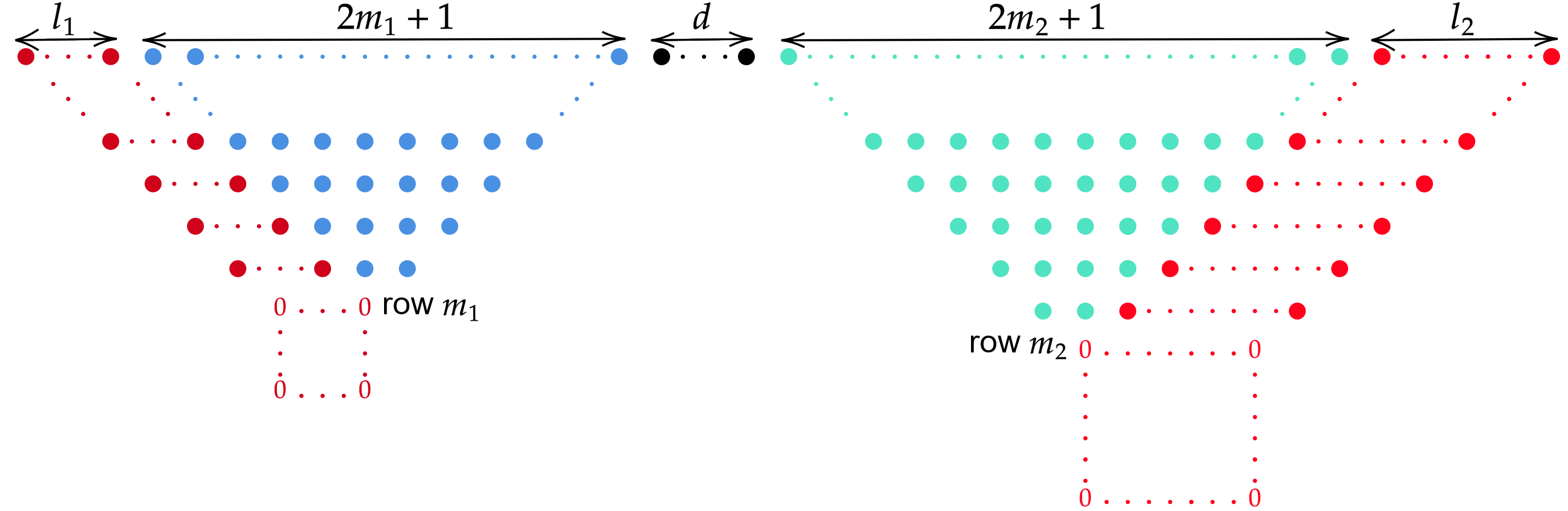}
    \caption{Two minimal generators for $\Box_1$ and $\Box_2$ respectively are separated by $d$ entries.}
\end{figure}\vspace{-0.35cm}\noindent From the original definition of a number wall (Definition \ref{nw}), the Toeplitz matrices making up $\Box_1$ and $\Box_2$ and their respective minimal walls share no entry. Hence, Corollary \ref{contain full} can be applied twice to show there are $q^{2m_1+1}$ choices for a minimal generator of $\Box_1$, and similarly for $\Box_2$. Therefore, each entry of $\mathbf{M}(\Box_1,\Box_2)$ is a finite sequence with length $L_M=2m_1+2m_2+l_1+l_2+d+2$ and $\mathbf{M}(\Box_1,\Box_2)$ has cardinality $q^{L_M-l_1-l_2}$, completing the proof in this case.
\subsubsection*{Case 2}
\noindent Case 2 is illustrated below.
\begin{figure}[H]
    \centering
    \includegraphics[width=0.65\linewidth]{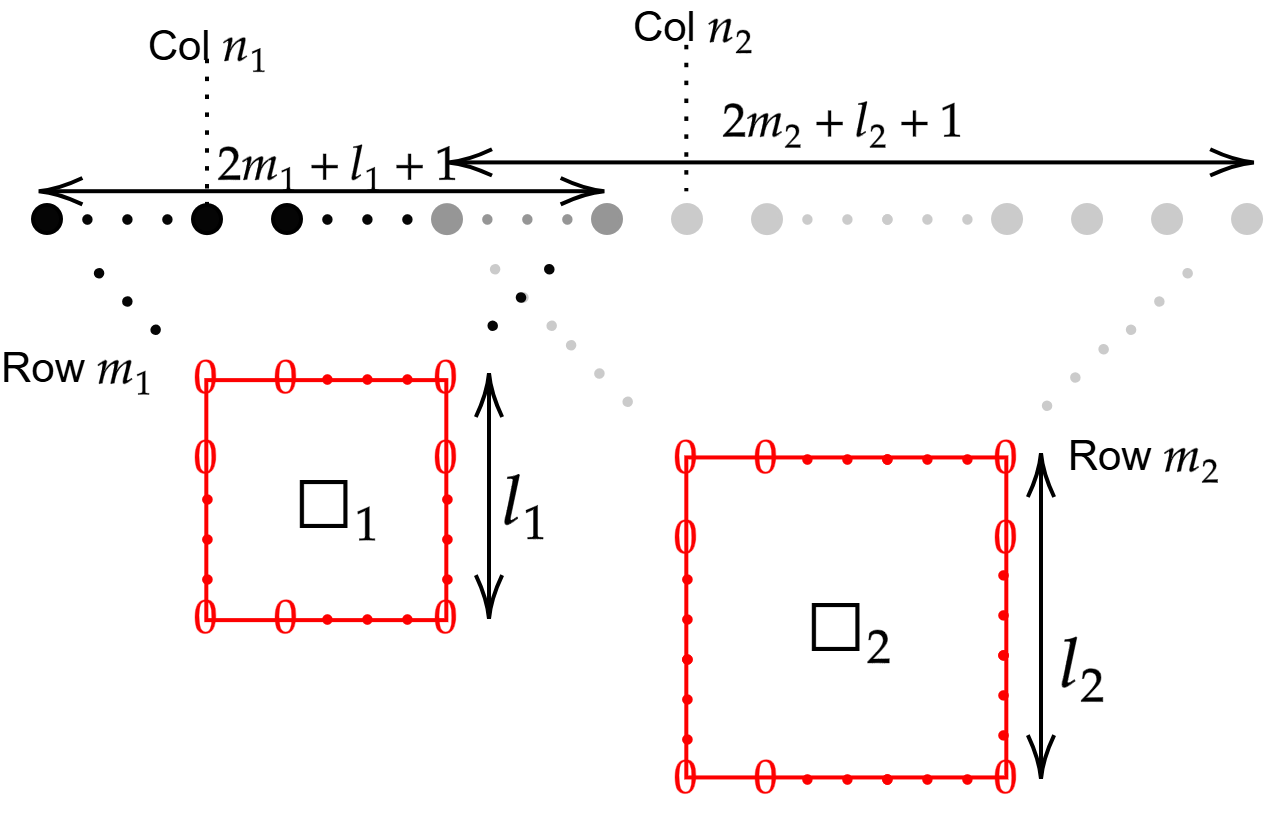}
    \caption{\textbf{Case 2:} $\mathbf{M}(\Box_1)$ (black and dark grey dots, top row) and $\mathbf{M}(\Box_2)$ (light grey and dark grey dots, top row) overlap but one is not contained within the other. The intersection of $\mathbf{M}(\Box_1)$ and $\mathbf{M}(\Box_2)$ is in dark grey.}
\end{figure}

\noindent Recall that $\oplus$ denotes the operator that concatenates finite sequences. To begin, $\mathbf{M}(\Box_1,\Box_2)$ is split into three parts in the following way.
\[\mathbf{M}(\Box_1,\Box_2)=\mathbf{P}(\Box_1)\oplus \mathbf{I}(\Box_1,\Box_2) \oplus \mathbf{P}(\Box_2),\] where $\mathbf{M}(\Box_1)=\mathbf{P}(\Box_1)\oplus \mathbf{I}(\Box_1,\Box_2)$ and $\mathbf{M}(\Box_2)=\mathbf{I}(\Box_1,\Box_2)\oplus \mathbf{P}(\Box_2)$.  \begin{figure}[H]
    \centering
    \includegraphics[width=0.75\linewidth]{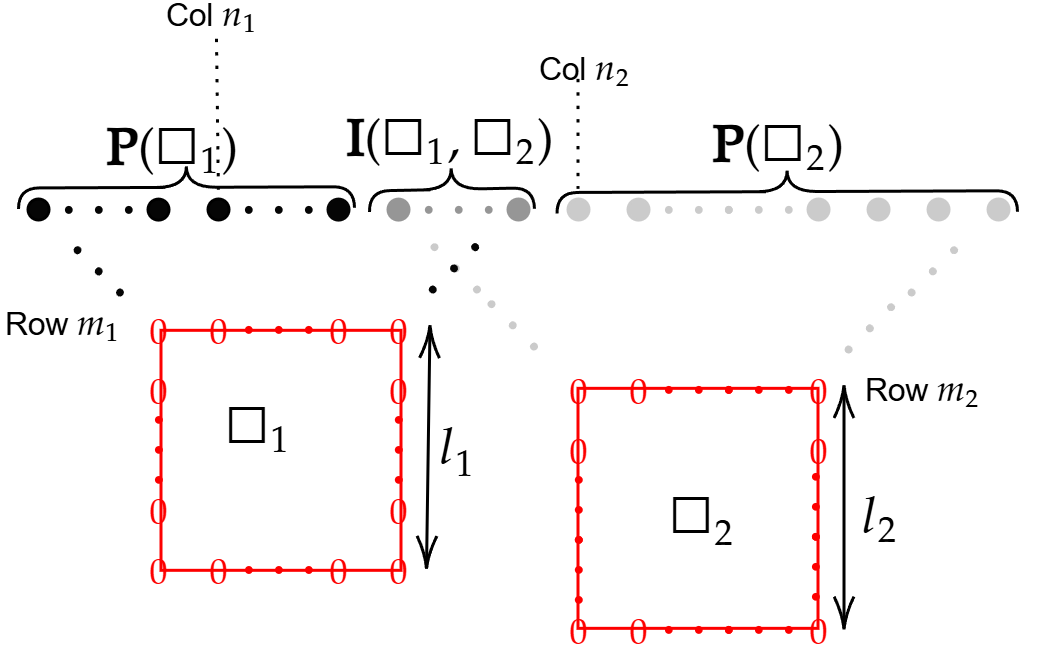}
    \caption{The minimal generator of $\Box_1$ and $\Box_2$ split into three parts.}\label{case2}
\end{figure}\noindent Note that the values of $\#\mathbf{P}(\Box_1)$, $\#\mathbf{I}(\Box_1,\Box_2)$ and $\#\mathbf{P}(\Box_2)$ (in terms of $m_1,m_2,n_1,n_2,l_1$ and $l_2$) are all easily calculable from the positions of $\Box_1$ and $\Box_2$. However, these are not required for this proof, and so they are omitted. The curious reader can calculate them from Figure \ref{case2}, if desired.\\

\noindent By Corollary 3.2.3, there are $q^{\#\mathbf{P}(\Box_1)+\# \mathbf{I}(\Box_1,\Box_2)-l_1}$ choices for $\mathbf{P}(\Box_1)\oplus \mathbf{I}(\Box_1,\Box_2)$ such that $W_q(\mathbf{P}(\Box_1)\oplus \mathbf{I}(\Box_1,\Box_2))$ has a window containing $\Box_1$. One now splits into two cases: \begin{itemize}
    
    \item[]\textbf{Case 2.1:} $W_q(\textbf{I}(\Box_1,\Box_2))$ does not have a window that intersects $\Box_2$.
    \item[]\textbf{Case 2.2:} $W_q(\textbf{I}(\Box_1,\Box_2))$ has a window that intersects $\Box_2$.
\end{itemize}In Case 2.1, Lemma 3.2.2 implies there are $q^{\#\mathbf{P}(\Box_2)-l_2}$ choices for $\mathbf{P}(\Box_2)$ such that $\textbf{P}(\Box_1)\oplus\textbf{I}(\Box_1,\Box_2)\oplus \mathbf{P}(\Box_2)\in W_{r,q}(\Box_1,\Box_2)$. As an over estimate, it is trivial that at most every choice of $\mathbf{P}(\Box_1)\oplus\mathbf{I}(\Box_1,\Box_2)$ falls in Case 2.1, which contributes \[q^{\#\mathbf{P}(\Box_1)+ \#\mathbf{I}(\Box_1,\Box_2)+\# \mathbf{P}(\Box_2)-l_1-l_2}=q^{r-l_1-l_2}\] to the total.\\

\noindent All that remains now is to consider all the choices for in $\textbf{P}(\Box_1)\oplus\mathbf{I}(\Box_1,\Box_2)$ such that $W_q(\mathbf{I}(\Box_1,\Box_2))$ does have a window (call it $\mathcal{W}$) intersecting $\Box_2$. \\

\noindent First, it can be assumed that $\mathcal{W}$ either contains $\Box_2$ or it can be extended to do so. Indeed, if neither of these statements are true then $\mathcal{W}$ is right-side closed and this instance contributes nothing to the total because none of the possible extensions to $\textbf{I}(\Box_1,\Box_2)$ would be in $W_{r,q}(\Box_1,\Box_2)$. In the latter case, the extension of $\textbf{I}(\Box_1,\Box_2)$ that makes $\mathcal{W}$ contain $\Box_2$ is unique by Proposition \ref{wind_extend_lem}.  Additionally, one can also assume that $\mathcal{W}$ does not contain $\Box_1$ (else both $\Box_1$ and $\Box_2$ would be contained in one window, in which case $\textbf{P}(\Box_1)\oplus\mathbf{I}(\Box_1,\Box_2)\oplus \textbf{P}(\Box_2)$ is not in $W_{r,q}(\Box_1,\Box_2)$).\\
 
 \noindent In summary, all that remains is to count how many choices there are for $\textbf{P}(\Box_1)\oplus\textbf{I}(\Box_1,\Box_2)\oplus\textbf{P}(\Box_2)$ such that $W_q(\textbf{I}(\Box_1,\Box_2)\oplus\textbf{P}(\Box_2))$ has a window $\mathcal{W}$ that contains $\Box_2$ and does not intersect $\Box_1$. This is precisely the same as Case 2.1, with the roles of $\Box_2$ and $\Box_1$ reversed, which contributes another $q^{r-l_1-l_2}$ to the total, completing the proof.

\subsubsection*{Case 3}
\noindent Similarly to Case 2, split any given minimal generator for $\Box_1$ and $\Box_2$ up into three disjoint parts\[\mathbf{M}(\Box_1,\Box_2)=\mathbf{P}_L(\Box_2)\oplus \mathbf{I}'(\Box_1) \oplus \mathbf{P}_R'(\Box_2)\] where $\textbf{I}'(\Box_1)$ is a minimal generator for $\Box_1$ and $\mathbf{P}_L'(\Box_2)$ and $\mathbf{P}_R'(\Box_2)$ are extensions to $\mathbf{I}'(\Box_1)$ on the left and right side respectively that generate $\Box_2$.\\

\noindent Case 3 is split into 2 sub-cases, depending on the relative positions of $\Box_1$ and $\Box_2$. To state these cases, recall that $\Box_2$ is on row $m_2\in\N$ and that the depth of a finite number wall generated by a sequence of length $r\in\N$ is equal to $\lceil r/2\rceil-1$. 

\begin{itemize}
    \item[3.(a)] Any finite sequences of length either $\#\mathbf{P}_L'(\Box_2)+ \#\mathbf{I}'(\Box_1)$ or $\#\mathbf{I}'(\Box_1)+\#\mathbf{P}_R'(\Box_2)$ generates a finite number wall of depth less than $m_2$. That is, \[\min\left(\left\lceil\frac{\#\mathbf{P}_L'(\Box_2)+\#\mathbf{I}'(\Box_1)}{2}\right\rceil-1,\left\lceil\frac{\#\mathbf{I}'(\Box_1)+\#\mathbf{P}_R'(\Box_2)}{2}\right\rceil-1\right)<m_2.\]
    \item[3.(b)] Finite sequences of length either $\#\mathbf{P}_L'(\Box_2)+ \#\mathbf{I}'(\Box_1)$ or $\#\mathbf{I}'(\Box_1)+\#\mathbf{P}_R'(\Box_2)$ generate a finite number walls of depth greater than or equal to $m_2$. That is, \[\min\left(\left\lceil\frac{\#\mathbf{P}_L'(\Box_2)+\#\mathbf{I}'(\Box_1)}{2}\right\rceil-1,\left\lceil\frac{\#\mathbf{I}'(\Box_1)+\#\mathbf{P}_R'(\Box_2)}{2}\right\rceil-1\right)\ge m_2.\]
\end{itemize}
\subsubsection*{Case 3.(a)}
\noindent This case is depicted below.

\begin{figure}[H]
    \centering
    \includegraphics[width=0.75\linewidth]{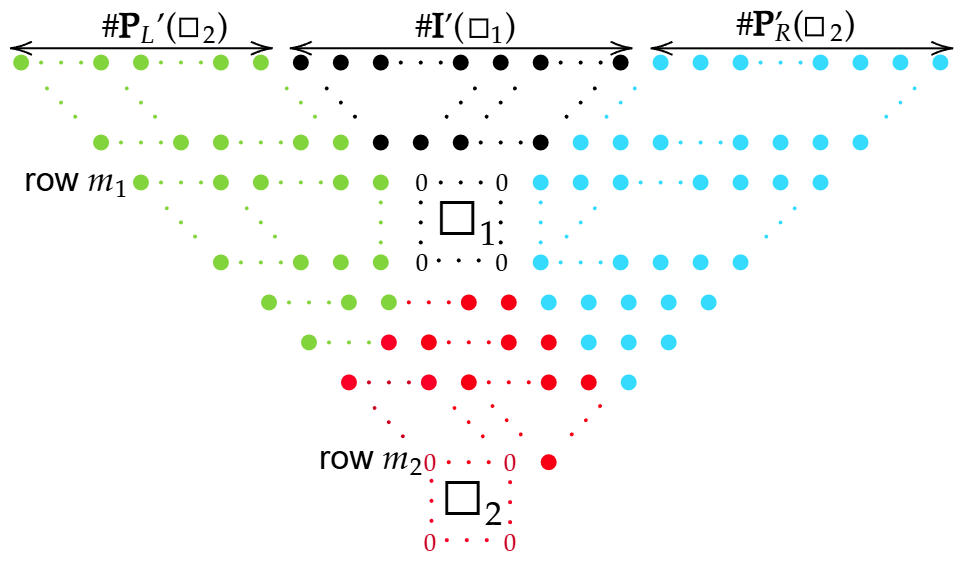}
    \caption{The finite number wall $W_q(I)$ is coloured in black, and the extensions given by $\mathbf{P}_L'(\Box_2)$ and $\mathbf{P}_R'(\Box_2)$ are in green and blue respectively. In this image, neither the sequence of length $\#\mathbf{P}_L'(\Box_2)+\#\mathbf{I}'(\Box_1)$ or $\#\mathbf{I}'(\Box_1)+\#\mathbf{P}_R'(\Box_2)$ generates a number wall with depth greater than or equal to $m_2$. However, Case 3.(a) allows for one of the aforementioned sequence to generate a finite number wall with depth greater than $m_2$.}
\end{figure}

\noindent  Without loss of generality, assume that $\lceil(\#\mathbf{I}'(\Box_1)+\#\mathbf{P}_R'(\Box_2))/2\rceil-1<m_2$. Then, the method is similar to that used in Case 2: indeed, begin by applying Corollary \ref{contain full} to show there are $q^{\#\textbf{I}'(\Box_1)+\#\mathbf{P}_R'(\Box_2)-l_1}$ choices for $\textbf{I}'(\Box_1)\oplus\mathbf{P}_R'(\Box_2)$ such that $W_q(\textbf{I}'(\Box_1)\oplus\mathbf{P}_R'(\Box_2))$ has a window containing $\Box_1$. Then, split into two cases as follows: \begin{itemize}
    \item[] \textbf{Case 3.(a).1:} The finite number wall $W_q(\textbf{I}'(\Box_1)\oplus\mathbf{P}_R'(\Box_2))$ does not have a window $\mathcal{W}$ that intersects $\Box_2$.
    \item[]\textbf{Case 3.(a).2:} The finite number wall $W_q(\textbf{I}'(\Box_1)\oplus\mathbf{P}_R'(\Box_2))$ has a window $\mathcal{W}$ that intersects $\Box_2$.
\end{itemize}
\noindent In Case 3.(a).1, one applies Lemma \ref{contain} to the left side of $\textbf{I}'(\Box_1)\oplus\mathbf{P}_R'(\Box_2)$ to obtain a contribution of $q^{\#\textbf{P}_L'(\Box_2)+\#\textbf{I}'(\Box_1)+\#\mathbf{P}_R'(\Box_2)-l_1-l_2}=q^{r-l_1-l_2}$ to $W_{r,q}(\Box_1,\Box_2)$.\\

\noindent For Case 3.(a).2, it can always be assumed that $\mathcal{W}$ either contains $\Box_2$ or is left-side open, as otherwise it is left-side closed and none of the choices for $\textbf{P}'_L(\Box_2)$ result in a sequence in $W_{r,q}(\Box_1,\Box_2)$. Furthermore, if $\mathcal{W}$ does not contain $\Box_2$ then there is a unique left-side extension to $\textbf{I}'(\Box_1)\oplus\mathbf{P}_R'(\Box_2)$ that extends $\mathcal{W}$ to do so. Finally, it can be assumed that $\mathcal{W}$ does not intersect $\Box_1$, otherwise both $\Box_1$ and $\Box_2$ would be inside the same window and hence this sequence would not be counted in $W_{r,q}(\Box_1,\Box_2)$.\\

\noindent Let $l_\mathcal{W}$ be the side-length of $\mathcal{W}$, and let $l'$ be the amount one must extend $\mathcal{W}$ on the left-side so that it contains $\Box_2$ (if $\mathcal{W}$ contains $\Box_2$, $l'=0$). Note here, that $\mathcal{W}$ (and hence $l_\mathcal{W}$ and $l'$) depends on the choice of $\textbf{I}'(\Box_1)\oplus\mathbf{P}_R'(\Box_2)$. However, for any choice of $\textbf{I}'(\Box_1)\oplus\mathbf{P}_R'(\Box_2)$ there is a corresponding value of $l_\mathcal{W}$ and $l'$, and hence the following method is always applicable. \\

\noindent For any finite sequence $\textbf{I}'(\Box_1)\oplus\mathbf{P}_R'(\Box_2)$ that falls in Case 3.(a).2, there are $q^{\#\textbf{P}_L'(\Box_2)-l'}$ extensions on the left such that $\textbf{P}_L'(\Box_1)\oplus\textbf{I}'(\Box_1)\oplus\mathbf{P}_R'(\Box_2)\in W_{r,q}(\Box_1,\Box_2)$. All that remains is to count how many times Case 3.(a).2 occurs. Indeed, By Corollary \ref{contain full} there are $q^{\#\textbf{I}'(\Box_1)-l_1}$ choices for $\textbf{I}'(\Box_1)$. From the assumption that $\mathcal{W}$ and $\Box_1$ do not intersect, one then applies Lemma \ref{contain} to show that there are $q^{\#\textbf{I}'(\Box_1)+\#\textbf{P}'_R(\Box_2)-l_1-l_\mathcal{W}}$ choices for $\textbf{I}'(\Box_1)\oplus\textbf{P}'_R(\Box_2)$ that fall into Case 3.(a).2. Therefore, this case contributes a total of \begin{equation}
    q^{\#\textbf{P}_L'(\Box_2)+\#\textbf{I}'(\Box_1)+\#\mathbf{P}_R'(\Box_2)-l_1-l_\mathcal{W}-l'} \label{cor_eqn_3}
\end{equation}to $W_{r,q}(\Box_1,\Box_2)$. Using that $\mathcal{W}$ contains $\Box_2$ (once it has been extended by length $l'$), one has that $l_\mathcal{W}+l'\ge l_2$ and hence \[(\ref{cor_eqn_3})\le q^{\#\textbf{P}_L'(\Box_2)+\#\textbf{I}'(\Box_1)+\#\mathbf{P}_R'(\Box_2)-l_1-l_2},\] as required.
\subsubsection*{Case 3.(b)}

\noindent The final sub-case of Case 3 is illustrated below. For now, ignore the distinction between light and dark blue dots and also ignore $\Box_3$ and the green square, $\mathcal{W}$.\begin{figure}[H]
    \centering
    \includegraphics[width=0.75\linewidth]{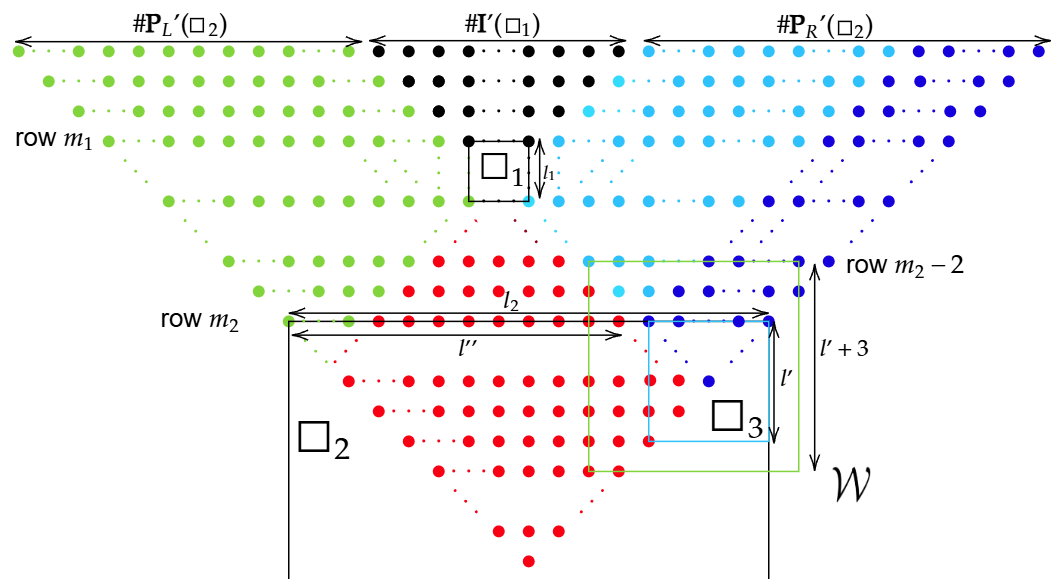}
    \caption{The finite sequence $\mathbf{I}'(\Box_1)$ (top row, black) is a minimal generator for $\Box_1$. The portion of the number wall of $\mathbf{P}_L'(\Box_2)\oplus \mathbf{I}'(\Box_1)\oplus \mathbf{P}_R'(\Box_2)$ that depends only on $\mathbf{P}_L'(\Box_2)\oplus \mathbf{I}'(\Box_1)$ ($\mathbf{I}'(\Box_1)\oplus \mathbf{P}_R'(\Box_2)$, respectively) is coloured in green (blue, respectively). The red portion depends on both $\mathbf{P}_L'(\Box_2)$, $\mathbf{I}'(\Box_1)$ and $\mathbf{P}_R'(\Box_2)$. This depicts Case 3.b, since both the blue and green portions go below row $m_2$.}
    \label{case3b}
\end{figure}

\noindent There is a unique square portion, call it $\Box_3$, that its top left corner on row $m_2$ and has $\mathbf{I}'(\Box_1)\oplus \mathbf{P}_R'(\Box_2)$ as a minimal generator. This is depicted as the blue square in Figure \ref{case3b}. By Corollary \ref{contain full}, there are $q^{\#\textbf{I}'(\Box_1) -l_1}$ choices for $I'(\Box_1)$. By Lemma \ref{contain} (with $\mathbf{I}'(\Box_1)$ in place of $\mathbf{S}_k$), one has $q^{\#\mathbf{P}_R'(\Box_2)-l'}$ such possibilities for $\mathbf{P}_R'(\Box_2)$ that would generate a window containing $\Box_3$. Call this window $\mathcal{W}$. Then, $\mathbf{P}_L'(\Box_2)$ is given by the unique left-side extension of $\mathbf{I}'(\Box_1)$ that extends $\mathcal{W}$. It is possible that $\mathcal{W}$ is left-side closed, in which case there are zero possibilities for $\mathbf{P}_L'(\Box_2)$, but as an overestimate it suffices to assume $\mathcal{W}$ is left-side open.\\

\noindent There is an added layer of complexity to this argument, however. This is illustrated by considering two separate choices for $\mathcal{W}$. First, assume $\mathcal{W}$ coincides exactly with $\Box_3$. Then, by Proposition \ref{wind_extend_lem}, there is a unique extension of length $l''=\#\mathbf{P}_L'(\Box_2)$ to $\mathcal{W}$ so that it contains $\Box_2$. Now let $\mathcal{W}$ be the window depicted by the green square in Figure \ref{case3b}. In this case, $\mathcal{W}$ has side length $l'+3$ and it begins on row $m_2-2$, and hence the bottom row of this window is $m_2+1$. Therefore, it only needs to be extended by $l''-1$ entries on the left side in order to contain $\Box_2$, and so the left-most entry of $\mathbf{P}_L'(\Box_2)$ is free. \\

\noindent In general, if $\mathcal{W}$ has its top row on row $m_2-i$ for some $i\in\N$, then $\mathcal{W}$ must have size at least $l'+i$ in order to contain $\Box_3$. In fact, $\mathcal{W}$ could have side length $l'+i+j$ for any $0\le j \le i$. In this case, one must extend $\mathcal{W}$ on the left by $l''-j$ entries to ensure that it contains $\Box_2$. This leaves the $j$ left-most entries of $\mathbf{P}_L'(\Box_2)$ free. As there are $q^{\#\mathbf{P}_R'(\Box_2)-l'-i-j}$ choices for $\mathbf{P}_R'(\Box_2)$ that contain this window $\mathcal{W}$, and each has $q^j$ choices for $\mathbf{P}_L'(\Box_2)$, one obtains that the number of choices for $\mathbf{P}_L'(\Box_2)$ and $\mathbf{P}_R'(\Box_2)$ combines is bounded from above by \begin{align*}
    \sum_{i=0}^\infty q^{\#\mathbf{P}_R'(\Box_2)-l'-i-j}\sum_{j=0}^i q^j &< q^{\#\mathbf{P}_R'(\Box_2) -l'}\sum_{i=0}^\infty (i+1)q^{ - i}\\&=q^{\#\mathbf{P}_R'(\Box_2) -l'}\sum_{k=0}^\infty \sum_{i=0}^\infty q^{-i-k} \\
 &\le 2q^{\#\mathbf{P}_R'(\Box_2) -l'} \sum_{k=0}^\infty q^{-k} \le 4q^{\#\mathbf{P}_R'(\Box_2) -l'}.
\end{align*}
\noindent Finally, using that $l''=\#\mathbf{P}_L'(\Box_2)$, there are $$q^{\#\mathbf{P}_L'(\Box_2)+\#\mathbf{P}_R'(\Box_2)+\#\mathbf{I}'(\Box_1)-l_1-l'-l''+2}< 4\cdot q^{\#\mathbf{M}(\Box_1,\Box_2)-l_1-l_2}$$ such sequences of length $\#\mathbf{P}'_L(\Box_2)+\#\mathbf{I}'(\Box_1)+\#\mathbf{P}_R'(\Box_2)$ that have windows containing $\Box_1$ and $\Box_2$. 
\end{proof}

\bibliographystyle{alphaurl}
\bibliography{refs}
\newpage
\addcontentsline{toc}{section}{\textbf{Glossary}}
\section*{Glossary}
\noindent The following glossary collects the main notations and definitions introduced in the paper. 
\subsection*{General Notation}
\begin{itemize}
    \item $\lceil\cdot\rceil$ - ceiling function;
    \item $\lfloor\cdot\rfloor$ - floor function;
    \item $\dim$ - Hausdorff dimension;
    \item $\log_q$ - logarithm base $q\in\N$;
    \item $\gg$ \& $\ll$ - Vinogradov notation: For functions $f,g:\R\to\R_{>0}$, $f\ll g$ if there exists some constant $c$ such that for all $x\in\R$, $|f(x)|\le c\cdot |g(x)|$. Similarly for $\gg$. If there is a subscript, then the constant $c$ depends only on the stated variables.
\end{itemize}
\subsection*{Function Field Set Up}
\begin{itemize}
    \item $B(\Theta(t),q^{-l})$ - ball of radius $q^{-l}$ around $\Theta(t)$, defined in equation (\ref{ball})
    \item $\F_q$: the finite field of cardinality $q$;
    \item $\F_q[t]=\{\sum_{i=0}^h a_it^i: a_i\in\F_q, h\in\N\}$ - polynomial ring over $\F_q$;
    \item $\F_q\!\left(\!\left(t^{-1}\right)\!\right)=\{\sum_{i=-h}^\infty a_it^{-i}: a_i\in\F_q, h\in\Z\}$ - field of Laurent series over $\F_q$;
    \item $\langle \sum_{i=-h}^\infty a_it^{-i}\rangle:=\sum_{i=1}^\infty a_it^{-i}$ - fractional part of a Laurent Series;
    \item $\mu$ - the Haar measure over $\F_q$, defined in equation (\ref{Haar});
    \item $W_q(P(t),f)$ - Laurent series satisfying $P(t)$-LC with growth function $f$ - equation (\ref{W_q(P(t),f)}); 
    \item $M_q(P(t),f)=\I\backslash W_q(P(t),f)$;
    \item $|\sum_{i=-h}^\infty a_i t^{-i}|=q^{h}$ where $a_{-h}\neq0$ - norm of a Laurent series;
    \item $|N(t)|_{P(t)}=q^{-\deg(P(t))\cdot \max\{i\ge0:P(t)^i|N(t)\}}$ where $N(t),P(t)\in\F_q[t]$ - $P(t)$-adic norm;
    \item $\I=\left\{\Theta(t)\in{\F_q}\left(\!\left(t^{-1}\right)\!\right): |\Theta(t)|<1\right\}$ - unit interval over function fields;

\end{itemize}
\subsection*{Notation Relating to Number Walls}
\begin{itemize}
    \item The \textbf{blade} of a finite number wall: Definition \ref{blade_def};
    \item The \textbf{blade continuation function} $C$: Definition \ref{Qdef};
    \item \textbf{Closed window}: Definition \ref{closed};
    \item \textbf{Complete window}: Definition \ref{complete};
    \item $\oplus$: \textbf{concatenation} of finite sequences: bottom of page 15;
    \item A window \textbf{containing} a square portion: Definition \ref{contain_def};
    \item $k^\nth$ \textbf{diagonal} of a number wall: Definition \ref{diag};
    \item \textbf{Dot diagrams}: Figure \ref{dotdiagram}, page 13;
    \item \textbf{Finite number wall}: beginning of Section \ref{Sect: 3.3};
    \item \textbf{Free} or \textbf{determined} entry in a number wall: Definition \ref{determined};
    \item \textbf{Hankel} and \textbf{Toeplitz} matrices: Definition \ref{hanktoe};
    \item \textbf{Number Walls}: Definition \ref{nw};
    \item \textbf{Windows}/\textbf{Square Window Theorem}: Theorem \ref{window};
    \item \textbf{Inner frame}: Definition \ref{frame};
    \item \textbf{Open window}: Definition \ref{open_wind}
    \item The \textbf{profile} of a number wall: Definition \ref{profile_def};
    \item \textbf{Minimal generator}: Definition \ref{hat}
    \item $R_{r,q}(l,d,m,n)$ - \textbf{rectangular portion}: Definition \ref{Rect_prtn};
    \item $\Box(l,m,n)$ - \textbf{square portion} with top left corner row $m$ and column $n$: top of page 15;
    \item $\widehat\Box(l,m,k)$ - \textbf{square portion} with top left corner row $m$ and diagonal $k$: middle of page 19;
    \item $l$-\textbf{tapered} sequence: Definition \ref{tapered};
    \item The \textbf{Tree Diagrams}: Figures \ref{XXOTD}, \ref{XXXTD} and \ref{XOXTD};
    \item $W_{r,q}(\Box_1,\Box_2)$ sequences with windows containing two square portions in their number wall: Section \ref{Sect: Pair of Wind};

\end{itemize}
\raisebox{-10ex}
{
    \begin{minipage}{10cm}
        \textbf{Steven Robertson}\\
        School of Mathematics\\
        The University of Manchester\\
        Alan Turing Building\\
        Manchester, M13 9PL\\
        United Kingdom\\
        \texttt{steven.robertson@manchester.ac.uk}  
    \end{minipage}
} 
\end{document}